\documentclass[12pt,a4paper,reqno]{amsart}
\usepackage[all,poly,color]{xy}
\usepackage{amsmath,amssymb, amsthm, amsfonts}
\usepackage{color}
\usepackage[pagebackref,colorlinks]{hyperref}
\usepackage{enumerate}
\usepackage{comment}

\usepackage[dvipsnames]{xcolor}

\usepackage{appendix}

\theoremstyle{plain}
\newtheorem {lemma}{Lemma}
\newtheorem {proposition}[lemma]{Proposition}
\newtheorem {theorem}[lemma]{Theorem}
\newtheorem {corollary}[lemma]{Corollary}

\theoremstyle{definition}
\newtheorem {definition}[lemma]{Definition}

\newtheorem {example}[lemma]{Example}
\newtheorem {question}[lemma]{Question}

\newcommand{\N}{\mathbb{N}}

\newcommand{\End}{\operatorname{End}}
\newcommand{\Ann}{\operatorname{ann}}

\newcommand{\id}{\operatorname{id}}
\newcommand{\tg}{\operatorname{tag}}
\newcommand{\mot}{\operatorname{st}}
\newcommand{\Path}{\operatorname{Path}}
\newcommand{\Ob}{\operatorname{Ob}}

\newcommand{\op}{\operatorname{op}}

\newcommand{\vc}{\operatorname{Vec}}
\newcommand{\Rep}{\operatorname{Rep}}
\newcommand{\Hom}{\operatorname{Hom}}

\DeclareMathOperator{\len}{len}
\DeclareMathOperator{\tagg}{tag}

\DeclareMathOperator{\RG}{\mathbf{RG}}

\def\star{{\rm St}}

\def\op{{\rm op}}

\newcommand{\Modd}{\operatorname{Mod}}
\newcommand{\nb}{\operatorname{nb}}

\usepackage{tikz}
\usepackage{ifthen}
\usetikzlibrary{calc}
\tikzset{
  my label/.style={font=\scriptsize,inner sep=2pt},
  a/.style={my label,above,node contents={$\mathtt{e}$}},
  b/.style={my label,right,node contents={$\mathtt{f}$}},
  c/.style={my label,above,node contents={$\mathtt{g}$}},
}
\newcommand\caley[5]{
  \ifthenelse{0<#1}{
    \pgfmathtruncatemacro\newlev{#1-1}
    \pgfmathtruncatemacro\len{#2}
    \ifthenelse{#1>2}{
    \draw[draw=black,-latex] (0,0) -- +(30:\len pt) node[pos=.6,#3] coordinate (O);
}
 {  
    \draw[draw=black,-latex] (0,0) -- +(30:\len pt)  coordinate (O);
}
    \begin{scope}[shift={(O)}]
     \begin{scope}[rotate=0]  \caley{\newlev}{\len/2}{#3}{#4}{#5} \end{scope}

     \begin{scope}[rotate=60] \caley{\newlev}{\len/2}{#4}{#5}{#3} \end{scope}
       
      \begin{scope}[rotate=-60] \caley{\newlev}{\len/2}{#5}{#5}{#3} \end{scope}

    \end{scope}
  }
  {\fill[red] circle(1pt);}
}

\begin{document}

\title[Irreducible representations of Leavitt algebras]{Irreducible representations of Leavitt algebras}

\author{Roozbeh Hazrat}
\address{Hazrat: 
Centre for Research in Mathematics\\
Western Sydney University\\
Australia} \email{r.hazrat@westernsydney.edu.au}

\author{Raimund Preusser}
\address{Preusser: Chebyshev Laboratory, St. Petersburg State University, Russia} \email{raimund.preusser@gmx.de}

\author{Alexander Shchegolev}
\address{Shchegolev: Department of Mathematics, St. Petersburg State University, Russia} \email{iRyoka@gmail.com}


\maketitle

\begin{abstract}

For a weighted graph $E$, we construct  representation graphs $F$, and consequently, $L_K(E)$-modules $V_F$, where  $L_K(E)$ is the Leavitt path algebra associated to 
$E$, with coefficients in a field $K$. We characterise representation graphs $F$ such that $V_F$ are simple $L_K(E)$-modules. We show that the category of representation graphs of $E$, $\RG(E)$, is a disjoint union of subcategories, each of which contains a unique  universal object $T$ which gives an indecomposable $L_K(E)$-module $V_T$ and a unique irreducible representation graph $S$, which gives a simple $L_K(E)$-module $V_S$.

Specialising to graphs with one vertex and $m$ loops of weight $n$, we construct irreducible representations for the celebrated Leavitt algebras $L_K(n,m)$. 
On the other hand, specialising to graphs with weight one, we recover the simple modules of Leavitt path algebras constructed by Chen via infinite paths or sinks and give a large class of non-simple indecomposable modules.  
\end{abstract}

\tableofcontents

\section{Introduction}

In a series of papers~\cite{vitt56,vitt57,vitt62}, William Leavitt studied algebras that are now denoted by $L_K(n,n+k)$ and have been coined {\it Leavitt algebras}. Let $X=(x_{ij})$ and $Y=(y_{ji})$ be $n\times (n+k)$ and $(n+k)\times n$ matrices consisting of symbols $x_{ij}$ and 
$y_{ji}$, respectively. Then for a field $K$, $L_K(n,n+k)$ is the unital $K$-algebra generated by all $x_{ij}$ and $y_{ji}$ subject to the relations $XY=I_n$ and $YX=I_{n+k}$.  Leavitt  established that these algebras are of type $(n,k)$. Recall that a ring $A$ is of type $(n,k)$ if $n$ and $k$ are the least positive integers such that $A^n\cong A^{n+k}$ as right $A$-modules. He further showed that   
$L_K(1,k+1)$ are (purely infinite) simple rings and $L_K(n,n+k)$, $n\geq 2$ are domains (see also~\cite{cohn66}). 

A {\it Leavitt path algebra}, first introduced in~\cite{AA,AMP},   is a certain quotient of the path algebra of a directed graph.  When the graph consists of one vertex,  $k+1$-loops, and their doubles, its Leavitt path algebra is $L_K(1,k+1)$.  The Leavitt path algebras $L_K(E)$ associated to the graphs $E$ with coefficients in the field $K$ in general and the algebras $L_K(1,k+1)$ in particular turned out to be a very rich and interesting class of algebras whose studies so far have constitute over 150 research papers and counting. A comprehensive treatment of the subject can be found in the book~\cite{AASbook}.

There have been a substantial number of papers devoted to  (irreducible) representations of a Leavitt path algebra $L_K(E)$, i.e., (simple) $L_K(E)$-modules. Starting with a path algebra, the celebrated theorem of Gabriel characterises graphs $E$ whose path algebras $KE$ have a finite number of isomorphism classes of indecomposable representations. For Leavitt path algebras, Ara and Brustenga~\cite{AB1,AB} studied their finitely presented modules, proving that the category of finitely presented modules over a Leavitt path algebra $L_K(E)$ is equivalent to a quotient category of the corresponding category of modules over the path algebra $KE$. A similar statement for graded modules over a Leavitt path algebra was established by Paul Smith~\cite{smith2}, allowing for making a bridge to study finite dimensional algebras with radical square zero.  On the other hand, by introducing the notion of 
branching systems for a graph $E$, Gon\c{c}alves, and Royer~\cite{GR} could construct representations for $L_K(E)$, providing sufficient conditions when these representations are faithful.

In \cite{C}, starting from a one-sided infinite path $p$ in a graph $E$, Chen considered the $K$-vector space $V_{[p]}$ with basis given by the infinite paths  tail equivalent to $p$. Observing that this vector space  is endowed with a natural action of $L_K(E)$, he proved that $V_{[p]}$ is a simple $L_K(E)$-module. A similar construction was also given for paths ending on a sink vertex.

 Numerous work followed, noteworthy the work of Ara-Rangaswamy and Rangaswamy~\cite{R-2,R-3,ARa}  producing new simple modules associated to infinite emitters and characterising those algebras which have countably (finitely) many distinct isomorphism classes of simple  modules.  Abrams, Mantese and Tonolo~\cite{AMT} studied the projective resolutions for these simple modules. 
The recent work of \'Anh and Nam \cite{nam} provides another way to describe the so-called Chen and Rangaswamy simple modules. We note that these simple modules can also be recovered via the setting of Steinberg algebras (see ~\cite{ahls,nguyen,stein}). 

The algebras $L_K(n,n+k)$, for  $n> 1$, $k\geq 0$, can not be obtained via  Leavitt path algebras. For this reason, the  {\it weighted Leavitt path algebras} were introduced in~\cite{H-1} which give $L_K(n,n+k)$ for a weighted graph with one vertex and $n+k$ loops of weight $n$ (Example~\ref{wlpapp}). If the weights of all the edges are $1$ (i.e., the graph is unweighted), then the weighted Leavitt path algebras reduce to the usual Leavitt path algebras (Example~\ref{exex1}). 

 In this note for a weighted graph $E$ we construct a representation graph $F$, and consequently a representation $V_F$ for the weighted Leavitt path algebra $L_K(E)$. We characterise the representation graphs $F$ such that $V_F$ are irreducible representations, i.e, they are simple $L_K(E)$-modules. A  graph $F$ is a representation graph for $E$, if there is a graph homomorphism $F\rightarrow E$ such that, roughly speaking,  each edge of $E$ `uniquely' lifts to $F$ (Definition~\ref{defwp}). Then $V_F$ is generated as a $K$-vector space by vertices of $F$ as the basis  and the action of $L_K(E)$  on each vertex $v$ is uniquely determined by moving $v$ on the graph $F$ along the edges determined by the monomials of $L_K(E)$.

 Specialising to graphs with one vertex and $m$ loops of weight $n$, we can construct irreducible representations for the classical Leavitt algebras $L_K(n,m)$.  As an example, the  algebra $L_K(2,3)$ can be obtained as a Leavitt path algebra of a weighted graph with one vertex and two loops of weight $3$: \[\xymatrix{ v\ar@(dr,ur)_{\mathtt{e_1,e_2,e_3}}\ar@(dl,ul)^{\mathtt{f_1,f_2,f_3}}}.\] 
 The representation graph $F$ below gives rise to a simple $L(2,3)$-module $V_F$, as follows: 
 
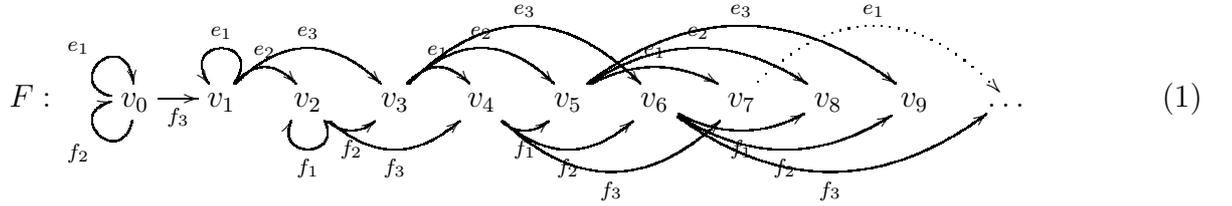
\begin{equation}\label{wlpa3}
  \xymatrix@=15pt{ F: \,\,&v_0 \ar@[blue]@(l,u)^{e_1}  \ar@[red]@(d,l)^{f_2}  \ar[r]_{f_3}  & v_1   \ar@[blue]@(ur,ul)_{e_1}    \ar@[blue]@/^1pc/[r]^{e_2}  \ar@[blue]@/^1.6pc/[rr]^{e_3} & v_2\ar@[red]@(dr,dl)^{f_1}    \ar@[red]@/_1pc/[r]_{f_2}  
 \ar@[red]@/_1.6pc/[rr]_{f_3} & v_3   \ar@[blue]@/^1pc/[r]^{e_1}  \ar@[blue]@/^1.6pc/[rr]^{e_2}  \ar@[blue]@/^2.3pc/[rrr]^{e_3} & 
 v_4  \ar@[red]@/_1pc/[r]_{f_1}  \ar@[red]@/_1.6pc/[rr]_{f_2}  \ar@[red]@/_2.3pc/[rrr]_{f_3} & v_5 \ar@[blue]@/^1pc/[rr]^{e_1}  \ar@[blue]@/^1.6pc/[rrr]^{e_2}   \ar@[blue]@/^2.3pc/[rrrr]^{e_3}
 & v_6    \ar@[red]@/_1pc/[rr]_{f_1}  \ar@[red]@/_1.6pc/[rrr]_{f_2}  \ar@[red]@/_2.3pc/[rrrr]_{f_3} & v_7   \ar@{.>}@/^2.3pc/[rrr]^{e_1} & v_8& v_9 & \dots}
 \end{equation}
Here  $V_F$ is a $K$-vector space with basis $\{ v_i \}_{i \in \mathbb N_0}$ with the action of $\mathtt{e_i,e_i^*}\in L_K(2,3)$ on $v_k$ defined by $v_k \mathtt{e_i} = r(e_i)$ and 
 $v_k \mathtt{e_i^*} = s(e_i)$ and similarly for $\mathtt{f_i}$'s. Here $\mathtt{e_i}$ and $v_k$, uniquely determine an edge $e_i$ in $F$ with $s(e_i)=v_k$ and thus $v_k\mathtt{e_i}$ slides $v_k$ along the edge $e_i$ to $r(e_i)$. Therefore $v_5\mathtt{e_3}=v_9$ and $v_6\mathtt{f_2^*}=v_4$.

On the other hand, specialising to graphs with weight one, we recover simple modules of Leavitt path algebras constructed by Chen via infinite paths. Our approach gives a completely new way to represent simple modules of these algebras. Besides being more visual, this approach allows for carrying calculus on these modules with ease. 

As an example, for the graph 
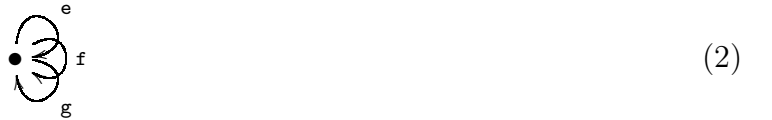
\begin{equation}\label{lustigdog}
\xymatrix{
 \bullet \ar@(u,r)^{\mathtt{e}} \ar@(ur,rd)^{\mathtt{f}} \ar@(r,d)^{\mathtt{g}} &
}
\end{equation}
the infinite paths $p=\mathtt{efg\,efg\dots}$ and $q=\mathtt{efef^2ef^3\dots}$, give rise to Chen simple $L_K(E)$-modules $V_{[p]}$ and $V_{[q]}$. Using our approach, the representation graphs $F$ and $G$ below give rise to simple $L_K(E)$-modules $V_F$ and $V_G$ such that $V_F\cong V_{[p]}$ and $V_G\cong V_{[q]}$. 
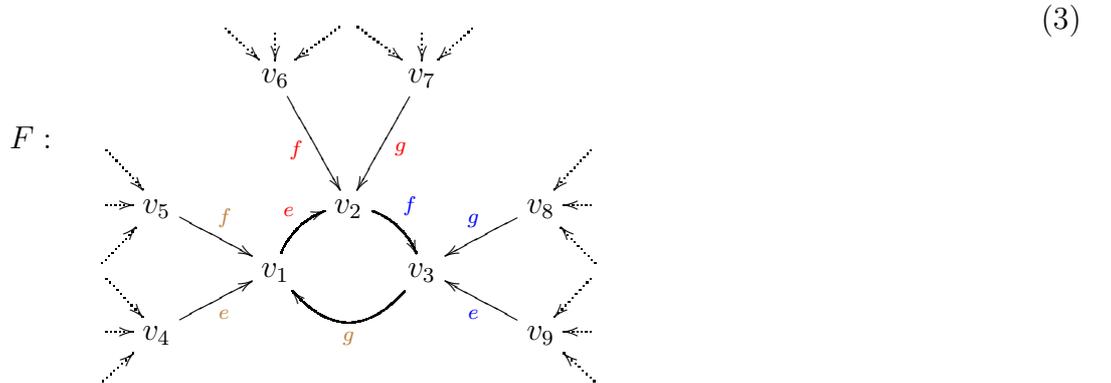
\begin{equation}\label{wlpa4}
\xymatrix@=10pt{
& && & &  & & & & &\\
& &  & & &v_6  \ar@{<.}[u] \ar@{<.}[ul] \ar@{<.}[ur] \ar[ddr]_{\textcolor{red}{f}} & &v_7  \ar@{<.}[u] \ar@{<.}[ul] \ar@{<.}[ur] \ar[ddl]^{\textcolor{red}{g}}&  \\
& F:& & & &  & & & & &\\
& & & v_5  \ar@{<.}[l] \ar@{<.}[ul] \ar@{<.}[dl] \ar[drr]^{\textcolor{brown}{f}} &&&  v_2  \ar@/^/[dr]^{\textcolor{blue}{f}} &&&v_8 \ar@{<.}[r] \ar@{<.}[ur]  \ar@{<.}[dr] \ar[lld]_{\textcolor{blue}{g}}& & &  \\
& & & & & v_1 \ar@/^/[ur]^{\textcolor{red}{e}} &&v_3 \ar@/^1.7pc/[ll]^{\textcolor{brown}{g}}  & & & & & & & &\\
& & & v_4  \ar@{<.}[l] \ar@{<.}[ul] \ar@{<.}[dl] \ar[urr]_{\textcolor{brown}{e}}&&&&&& v_9 \ar@{<.}[r] \ar@{<.}[ur] \ar@{<.}[dr]   \ar[llu]^{\textcolor{blue}{e}}& &  \\
& & & & &  & & & & &\\
}
\end{equation}
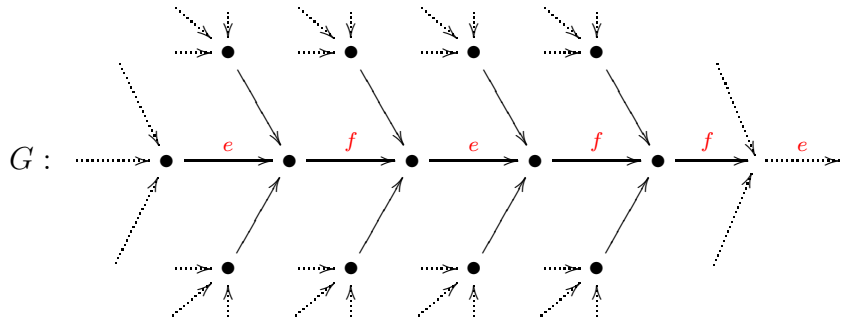
\begin{equation}\label{wlpa45}
\xymatrix@=10pt{
                   &&          \ar@{.>}[dr]            &  \ar@{.>}[d] &      \ar@{.>}[dr]                &   \ar@{.>}[d] &       \ar@{.>}[dr]                &   \ar@{.>}[d] &                      \ar@{.>}[dr]             &  \ar@{.>}[d] &     && \\
                & \ar@{.>}[ddr] &              \ar@{.>}[r]        &          \bullet  \ar[ddr]     &  \ar@{.>}[r] &  \bullet \ar[ddr]    &  \ar@{.>}[r] &     \bullet  \ar[ddr]                    &  \ar@{.>}[r]&      \bullet  \ar[ddr] && \ar@{.>}[ddr]\\
                   &&                      &&                      &&                       &&                               && &&\\
G:\, \, \, \ar@{.>}[rr] && \bullet \ar[rr]^{\textcolor{red}{e}}   &&  \bullet \ar[rr]^{\textcolor{red}{f}} &&   \bullet \ar[rr]^{\textcolor{red}{e}}  && \bullet \ar[rr]^{\textcolor{red}{f}}    && \bullet \ar[rr]^{\textcolor{red}{f}}&&  \ar@{.>}[rr]^{\textcolor{red}{e}} && \\
              &&                      &&                      &&                       &&                               && \\
                   & \ar@{.>}[uur]  &         \ar@{.>}[r]          &          \bullet   \ar[uur]      &  \ar@{.>}[r]    &  \bullet   \ar[uur]    &  \ar@{.>}[r]   &    \bullet     \ar[uur]                     &  \ar@{.>}[r]  &      \bullet  \ar[uur] &&  \ar@{.>}[uur] \\
                    &&            \ar@{.>}[ur]              &  \ar@{.>}[u]&            \ar@{.>}[ur]              &  \ar@{.>}[u]&             \ar@{.>}[ur]              &  \ar@{.>}[u]&                  \ar@{.>}[ur]                & \ar@{.>}[u]&  &&\\
}
\end{equation}
Here $V_F$ is the $K$-vector space with basis $\{v_i \mid i \in \mathbb N\}$ and the action of edges slides the vertices along the graph $F$ such as 
$v_1 \mathtt{efg}=v_1$ and $v_9 \mathtt{egef^*}=v_6$.


We will next study the functor from the category of representation graphs of the graph $E$ (see~\S\ref{catrepa}) to the category of (right) $L_K(E)$-modules, arising from our construction: 
\begin{align*}
V: \RG(E) &\longrightarrow \Modd L_K(E),\\
(F,\phi) & \longmapsto V_F.\notag
\end{align*}

We show that $\RG(E)$ can be written as a disjoint union of certain subcategories, each of which contains a unique universal representation and a unique irreducible representation of $E$, up to isomorphism. We show that the unique universal representation $T$ of each of these subcategories gives an indecomposable $L_K(E)$-module $V_T$, whereas the irreducible representation $S$, gives a simple $L_K(E)$-module $V_S$.

Next in Section~\ref{branchhonda} we describe branching systems for weighted graphs and show how the representation graphs give rise to examples of branching systems. Branching systems for Leavitt path algebras were systematically studied by Gon\c{c}alves and Royer (\cite{GR,GR-0,GR-1}).

As expounded by Green~\cite{green}, one can describe the module category of a certain class of quotient path algebras $A_K(E,r):= KE/\langle r \rangle$, where $KE$ is the path algebra of the finite graph $E$ with coefficients in the field $K$ and $r$ is a set of certain relations,  via  the following equivalence of categories
\begin{equation}\label{catgluts}
 \Modd \, A_K(E,r) \longrightarrow \Rep(E,r).
\end{equation}
Here $\Modd \, A_K(E,r)$ is the category of right $A_K(E,r)$-modules and $\Rep(E,r)$ is the category of representations of the graph $E$ with relations as described in~\cite{green}. The objects of the category $\Rep(E,r)$ consist of placing arbitrary $K$-vector spaces on the vertices of the graph $E$ and assigning linear transformations to the edges that satisfy the relations $r$ (see Appendix~\ref{appendaa}). 
Since (weighted) Leavitt path algebras are among this class, this gives a justification of why branching systems would give representations for (weighted) Leavitt path algebras. However it is not clear how this general machinery of (\ref{catgluts}) can be used to systematically describe irreducible or indecomposable representations of such algebras as the delicate case of acyclic graphs with no relations, which gives the celebrated Gabriel theorem of indecomposability, shows.

 As such we believe that the notion of representation graphs of this paper allows us, for the  first time, to produce irreducible and indecomposable representations for a wide class of algebras, such as Leavitt algebras $L_K(n,m)$.  
 

Throughout the paper $K$ denotes a field and $K^{\times}:=K\setminus\{0\}$. By a $K$-algebra we mean an associative (but not necessarily commutative or unital) $K$-algebra. The semigroup of positive integers is denoted by $\mathbb N$ and the monoid of non-negative integers by $\mathbb N_0$.

\section{Representation graphs of a given graph} 

A \emph{directed graph} $E$ is a tuple $(E^{0}, E^{1}, r, s)$, where $E^{0}$ and $E^{1}$ are
sets and $r,s$ are maps from $E^1$ to $E^0$.  We think of each $e \in E^1$ as an edge  pointing from vertex $s(e)$ to vertex $r(e)$.  In this paper all directed graphs are assumed to be {\it row-finite}, i.e. no vertex emits infinitely many edges. 

A {\it weighted graph}  is a directed graph $E$ equipped with a map  $w:E^1\rightarrow \N$. If $e\in E^1$, then $w(e)$ is called the {\it weight} of $e$. 
We write $(E,w)$ to emphasise that the graph is weighted. Throughout we develop our concepts in the setting of weighted graphs. When the weight map is the constant map $1$, i.e., $w(e)=1$ for any $e\in E^1$, we retrieve the notions in the classical case of directed graphs.  

The main notion introduced and studied in this paper in relation with the theory of Leavitt path algebras, is the notion of a representation graph of a given weighted graph. The concept is closely related to the theory of covering and immersions in graph theory~\cite{stallings,kp}. We start by recalling these notions in the setting of weighted graphs. 

%

\subsection{Covering and immersions} \label{hnhfgftgrgr}
For weighted graphs $E$ and $G$, a \emph{weighted graph homomorphism} $\phi:E\rightarrow  G$, consists of two maps $\phi^0:E^0\rightarrow  G^0$ and $\phi^1:E^1\rightarrow  G^1$ such that for any $e\in E^1$, $s(\phi^1(e))=\phi^0(s(e))$, $r(\phi^1(e))=\phi^0(r(e))$ and $w(\phi^1(e))=w(e)$. Namely, $\phi$ is a homomorphism of graphs which preserve the weight. For a vertex $v\in E^0$, we define
\[w(v):=\max\{w(e)\mid e\in s^{-1}(v)\}.\] If $v$ is a sink we set $w(v)=0$.

\begin{definition}\label{defcovering}
Let $E$ and $T$ be weighted graphs and $\phi=(\phi^0,\phi^1):T\rightarrow E$ a homomorphism of weighted graphs. 

\begin{enumerate}
\item The pair $(T,\phi)$ is called an \emph{immersion} in $E$, if  for any $v\in T^0$, the map  $\phi: s^{-1}(v)\rightarrow s^{-1}(\phi^0(v))$ is injective. 

\medskip

\item The pair $(T,\phi)$ is called a \emph{covering} of $E$, if the following hold:
\smallskip
\begin{enumerate}[(i)]
\item The morphism $\phi$ is onto, i.e, $\phi^0$ and $\phi^1$ are surjective. 
\smallskip
\item For any $v\in T^0$, the map $\phi^1: r^{-1}(v)\rightarrow r^{-1}(\phi^0(v))$ and $\phi: s^{-1}(v)\rightarrow s^{-1}(\phi^0(v))$ are bijective. 
\end{enumerate}

\end{enumerate}
\end{definition}

Putting it another way, a weighted graph immersion or covering is a classical immersion or covering which preserves the weights.

Let $(E,w)$ be a weighted graph. The directed graph $\hat E=(\hat E^0, \hat E^1, \hat s, \hat r)$, where $\hat E^0=E^0$, $\hat E^1:=\{e_1,\dots,e_{w(e)}\mid e\in E^1\}$, $\hat s(e_i)=s(e)$ and $\hat r(e_i)=r(e)$, is called the {\it directed graph associated to $(E,w)$}. If $e_i\in \hat E^1$, then $\tg(e_i):=i$ is called the {\it tag} of $e_i$ and $\mot(e_i):=e$ is called the {\it structure edge} of $e_i$. There is a forgetful homomorphism 
\begin{align}\label{forhoma}
\hat E &\longrightarrow E,\\
u&\longmapsto u\notag\\
 e_i &\longmapsto e\notag
 \end{align} relating these two graphs. 
 
 It is easy to see that if $\phi:T\rightarrow E$ is an immersion or a covering of weighted graphs, so is the graph homomorphism $\hat \phi: \hat T \rightarrow \hat E$ defined by $\hat\phi(u)=\phi(u)$, $u\in \hat T$, and $\hat \phi(e_i):=\phi(e)_i$, $1\leq i \leq w(e)$.

We use the convention that a (finite) path $p$ in a weighted graph $E$ is either a single vertex $p=v\in E^0$ or a sequence $p=e_{1} e_{2}\cdots e_{n}$ of edges $e_{i}$ in $E$ such that
$r(e_{i})=s(e_{i+1})$ for $1\leq i\leq n-1$. We define $s(p) = s(e_{1})$, and $r(p) =
r(e_{n})$.  We denote by $\Path(E)$ the set of all finite paths in $E$. Moreover, if $u,v\in E^0$, then we denote by $_u\!\Path(E)$ the set of all finite paths starting in $u$, by $\Path_v(E)$ the set of all finite paths ending in $v$ and by $_u\!\Path_v(E)$ the intersection of $_u\!\Path(E)$ and $\Path_v(E)$.

Given a weighted graph $E$, we define the \emph{double graph} $E_d$ of $E$ as weighted graph with $E_d^0=E^0$, $E_d^1=\{e, e^* \mid e\in E^1\}$ with $w(e^*)=w(e)$, where $e^*$ has orientation the reverse of that of its counterpart $e\in E^1$  (see \cite[p. 6]{AASbook}). Note that for a weighted graph $E$, one can identify $(\hat E)_d$ and $\widehat{(E_d)}$. 

A path $p=x_1\dots x_n\in \Path(E_d)$ is called \emph{backtracking} if there is a $1\leq j\leq n-1$ such that $x_jx_{j+1}=ee^*$ or $x_jx_{j+1}=e^*e$ for some $e\in  E^1$. We say $p$ is \emph{reduced} if it is not backtracking. We also use the following convention: when we say $p$ is a \emph{reduced path in $E$}, we mean $p$ is not backtracking path of $E_d$. This is in line with literature in graph theory.

Let $E$ be a connected weighted graph. Fix a base point $v\in E^0$. The \emph{universal covering graph of $E$} is a directed weighted graph  $T = T(E, v)$ constructed as follows: let $T^0$ be the set of all reduced path in $E$ starting from $v$,  $T^1 = \{ (a,e) \in  T^0 \times  E^1 \mid r(a) = s(e) \}$ and put 
$s(a,e)=a$, $r(a,e)=ae$. Furthermore for $(a,e) \in T^1$ we set $w(a,e)=w(e)$ and identify $(a,e)^*$ with $(ae,e^*)$. If the weight of the graph is $1$, we obtain the classical case of universal covering. Similar to the classical case, one can show that $T$ is a tree and the isomorphism class of $T = T (E, v)$ is independent of the choice of base point $v$ (see~\cite{kp}). The notion of universal covers allows us to easily construct representation graphs for a given weighted graph (see Lemma~\ref{hfgfghhffeee} and Example~\ref{hfynvhfhr}).  We note that if $T$ is a universal cover of weighted graph $E$, $\hat T$ is not necessarily the universal cover for $\hat E$. 

For a connected (weighted) graph $E$, we denote by $\pi(E)$ its fundamental group (which is independent of the choice of base-point).
 We define a \emph{length map}  
\begin{align}\label{hfbvhdfhddds}
| \, |: \pi(\hat E) &\longrightarrow \mathbb Z^n,\\
 p &\longmapsto |p| \notag
\end{align}
where $n=\max\{w(e) \mid e\in E^1\}$, as follows: For  
 $\{ e_1,\dots e_{w(e)} \mid e \in E^1 \}$ and $\{ e_1^*,\dots e_{w(e)}^* \mid e \in E^1 \}$ and  $v \in E^0$, set $|v|=0$,  $|e_i|=(0,\dots,0,1,0,\dots)$ and $|e_i^*|=(0,\dots,0,-1,0,\dots) \in  \mathbb Z^n$, where $1\leq i\leq w(e)$ and 
$1$ and $-1$ are in the $i$-th component, respectively. One can extend this to a well-defined map (\ref{hfbvhdfhddds}) by counting the length of the path $p$  which is a homomorphism of groups. Note that $|\pi(\hat E)|=0$ if and only if for paths $p,q$ in $\hat E$ with $s(p)=s(q)$ and $r(p)=r(q)$ we have $|p|=|q|$.  This will be used to give a criterion when a representation graph gives rise to a graded representation (Theorem~\ref{gradedrepres}).

\subsection{Representation graphs} We are in a position to define the main notion of this paper, namely a representation graph of a given weighted graph $E$. Roughly, a graph $F$ is a representation graph of the graph $E$ if  ``locally'' $F$ covers all the structure edges arriving at a vertex and all the tags emitting from a vertex in $E$. For the next definition, recall that to a weighted graph $E$ one can associate a directed graph $\hat E$ (see \S\ref{hnhfgftgrgr}). 

\begin{definition}\label{defwp}
Let $(E,w)$ be a weighted graph. A pair $(F,\phi)$, where $F=(F^0,F^1,s_F,r_F)$ is a directed graph and  $\phi=(\phi^0,\phi^1):F\rightarrow \hat E$ is a homomorphism of directed graphs is called a  
{\it representation graph} of $E$, if  the following hold:
\begin{enumerate}
\item For any $v\in F^0$ and $1\leq i\leq w(\phi^0(v))$, there is precisely one $f\in s_F^{-1}(v)$ such that $\tg(\phi^1(f))=~i$;

\smallskip

\item For any $v\in F^0$ and $e\in r^{-1}(\phi^0(v))$, there is precisely one $f\in r_F^{-1}(v)$ such that $\mot(\phi^1(f))=e$.
\end{enumerate}
\end{definition}

In the definition above, using (\ref{forhoma}) we identify vertices of $\hat E$ with those of $E$ when needed.  Throughout the paper, we simply denote a representation graph $(F,\phi)$ of $(E,w)$ by $F$ if there is no cause for confusion in context. 
Some examples of representation graphs are given in Introduction~(\ref{wlpa3}) and (\ref{wlpa4}). 

Let $E$ be a weighted graph, $\hat E$ the directed graph associated to $E$ and $(F,\phi)$ a representation graph for $E$.  Let $\hat E_d$ and $F_d$ be the double graphs of $\hat E$ and $F$, respectively. Clearly the homomorphism $\phi :F\rightarrow \hat E$ induces a map $\Path(F_d)\rightarrow \Path(\hat E_d)$, which we also denote by $\phi$. We call two vertices $u,v\in F^0$ in $F$ {\it connected} if $ _u\!\Path_v(F_d)\neq\emptyset$.
A representation graph $F$ is called {\it connected} if any $u,v\in F^0$ are connected. This is equivalent to say that the graph $F$ is connected as an undirected graph. If $C$ is a connected component of $F$, then $(C,\phi|_C)$ is again a representation graph for $E$.

 
 
\begin{lemma}\label{lemwell} Let $E$ be a weighted graph and $(F,\phi)$ a representation of $E$ with the induced map $\phi: \Path(F_d)\rightarrow \Path(\hat E_d)$. 
Let $q, q'\in \Path(F_d)$ such that $\phi(q)=\phi(q')$. If $s(q)=s(q')$ or $r(q)=r(q')$, then $q=q'$.
\end{lemma} 
\begin{proof}
First suppose that $r(q)=r(q')=v$. If one of the paths $q$ and $q'$ is trivial (i.e., is a vertex), then the other must also be trivial and we have $q=v=q'$ as desired. Assume now that $q$ and $q'$ are not trivial. Then $q=x_n\dots x_1$ and $q'=y_n\dots y_1$, for some $n\geq 1$ and $x_1,\dots,x_n, y_1,\dots, y_n\in F_d^1$. We proceed by induction on $n$.

Case $n=1$: Suppose $\phi(x_1)=\phi(y_1)=e_i$ for some $e\in E^1$ and $1\leq i\leq w(e)$. It follows from Definition \ref{defwp}(2) that $x_1=y_1$ and hence $q=q'$. Suppose now that $\phi(x_1)=\phi(y_1)=e_i^*$ for some $e\in E^1$ and $1\leq i\leq w(e)$. Then it  follows from Definition \ref{defwp}(1) that $x_1=y_1$ and hence $q=q'$.

Case $n\to n+1$:
Suppose that $q=x_{n+1}\dots x_1$ and $q'=y_{n+1}\dots y_1$. By the inductive assumption we have $x_i=y_i$ for any $1\leq i\leq n$. It follows that $s_{F_d}(x_n)=s_{F_d}(y_n)=:u$. Hence $x_{n+1}, y_{n+1}\in \Path_u(F_d)$. Now we can apply the case $n=1$ and obtain $x_{n+1}=y_{n+1}$.

Now suppose that $s(q)=s(q')$. Then $r(q^*)=r((q')^*)$. Since clearly $\phi(q^*)=\phi(q)^*=\phi(q')^*=\phi((q')^*)$, we obtain $q^*=(q')^*$. Hence $q=q'$.
\end{proof}

Since the notion of $\Path(F)$ of a graph $F$ plays a prominent role, we remark that in the language of category Lemma~\ref{lemwell} takes on a familiar form. Recall that  a functor $F : \mathcal{C}\rightarrow \mathcal{D}$ between two small categories $\mathcal{C}$ and $\mathcal{D}$ is called \emph{star injective}, if the map \[F |_{\star(x)} : \star(x) \longrightarrow \star(F(x)),\] is injective, where 
\begin{equation}\label{finalclunt}
\star(x) =\{f : x\rightarrow y \text{~a morphism in~} \mathcal{C} \;|\; y \in \mathcal{C}\},
\end{equation} for every object $x \in \mathcal{C}$. Similarly $F$ is called \emph{co-star injective} if 
$F |_{\star(x)^\op}$ is injective, where $\star(x)^\op =\{f : y\rightarrow x \text{~a morphism in~} \mathcal{C} \;|\; y \in \mathcal{C}\}$. 

 One can consider a graph $F$ as a category with vertices as objects and paths as morphisms.  Then the notion $\Path(F)$ represents the morphisms of the category $F$. A graph homomorphism  $\phi:F \rightarrow G$ gives rise to a functor, called $\phi$ again, $\phi:F \rightarrow G$ between the categories $F$ and $G$. With this interpretation, Lemma~\ref{lemwell} states that the functor $\phi: F_d\rightarrow \hat E_d$ is 
 star and co-star injective.

The next lemma shows that any representation for a covering graph induces a representation for the graph as well. This allows us to construct many representations for a given graph (see Example~\ref{hfynvhfhr}). 

\begin{lemma}\label{hfgfghhffeee}
Let $E$ be a weighted graph and  $T$ a covering of $E$. Then any representation graph $F$ of $T$ is a representation graph for $E$. On the other hand, if $F$ is a representation of $E$ and $G$ is a covering of $F$ then $G$ is also a representation of $E$. 
\end{lemma}
\begin{proof}
Let $\psi: T\rightarrow E$ be a covering map and $\phi:F\rightarrow \hat T$ a representation for $T$. It is easy to see that there is a covering of unweighted graphs $\hat \psi:\hat T \rightarrow \hat E$ which respects the tags, i.e., if $\psi(t)=e$ then $\hat \psi(t_i)=e_i$, for $1\leq i \leq w(t)=w(e)$.  Consider the following diagram, where $\sigma$ is the graph homomorphism $\hat \psi\phi$.  
\[\xymatrix{
F \ar[r]^{\phi} \ar@{.>}[rd]_{\sigma}& \hat T\ar[d]^{\hat \psi}\\ 
& \hat E}
\]
We check that $(F,\sigma)$ is a representation graph for $E$.

Let $v\in F^0$, and $1\leq i\leq w(\sigma^0(v))$. Since $\phi$ is a representation of $T$, by Definition~\ref{defwp}(1), there is a  unique $f\in s_F^{-1}(v)$ such that $\tg(\phi^1(f))=i$. 
Since $\hat \psi$ is a covering, there is a bijection $s^{-1}(\phi(v))\rightarrow s^{-1}(\sigma(v))$, preserving tags.   Thus $f$ is also the unique edge such that  $\tg(\sigma^1(f))=i$. 

To see the second condition of a representation graph for the morphism $\sigma$, let $v\in F^0$ and $e\in r^{-1}(\sigma^0(v))$. Since $\psi$ is a covering, we have a bijection $\psi: r^{-1}(\phi(v))\rightarrow r^{-1}(\sigma(v))$. Thus there is a unique $e'\in r^{-1}(\phi(v))$ such that $\psi(e')=e$. By Definition~\ref{defwp}(2) for the representation $\phi$,  there is precisely one $f\in r_F^{-1}(v)$ such that $\mot(\phi^1(f))=e'$ and consequently  $\mot(\hat \psi \phi^1(f))=e$.

The second part of the lemma is easy and we leave it to the reader. 
\end{proof}


\begin{example}\label{hfynvhfhr}
Let $E$ be a weighted graph with one vertex, three loops of weight three.
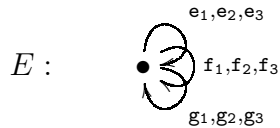
\begin{equation*}
\xymatrix{
E: & \bullet \ar@(u,r)^{\mathtt{e_1,e_2,e_3}} \ar@(ur,rd)^{\mathtt{f_1,f_2,f_3}} \ar@(r,d)^{\mathtt{g_1,g_2,g_3}} &
}
\end{equation*}
The universal cover of the weighted graph $E$ resembles the Cayley graph of the free group $\mathbb F_3$ with three generators (note that in comparison, the universal cover of $\hat E$ is the Cayley graph of the free group $\mathbb F_9$).  Below we show the part of the universal cover $T$ generated by $\mathtt{e},\mathtt{f}$ and $\mathtt{g}$. For each node, choosing $e_1,f_2,g_3$ to emit and the same arrows to arrive at the node, they satisfy Conditions (1) and (2) of Definition~\ref{defwp} of the representation graphs and thus give a representation graph $F$ for $T$ and consequently one for $E$.


\noindent
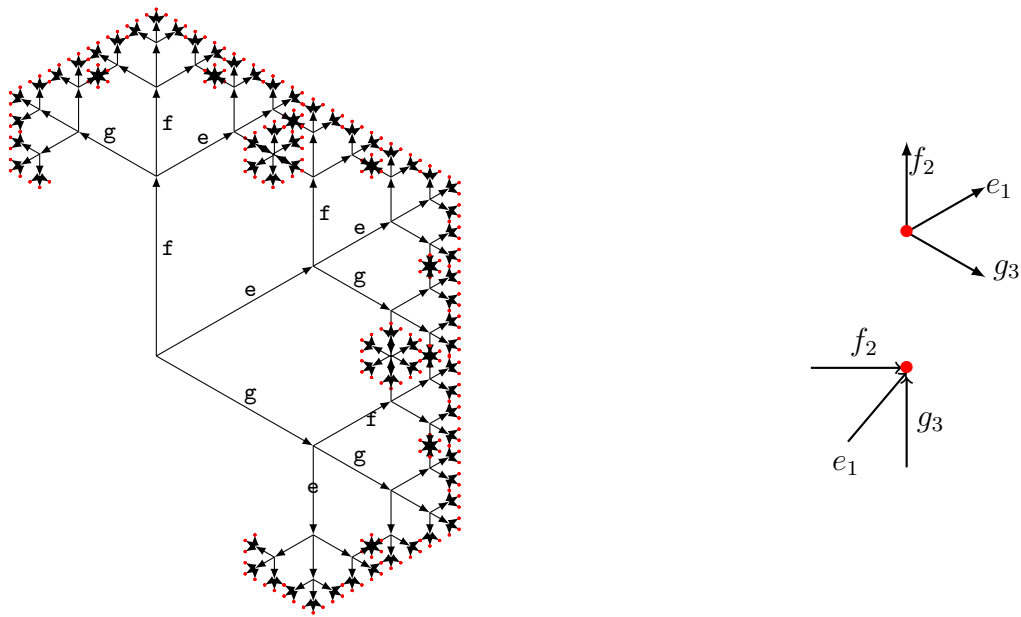
\begin{minipage}{.5\textwidth}
\centering
\begin{tikzpicture}[scale=0.6]
  \begin{scope}[rotate=0]   \caley{5}{4cm}{a}{b}{c} \end{scope}
 \begin{scope}[rotate=60]  \caley{5}{4cm}{b}{c}{a} \end{scope}
\begin{scope}[rotate=-60] \caley{5}{4cm}{c}{b}{a} \end{scope}
 \end{tikzpicture}

\end{minipage}%
\begin{minipage}{.5\textwidth}
\centering
\begin{tikzpicture}[scale=0.6]
  \draw[draw=black,-latex, thick] (0,0) -- node[above,xshift=0.2cm] {$f_2$} +(90:2);
  \draw[draw=black,-latex, thick] (0,0) -- node[above,xshift=0.7cm] {$e_1$} +(30:2);
    \draw[draw=black,-latex, thick] (0,0) -- node[yshift=-0.2cm, xshift=0.8cm] {$g_3$} +(-30:2);
    \node [red] at (0,0) {\textbullet};

    \draw[draw=black,  <-, thick] (-0.1,-3) -- node[above,xshift=0.1cm] {$f_2$} +(-180:2);
  \draw[draw=black, <-, thick ] (0,-3.2) -- node[above, xshift=-0.8cm, yshift=-.85cm] {$e_1$} +(-90:2);
    \draw[draw=black, <-, thick] (0,-3.1) -- node[yshift=-0.2cm, xshift=0.7cm] {$g_3$} +(-130:2);
     \node [red] at (0,-3) {\textbullet}; 

\end{tikzpicture}

\end{minipage}


\end{example}


We will see in Section~\ref{weightedrbgh} that any  representation graph such as $F$ in Example~\ref{hfynvhfhr} will give a module $V_F$ for the weighted Leavitt path algebra $L_K(E)$. However $V_F$ in this example is not a simple module. Next we characterise those representations that the associated module is simple.

\begin{definition}\label{defrepirres}
Let $E$ be a weighted graph and $(F,\phi)$ a representation graph for $E$ with the induced map $\phi: \Path(F_d)\rightarrow \Path(\hat E_d)$. 
We say  $F$ is an {\it irreducible representation} for $E$ if the following  equivalent conditions hold. 

\begin{enumerate}

\item The graph  $F$ is connected and 
\[\phi(\Path_u(F_d))\neq \phi(\Path_v(F_d)), \text{ for any } u\neq v\in F^0.\]

\item The graph  $F$ is connected and 
\[\phi({}_u\!\Path(F_d))\neq \phi({}_v\!\Path(F_d)), \text{ for any } u\neq v\in F^0.\]

\end{enumerate}

\end{definition}

Note that the condition of irreducibility in Definition~\ref{defrepirres}  in the categorical language (see (\ref{finalclunt})) amounts to saying that for $u\not = v\in F^0$, $\phi(\star(u)) \not = \phi(\star(v))$. 

\begin{example}
Consider the representation graph $H$ below of the graph $E$ with one vertex and three loops (\ref{lustigdog}).  One can check that the conditions of Definition~\ref{defrepirres}  are not satisfied and therefore $H$ is not an irreducible representation for $E$. On the other hand the representation graphs  (\ref{wlpa4}) and (\ref{wlpa45}), respectively, are irreducible. 
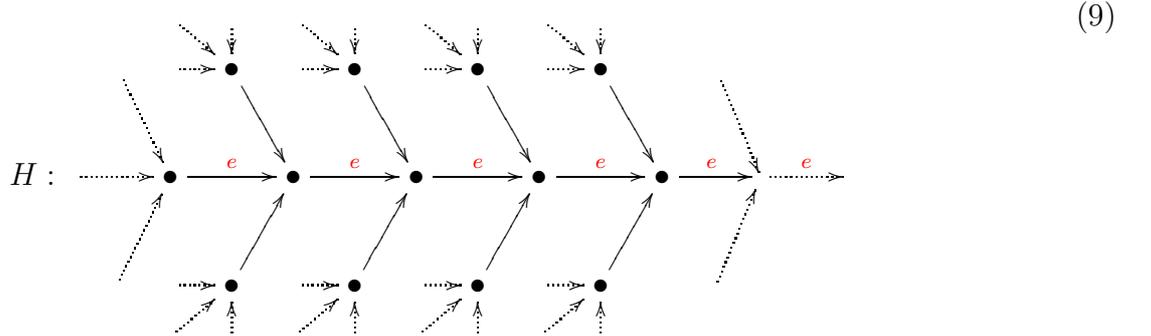
\begin{equation}\label{wlpa456}
\xymatrix@=10pt{
                   &&          \ar@{.>}[dr]            &  \ar@{.>}[d] &      \ar@{.>}[dr]                &   \ar@{.>}[d] &       \ar@{.>}[dr]                &   \ar@{.>}[d] &                      \ar@{.>}[dr]             &  \ar@{.>}[d] &     && \\
                & \ar@{.>}[ddr] &              \ar@{.>}[r]        &          \bullet  \ar[ddr]     &  \ar@{.>}[r] &  \bullet \ar[ddr]    &  \ar@{.>}[r] &     \bullet  \ar[ddr]                    &  \ar@{.>}[r]&      \bullet  \ar[ddr] && \ar@{.>}[ddr]\\
                   &&                      &&                      &&                       &&                               && &&\\
H:\, \, \, \ar@{.>}[rr] && \bullet \ar[rr]^{\textcolor{red}{e}}   &&  \bullet \ar[rr]^{\textcolor{red}{e}} &&   \bullet \ar[rr]^{\textcolor{red}{e}}  && \bullet \ar[rr]^{\textcolor{red}{e}}    && \bullet \ar[rr]^{\textcolor{red}{e}}&&  \ar@{.>}[rr]^{\textcolor{red}{e}} && \\
              &&                      &&                      &&                       &&                               && \\
                   & \ar@{.>}[uur]  &         \ar@{.>}[r]          &          \bullet   \ar[uur]      &  \ar@{.>}[r]    &  \bullet   \ar[uur]    &  \ar@{.>}[r]   &    \bullet     \ar[uur]                     &  \ar@{.>}[r]  &      \bullet  \ar[uur] &&  \ar@{.>}[uur] \\
                    &&            \ar@{.>}[ur]              &  \ar@{.>}[u]&            \ar@{.>}[ur]              &  \ar@{.>}[u]&             \ar@{.>}[ur]              &  \ar@{.>}[u]&                  \ar@{.>}[ur]                & \ar@{.>}[u]&  &&\\
}
\end{equation}

\end{example}

\subsection{The category of the representation graphs of a given graph}\label{catrepa} 

Let $E$ be a weighted graph. Denote by $\RG(E)$  the category of nonempty, connected representation graphs for $E$. A morphism $\alpha:(F,\phi)\to (G,\psi)$ in $\RG(E)$ is a homomorphism $\alpha:F\to G$ of directed graphs such that $\psi\alpha=\phi$. One checks easily that a morphism $\alpha:(F,\phi)\to (G,\psi)$ is an isomorphism if and only if $\alpha^0$ and $\alpha^1$ are bijective.  Lemma~\ref{hfgfghhffeee} states that if $T\rightarrow E$ is a covering of the weighted graph $E$, then we have a natural functor $\RG(T) \rightarrow \RG(E)$. 

In this section we show that the category $\RG(E)$ is a disjoint union of certain subcategories, each of which contains a unique universal representation and a unique irreducible representation of $E$. In \S\ref{weightedrbgh} we will show that the universal representations of $E$ give us indecomposable $L_K(E)$-modules, whereas irreducible representations of $E$ give rise to simple $L_K(E)$-modules. 

We need to establish several lemmas which allows us to define equivalence relations on representations. 

\begin{lemma}\label{lempath}
Let $(F,\phi)$ and $(G,\psi)$ be objects in $\RG(E)$. Let $u\in F^0$ and $v\in G^0$. If $\phi({}_u\!\Path(F_d))\subseteq \psi({}_v\!\Path(G_d))$, then $\phi({}_u\!\Path(F_d))=\psi({}_v\!\Path(G_d))$.
\end{lemma}
\begin{proof}
Suppose that $\phi({}_u\!\Path(F_d))\subseteq \psi({}_v\!\Path(G_d))$. It follows that $\phi^0(u)=\psi^0(v)$ since $u\in {}_u\!\Path(F_d)$. We have to show that $\psi({}_v\!\Path(G_d))\subseteq \phi({}_u\!\Path(F_d))$. Let $p\in{}_v\!\Path(G_d)$. If $p=v$, then $\psi(p)=\psi^0(v)=\phi^0(u)\in \phi({}_u\!\Path(F_d))$. Suppose now that $p$ is nontrivial. Then $p=y_1\dots y_n$ for some $y_1,\dots, y_n\in G_d^1$ where $n\geq 1$. We proceed by induction on $n$.

Case $n=1$: Suppose that $p=g$ for some $g\in G^1$. Then $\psi^1(g)=e_i$ for some $e\in s^{-1}(\psi^0(v))$ and $1\leq i\leq w(e)$. Since $(F,\phi)$ satisfies Condition (1) in Definition \ref{defwp}, there is precisely one $f\in s_F^{-1}(u)$ such that $\tg(\phi^1(f))=i$. Hence $\phi^1(f)=e'_i$ for some $e'\in s^{-1}(\phi^0(u))$. Assume that $e\neq e'$. Since $\phi({}_u\!\Path(F_d))\subseteq \psi({}_v\!\Path(G_d))$, there is a $g'\in s_G^{-1}(v)$ such that $\psi^1(g')=\phi^1(f)=e'_i$. But this is absurd since $(G,\psi)$ satisfies Condition (1) in Definition \ref{defwp}. Thus $\psi(g)=e_i=\phi(f)\in \phi({}_u\!\Path(F_d))$.

Suppose now that $p=g^*$ for some $g\in G^1$. Then $\psi^1(g)=e_i$ for some $e\in r^{-1}(\psi^0(v))$ and $1\leq i\leq w(e)$. Since $(F,\phi)$ satisfies Condition (2) in Definition \ref{defwp}, there is precisely one $f\in r_F^{-1}(u)$ such that $\mot(\phi^1(f))=e$. Hence $\phi^1(f)=e_j$ for some $1\leq j\leq w(e)$. Assume that $i\neq j$. Since $\phi({}_u\!\Path(F_d))\subseteq \psi({}_v\!\Path(G_d))$, there is a $g'\in r_G^{-1}(v)$ such that $\psi^1(g')=\phi^1(f)=e_j$. But this is absurd since $(G,\psi)$ satisfies Condition (2) in Definition \ref{defwp}. Thus $\psi(g^*)=e^*_i=\phi(f^*)\in \phi({}_u\!\Path(F_d))$.

Case $n\rightarrow n+1$:  Suppose $p=y_1\dots y_ny_{n+1}$. By the induction assumption we know that $\psi(y_1\dots y_n)\in \phi({}_u\!\Path(F_d))$. Hence $\psi(y_1\dots y_n)=\phi(x_1\dots x_n)$ for some path $x_1\dots x_n\in {}_u\!\Path(F_d)$. Set $u':=r_{F_d}(x_n)$ and $v':=r_{G_d}(y_n)$. Clearly $\phi^0(u')=\psi^0(v')$ since $\phi$ and $\psi$ are graph homomorphisms.\\
Suppose that $y_{n+1}=g$ for some $g\in G^1$. Then $\psi^1(g)=e_i$ for some $e\in s^{-1}(\psi^0(v'))$ and $1\leq i\leq w(e)$. Since $(F,\phi)$ satisfies Condition (1) in Definition \ref{defwp}, there is precisely one $f\in s_F^{-1}(u')$ such that $\tg(\phi^1(f))=i$. Hence $\phi^1(f)=e'_i$ for some $e'\in s^{-1}(\phi^0(u'))$. Assume that $e\neq e'$. Since $\phi({}_u\!\Path(F_d))\subseteq \psi({}_v\!\Path(G_d))$, there is a $g'\in s_G^{-1}(v')$ such that $\phi(x_1\dots x_nf)=\psi(y_1\dots y_ng')$, which implies $\psi^1(g')=\phi^1(f)=e'_i$. But this is absurd since $(G,\psi)$ satisfies Condition (1) in Definition \ref{defwp}. Thus $\psi(g)=e_i=\phi(f)$ and hence $\psi(y_1\dots y_ng)=\phi(x_1\dots x_nf)\in \phi({}_u\!\Path(F_d))$.

Suppose now that $y_{n+1}=g^*$ for some $g\in G^1$. Then $\psi^1(g)=e_i$ for some $e\in r^{-1}(\psi^0(v'))$ and $1\leq i\leq w(e)$. Since $(F,\phi)$ satisfies Condition (1) in Definition \ref{defwp}, there is precisely one $f\in r_F^{-1}(u')$ such that $\mot(\phi^1(f))=e$. Hence $\phi^1(f)=e_j$ for some $1\leq j\leq w(e)$. Assume that $i\neq j$. Since $\phi({}_u\!\Path(F_d))\subseteq \psi({}_v\!\Path(G_d))$, there is a $g'\in r_G^{-1}(v')$ such that $\phi(x_1\dots x_nf^*)=\psi(y_1\dots y_n(g')^*)$, which implies $\psi^1(g')=\phi^1(f)=e_j$. But this is absurd since $(G,\psi)$ satisfies Condition (2) in Definition \ref{defwp}. Thus $\psi(g^*)=e^*_i=\phi(f^*)$ and hence $\psi(y_1\dots y_ng^*)=\phi(x_1\dots x_nf^*)\in \phi({}_u\!\Path(F_d))$.
\end{proof}

\begin{proposition}\label{propcruc}
Let $\alpha:(F,\phi)\to (G,\psi)$ be a morphism in $\RG(E)$. If $u\in F^0$, then $\phi({}_u\!\Path(F_d))=\psi({}_{\alpha^0(u)}\!\Path(G_d))$.
\end{proposition}
\begin{proof}
Let $u\in F^0$ and $p\in \phi({}_u\!\Path(F_d))$. Then there is a $q\in {}_u\!\Path(F_d)$ such that $\phi(q)=p$. Clearly $\alpha(q)\in {}_{\alpha^0(u)}\!\Path(G_d)$. Hence $p=\phi(q)=\psi(\alpha(q))\in \psi({}_{\alpha^0(u)}\!\Path(G_d))$. We have shown that $\phi({}_u\!\Path(F_d))\subseteq\psi({}_{\alpha^0(u)}\!\Path(G_d))$. It follows from Lemma \ref{lempath} that $\phi({}_u\!\Path(F_d))=\psi({}_{\alpha^0(u)}\!\Path(G_d))$.
\end{proof}

\begin{proposition}\label{propsurj}
Let $\alpha:(F,\phi)\to (G,\psi)$ be a morphism in $\RG(E)$. Then $\alpha$ is a covering map. 
\end{proposition}
\begin{proof}
We first show that $\alpha^0$ is surjective. Choose a $u\in F^0$. If $v\in G^0$, then there is a path $q\in _{\alpha^0(u)}\!\Path_{v}(G_d)$ since $G$ is connected. By Proposition \ref{propcruc} we have $\phi({}_u\!\Path(F_d))=\psi({}_{\alpha^0(u)}\!\Path(G_d))$. Hence there is a $p\in {}_u\!\Path(F_d)$ such that $\phi(p)=\psi(q)$. Since $\phi=\psi\alpha$, we obtain $\psi(\alpha(p))=\psi(q)$. It follows from Lemma \ref{lemwell} that $\alpha(p)=q$. Thus $\alpha^0(r(p))=r(\alpha(p))=r(q)=v$.

We next show that $\alpha^1$ is surjective. Suppose $g\in G^1$ with $v:=s(g)$, $w:=s(\psi(g))$ and $i:=\tagg(\psi(g))$. Since $\alpha^0$ is surjective, there is a $u\in F^0$ such that $\alpha^0(u)=v$. Since $\phi(u)=\psi(\alpha(u))=w$,  by the definition of a representation graph, there is a unique edge $f \in F^1$ with $s(f)=u$ such that $\tagg(\phi(f))=i$. Since $s(\alpha(f))=s(g)$ and $\tagg(\psi(g))=\tagg(\phi(f))=\tagg(\psi(\alpha(f)))=i$, it follows that $\alpha(f)=g$. This shows that $\alpha^1$ is surjective. 

Next we show that that for any $u\in F^0$, the map $\alpha^1: s^{-1}(u)\rightarrow s^{-1}(\alpha^0(u))$ is injective. 
Let $f_1, f_2 \in F^1$ are distinct,  with $s(f_1)=s(f_2)=u$. By the definition of the representation, $\tagg(\phi(f_1))\not = \tagg(\phi(f_2))$ and thus  $\tagg(\psi(\alpha(f_1)))\not = \tagg(\psi(\alpha(f_1)))$. Hence $\alpha(f_1) \not = \alpha(f_2)$. 

A similar argument shows that for any $u\in F^0$, the map $\alpha^1: r^{-1}(u)\rightarrow r^{-1}(\alpha^0(u))$ is bijective. 
\end{proof}

\subsection{Quotients of representation graphs}
For any object $(F,\phi)$ in $\RG(E)$ we define an equivalence relation $\sim$ on $F^0$ by $u\sim v$ if $\phi({}_u\!\Path(F_d))=\phi({}_v\!\Path(F_d))$. Recall that if $\sim$ and $\approx$ are equivalence relations on a set $X$, then one writes $\approx~\leq~ \sim $ and calls $\approx$ {\it finer} than $\sim$ (and $\sim$ {\it coarser} than $\approx$) if $x\approx y$ implies that $x\sim y$, for any $x,y\in X$. 

\begin{definition}\label{defadm}
Let $(F,\phi) \in \RG(E)$ be a representation graph of $E$. An equivalence relation $\approx$ on $F^0$ is called {\it admissible} if the following hold:
\begin{enumerate}
\item $\approx~\leq~\sim$.
\item If $u\approx v$, $p\in{} _u\!\Path_x(F_d)$, $q\in{} _v\!\Path_y(F_d)$ and $\phi(p)=\phi(q)$, then $x\approx y$.
\end{enumerate}
\end{definition}

The admissible equivalence relations on $F^0$ (with partial order $\leq$) form a bounded lattice 
whose maximal element is $\sim$ and whose minimal element is the equality relation $=$.

Let $(F,\phi) \in \RG(E)$ be a representation graph of $E$ and $\approx$ an admissible equivalence relation on $F^0$. If $f, g\in F^1$ we write $f\approx g$ if $s(f)\approx s(g)$ and $\phi(f)=\phi(g)$. This defines an equivalence relation on $F^1$. Define a representation graph $(F_\approx,\phi_\approx)$ for $E$ by 
\begin{align*}
F_\approx^0&=F^0/\approx,\\
F_\approx^1&=F^1/\approx,\\
s([f])&=[s(f)],\\
r([f])&=[r(f)],\\
\phi_\approx^0([v])&=\phi^0(v),\\
\phi_\approx^1([f])&=\phi^1(f).
\end{align*}
We call $(F_\approx,\phi_\approx)$ a {\it quotient} of $(F,\phi)$. It is easy to check that $(F_\approx,\phi_\approx)$ is a representation graph of $E$. 

\begin{theorem}\label{thmadm}
Let $(F,\phi)$ and $(G,\psi)$ be objects in $\RG(E)$. Then there is a morphism $\alpha:(F,\phi)\to (G,\psi)$ if and only if $(G,\psi)$ is isomorphic to a quotient of $(F,\phi)$.
\end{theorem}
\begin{proof}
$(\Rightarrow)$ Suppose there is a morphism $\alpha:(F,\phi)\to (G,\psi)$. If $u,v\in F^0$, we write $u\approx v$ if $\alpha^0(u)=\alpha^0(v)$. Clearly $\approx$ defines an equivalence relation on $F^0$. Below we check that $\approx$ is admissible. 
\begin{enumerate}[(i)]
\item Suppose $u\approx v$. Then $\phi({}_u\!\Path(F_d))=\psi({}_{\alpha^0(u)}\!\Path(G_d))=\psi({}_{\alpha^0(v)}\!\Path(G_d))=\phi({}_v\!\Path(F_d))$ by Proposition \ref{propcruc}. Hence $u\sim v$.
\item Suppose $u\approx v$, $p\in{} _u\!\Path_x(F_d)$, $q\in{} _v\!\Path_y(F_d)$ and $\phi(p)=\phi(q)$. Clearly $\alpha(p)\in{} _{\alpha^0(u)}\!\Path_{\alpha^0(x)}$ $(G_d)$ and $\alpha(q)\in{} _{\alpha^0(v)}\!\Path_{\alpha^0(y)}(G_d)$. Moreover, $\psi(\alpha(p))=\phi(p)=\phi(q)=\psi(\alpha(q))$. Since $\alpha^0(u)=\alpha^0(v)$, it follows from Lemma \ref{lemwell} that $\alpha(p)=\alpha(q)$. Hence $\alpha^0(x)=r(\alpha(p))=r(\alpha(q))=\alpha^0(y)$ and therefore $x\approx y$.
\end{enumerate}

Note that by Lemma \ref{lemwell} we have $f\approx g$ if and only if $\alpha^1(f)=\alpha^1(g)$, for any $f,g\in F^1$. Define a graph homomorphism $\beta:F_\approx\to G$ by $\beta^0([v])=\alpha^0(v)$ and $\beta^1([f])=\alpha^1(f)$. Clearly $\psi\beta=\phi_{\approx}$ and therefore $\beta:(F_\approx,\phi_\approx)\to (G,\psi)$ is a morphism. In view of Proposition \ref{propsurj}, $\beta^0$ and $\beta^1$ are bijective and hence $\beta$ is an isomorphism.

$(\Leftarrow)$ Suppose now that $(G,\psi)\cong (F_\approx,\phi_\approx)$ for some admissible equivalence relation $\approx$ on $F^0$. In order to show that there is a morphism $\alpha:(F,\phi)\to (G,\psi)$ it suffices to show that there is a morphism $\beta:(F,\phi)\to (F_\approx,\phi_\approx)$. But this is obvious (define $\beta^0(v)=[v]$ and $\beta^1(f)=[f]$).
\end{proof}

\subsection{The subcategories $\RG(E)_C$} \label{subsecnn}

Let $(F,\phi)$ and $(G,\psi)$ be objects in $\RG(E)$. We write $(F,\phi)\leftrightharpoons(G,\psi)$ if there is a $u\in F^0$ and a $v\in G^0$ such that $\phi({}_u\!\Path(F_d))=\psi({}_v\!\Path(G_d))$. One checks easily that $\leftrightharpoons$ defines an equivalence relation on $\Ob(\RG(E))$. If $C$ is an $\leftrightharpoons$-equivalence class, then we denote by $\RG(E)_C$ the full subcategory of $\RG(E)$ such that $\Ob(\RG(E)_C)=C$. If $\alpha:(F,\phi)\to (G,\psi)$ is a morphism in $\RG(E)$, then $(F,\phi)\leftrightharpoons(G,\psi)$ by Proposition \ref{propcruc}. Thus $\RG(E)$ is the disjoint union of the subcategories $\RG(E)_C$, where $C$ ranges over all $\leftrightharpoons$-equivalence classes.

Fix an $\leftrightharpoons$-equivalence class $C$, a representation graph $(F,\phi)\in C$ and a vertex $u\in F^0$. 
We denote by $\phi({}_u\!\Path(F_d))_{\nb}$ the set of all paths in $\phi({}_u\!\Path(F_d))$ that are reduced (see~\S\ref{hnhfgftgrgr}). Define a representation graph $(T,\xi)=(T_C,\xi_C)$ for $E$ by
\begin{align*}
T^0&=\{v_p\mid p\in \phi({}_u\!\Path(F_d))_{\nb}\},\\
T^1&=\{e_p\mid p\in \phi({}_u\!\Path(F_d))_{\nb},p\neq \phi(u)\},\\
s(e_{x_1\dots x_n})&=\begin{cases}v_{x_1\dots x_{n-1}},&\text{ if }x_n\in \hat E^1,\\v_{x_1\dots x_n},&\text{ if }x_n\in (\hat E^1)^*,\end{cases}\\
r(e_{x_1\dots x_n})&=\begin{cases}v_{x_1\dots x_n},&\text{ if }x_n\in \hat E^1,\\v_{x_1\dots x_{n-1}},&\text{ if }x_n\in (\hat E^1)^*,\end{cases}\\
\xi^0(v_{x_1\dots x_n})&=\begin{cases}\phi(u),&\text{ if }n=1 \text{ and }x_1=\phi(u),\\r_{\hat E}(x_n),&\text{ if }x_n\in \hat E^1,\\s_{\hat E}(x_n^*),&\text{ if }x_n\in (\hat E^1)^*,\end{cases}\\
\xi^1(e_{x_1\dots x_n})&=\begin{cases}x_n,&\text{ if }x_n\in \hat E^1,\\x_n^*,&\text{ if }x_n\in (\hat E^1)^*.\end{cases}
\end{align*}
Here we use the convention that if $x_1\dots x_n \in \phi({}_u\!\Path(F_d))_{\nb}$, where $n=1$, then $x_1\dots x_{n-1}=\phi(u)$. Clearly $T$ is nonempty and connected. 

\begin{example}\label{exuniv}
Suppose $E$ is a weighted graph with one vertex and two loops of weight $2$: 
\[E:\xymatrix{
 v \ar@(dl,ul)^{\mathtt{e_1,e_2}} \ar@(dr,ur)_{\mathtt{f_1,f_2}} &
}\]
and $(F,\phi)$ is a representation graph for $E$, where $F$ is given by
\[F:\xymatrix{
 u \ar@(dl,ul)^{g} \ar@(dr,ur)_{h} &
}\]
and $\phi:F\to \hat E$ is defined by $\phi^0(u)=v$, $\phi^1(g)=e_1$, $\phi^1(h)=f_2$. Then ${}_u\!\Path(F_d)=\Path(F_d)$ and $\phi({}_u\!\Path(F_d))_{\nb}$ consists of all nonbacktracking paths in $\Path(\hat E_d)$ whose letters come from the alphabet $\{v,e_1,e_1^*,f_2,f_2^*\}$. Let $C$ be the $\leftrightharpoons$-equivalence class of $(F,\phi)$. Then $T_C$ is the graph
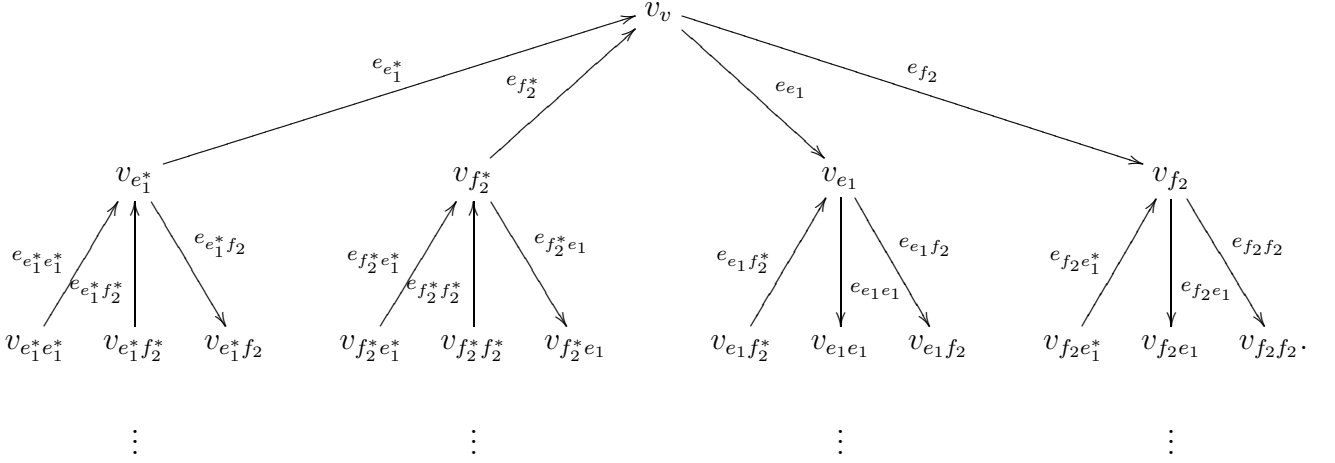
\[\xymatrix@C=7pt@R=20pt{
&&&&&&&v_v\ar[ddrr]^(.6){e_{e_1}}\ar[ddrrrrrr]^{e_{f_2}}&&&&&&&\\
&&&&&&&&&&&&&&\\
&v_{e_1^*}\ar[uurrrrrr]^{e_{e_1^*}}\ar[ddr]^(.5){e_{e_1^*f_2}}&&&&v_{f_2^*}\ar[uurr]^(.4){e_{f_2^*}}\ar[ddr]^(.5){e_{f_2^*e_1}}&&&&v_{e_1}\ar[dd]^(.7){e_{e_1e_1}}\ar[ddr]^(.5){e_{e_1f_2}}&&&&v_{f_2}\ar[dd]^(.7){e_{f_2e_1}}\ar[ddr]^(.5){e_{f_2f_2}}&\\
&&&&&&&&&&&&&&\\
v_{e_1^*e_1^*}\ar[uur]^(.4){e_{e_1^*e_1^*}}&v_{e_1^*f_2^*}\ar[uu]^(.3){e_{e_1^*f_2^*}}&v_{e_1^*f_2}&&v_{f_2^*e_1^*}\ar[uur]^(.4){e_{f_2^*e_1^*}}&v_{f_2^*f_2^*}\ar[uu]^(.3){e_{f_2^*f_2^*}}&v_{f_2^*e_1}&&v_{e_1f_2^*}\ar[uur]^(.4){e_{e_1f_2^*}}&v_{e_1e_1}&v_{e_1f_2}&&v_{f_2e_1^*}\ar[uur]^(.4){e_{f_2e_1^*}}&v_{f_2e_1}&v_{f_2f_2}.\\
&\vdots&&&&\vdots&&&&\vdots&&&&\vdots&
}\]
\end{example}

\begin{proposition}\label{propuniv}
If $(G,\psi)\in C$, then there is a morphism $\alpha:(T_C,\xi_C)\to (G,\psi)$. 
\end{proposition}
\begin{proof}
Since $(G,\psi)\leftrightharpoons(F,\phi)$, there is a $v\in G^0$ such that $\phi({}_u\!\Path(F_d))=\psi({}_v\!\Path(G_d))$.
Define a homomorphism $\alpha:T_C\to G$ as follows. Let $x_1\dots x_n\in \phi({}_u\!\Path(F_d))_{\nb}$. Since $\phi({}_u\!\Path(F_d))=\psi({}_v\!\Path(G_d))$, there is a (uniquely determined) path $y_1\dots y_n\in {}_v\!\Path(G_d)$ such that $\psi(y_1\dots y_n)=x_1\dots x_n$. Define $\alpha^0(v_{x_1\dots x_n})=r(y_n)$, $\alpha^1(e_{x_1\dots x_n})=y_n$ if $y_n\in G^1$ and respectively $\alpha^1(e_{x_1\dots x_n})=y_n^*$ if $y_n\in (G^1)^*$. We leave it to the reader to check that $\alpha$ is a graph homomorphism and that $\psi\alpha=\xi_C$.
\end{proof}

We are in a position to show that in each equivalence class $C$ of the representation graphs of the weighted graph $E$, there is only one irreducible representation graph up to isomorphism. 

\begin{corollary}\label{corcat1}
Up to isomorphism the representation graphs in $C$ are precisely the quotients of $(T_C,\xi_C)$, and consequently  \[(S_C,\zeta_C):=((T_C)_\sim,(\xi_C)_\sim)\]  is the unique irreducible representation graph in $C$.
\end{corollary}
\begin{proof}
The first statement follows from Theorem \ref{thmadm} and Proposition \ref{propuniv}. The second statement now follows since a quotient $((T_C)_\approx,(\xi_C)_\approx)$ satisfies the equivalent conditions in Definition \ref{defrepirres} if and only if $\approx$ equals $\sim$. 
\end{proof}

Recall that an object $X$ in a category {\bf $C$} is called {\it repelling} (resp. {\it attracting}) if for any object $Y$ in {\bf $C$} there is a morphism $X\rightarrow Y$ (resp. $Y\rightarrow X$). By Proposition \ref{propuniv} $(T_C,\xi_C)$ is a repelling object in $C$. On the other hand, if $(G,\psi)$ is an object in $C$, then clearly $(S_C,\zeta_C)$ is isomorphic to a quotient of $(G,\psi)$. It follows from Theorem \ref{thmadm} that $(S_C,\zeta_C)$ is an attracting object in $C$.

\begin{example}\label{excatone}
Suppose $E$ is a weighted graph with one vertex and two loops of weight $2$: 
\[E:\xymatrix{
 v \ar@(dl,ul
 )^{\mathtt{e_1,e_2}} \ar@(dr,ur)_{\mathtt{f_1,f_2}} &
}.\]

Consider the representation graphs $(F_1,\phi_1),\dots, (F_7,\phi_7)$ for $E$ given below.

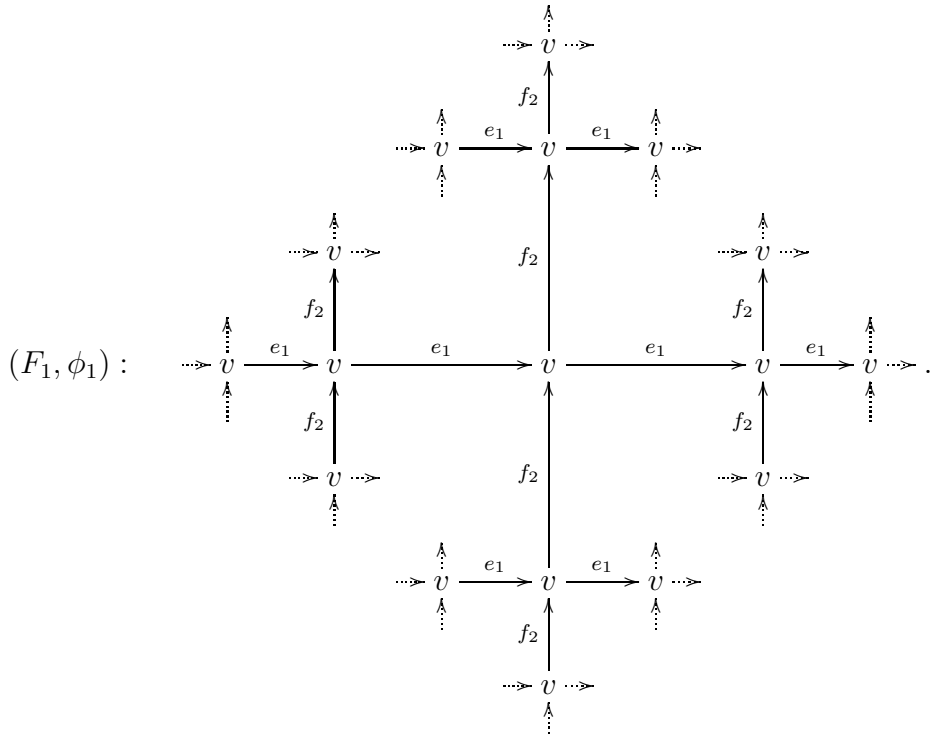
\begin{tabular}{l}
$\xymatrix@C=10pt@R=10pt{
&&&&&&&&&&&&&&&\\
&&&&&&&\ar@{..>}[r]&v\ar@{..>}[u]\ar@{..>}[r]&&&&&&&\\
&&&&&&&&&&&&&&&\\
&&&&&\ar@{..>}[r]&v\ar@{..>}[u]\ar^{e_1}[rr]&&v\ar[uu]^{f_2}\ar^{e_1}[rr]&&v\ar@{..>}[u]\ar@{..>}[r]&&&&&\\
&&&&&&\ar@{..>}[u]&&&&\ar@{..>}[u]&&&&&\\
&&&\ar@{..>}[r]&v\ar@{..>}[r]\ar@{..>}[u]&&&&&&&\ar@{..>}[r]&v\ar@{..>}[u]\ar@{..>}[r]&&&\\
&&&&&&&&&&&&&&&\\
(F_1,\phi_1):&\ar@{..>}[r]&v\ar@{..>}[u]\ar^{e_1}[rr]&&v\ar^{f_2}[uu]\ar^{e_1}[rrrr]&&&&v\ar^{f_2}[uuuu]\ar^{e_1}[rrrr]&&&&v\ar^{f_2}[uu]\ar^{e_1}[rr]&&v\ar@{..>}[u]\ar@{..>}[r]&.\\
&&\ar@{..>}[u]&&&&&&&&&&&&\ar@{..>}[u]&\\
&&&\ar@{..>}[r]&v\ar^{f_2}[uu]\ar@{..>}[r]&&&&&&&\ar@{..>}[r]&v\ar^{f_2}[uu]\ar@{..>}[r]&&&\\
&&&&\ar@{..>}[u]&&&&&&&&\ar@{..>}[u]&&&\\
&&&&&\ar@{..>}[r]&v\ar@{..>}[u]\ar^{e_1}[rr]&&v\ar^{f_2}[uuuu]\ar^{e_1}[rr]&&v\ar@{..>}[u]\ar@{..>}[r]&&&&&\\
&&&&&&\ar@{..>}[u]&&&&\ar@{..>}[u]&&&&&\\
&&&&&&&\ar@{..>}[r]&v\ar@{..>}[r]\ar^{f_2}[uu]&&&&&&&\\
&&&&&&&&\ar@{..>}[u]&&&&&&&
}$
\\
$
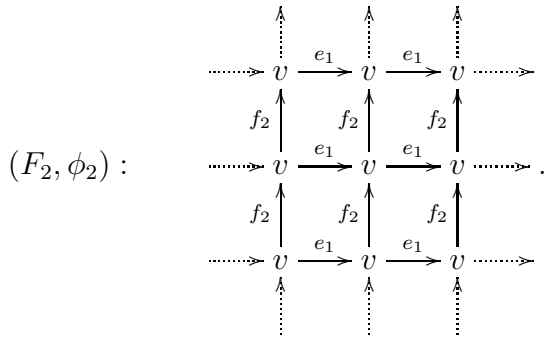
\xymatrix@C=20pt@R=20pt{
&&&&&\\
&\ar@{..>}[r]&v\ar^{e_1}[r]\ar@{..>}[u]&v\ar^{e_1}[r]\ar@{..>}[u]&v\ar@{..>}[r]\ar@{..>}[u]&\\
(F_2,\phi_2):&\ar@{..>}[r]&v\ar^{e_1}[r]\ar^{f_2}[u]&v\ar^{e_1}[r]\ar^{f_2}[u]&v\ar^{f_2}[u]\ar@{..>}[r]&.\\
&\ar@{..>}[r]&v\ar^{e_1}[r]\ar^{f_2}[u]&v\ar^{e_1}[r]\ar^{f_2}[u]&v\ar^{f_2}[u]\ar@{..>}[r]&\\
&&\ar@{..>}[u]&\ar@{..>}[u]&\ar@{..>}[u]&
}$
\end{tabular}\\
\begin{tabular}{ll}
$
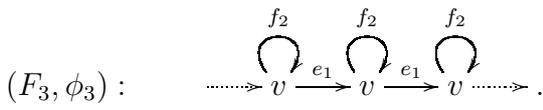
\xymatrix@C=20pt@R=20pt{
(F_3,\phi_3):&\ar@{..>}[r]&v\ar@(ul,ur)^{f_2}\ar^{e_1}[r]&v\ar@(ul,ur)^{f_2}\ar^{e_1}[r]&v\ar@(ul,ur)^{f_2}\ar@{..>}[r]&.}$
&$
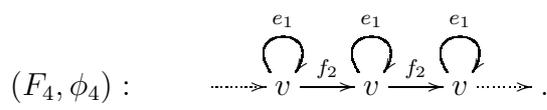
\xymatrix@C=20pt@R=20pt{
(F_4,\phi_4):&\ar@{..>}[r]&v\ar@(ul,ur)^{e_1}\ar^{f_2}[r]&v\ar@(ul,ur)^{e_1}\ar^{f_2}[r]&v\ar@(ul,ur)^{e_1}\ar@{..>}[r]&.
}$\\
$
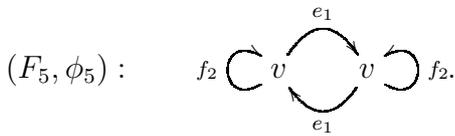
\xymatrix@C=20pt@R=20pt{
(F_5,\phi_5):&&v\ar@(dl,ul)^{f_2}\ar@/^1.3pc/^{e_1}[r]&v\ar@(dr,ur)_{f_2}\ar@/^1.3pc/^{e_1}[l]&.}$
&$
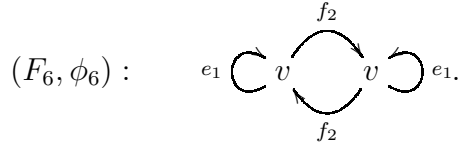
\xymatrix@C=20pt@R=20pt{
(F_6,\phi_6):&&v\ar@(dl,ul)^{e_1}\ar@/^1.3pc/^{f_2}[r]&v\ar@(dr,ur)_{e_1}\ar@/^1.3pc/^{f_2}[l]&.}$\\
$
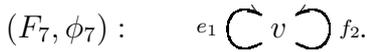
\xymatrix@C=20pt@R=20pt{
(F_7,\phi_7):&&v \ar@(dl,ul)^{e_1} \ar@(dr,ur)_{f_2}&.}$&
\end{tabular}\\
\\
\\
All the representation graphs $(F_i,\phi_i)~(1\leq i\leq 7)$ lie in the same $\leftrightharpoons$-equivalence class $C$. One checks easily that $(F_1,\phi_1)\cong (T_C,\xi_C)$ and $(F_7,\xi_7)\cong (S_C,\zeta_C)$ (cf. Example \ref{exuniv}). Moreover, we have
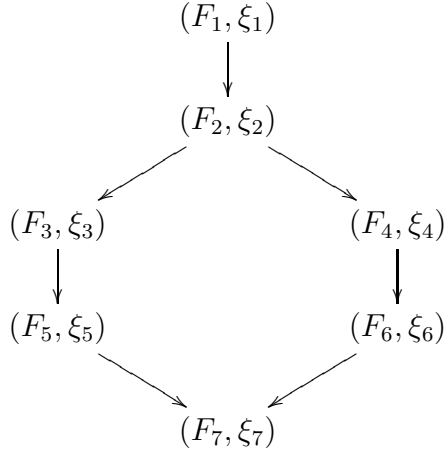
\[\xymatrix@C=20pt@R=20pt{
&(F_1,\xi_1)\ar[d]&\\
&(F_2,\xi_2)\ar[dl]\ar[dr]&\\
(F_3,\xi_3)\ar[d]&&(F_4,\xi_4)\ar[d]\\
(F_5,\xi_5)\ar[dr]&&(F_6,\xi_6)\ar[dl]\\
&(F_7,\xi_7)&
}\]
where an arrow $(F_i,\phi_i)\longrightarrow(F_j,\phi_j)$ means that $(F_j,\phi_j)$ is a quotient of $(F_i,\phi_i)$.
\end{example}


\section{Weighted Leavitt path algebras and their representations}\label{weightedrbgh}

Weighted Leavitt path algebras were introduced in~\cite{H-1}, as a generalisation of Leavitt path algebras. They are a model for the algebras $L_K(n,m)$ which could not be obtained via the classical theory of Leavitt path algebras. The structure of weighted Leavitt path algebras was further investigated in~\cite{HPR, Raimund1,Raimund3,Raimund4}. Further generalisations of these algebras were also carried out in \cite{Raimund2,mohan}. 

We recall the notion of weighted Leavitt path algebras.

\begin{definition}\label{def3}
Let $(E,w)$ be a weighted graph.  The $K$-algebra $L_K(E,w)$ presented by the generating set $\{v,e_i,e_i^*\mid v\in E^0, e\in E^1, 1\leq i\leq w(e)\}$ and the relations
\begin{enumerate}
\item $uv=\delta_{uv}u\quad(u,v\in E^0)$,
\medskip
\item $s(e)e_i=e_i=e_ir(e),~r(e)e_i^*=e_i^*=e_i^*s(e)\quad(e\in E^1, 1\leq i\leq w(e))$,
\medskip
\item 
$\sum\limits_{1\leq i\leq w(v)}e_i^*f_i= \delta_{ef}r(e)\quad(v\in E^0, e,f\in s^{-1}(v))$ and
\medskip 
\item $\sum\limits_{e\in s^{-1}(v)}e_ie_j^*= \delta_{ij}v\quad(v\in E^0,1\leq i, j\leq w(v))$
\end{enumerate}
is called the {\it weighted Leavitt path algebra} of $(E,w)$.  In relations (3) and (4) we set $e_i$ and $e_i^*$ zero whenever $i > w(e)$. 
\end{definition}

Throughout the paper, we continue to denote a weighted graph $(E,w)$ by $E$ and the weighted Leavitt path algebra $L_K(E,w)$ by $L_K(E)$.  We also interchangeably use the terminology the Leavitt path algebra of a weighted graph instead of a weighted Leavitt path algebra. If all the edges of the weighted graph $E$ have weight one, then $L_K(E,w)$ coincides with the classical Leavitt path algebra $L_K(E)$ (see Example~\ref{exex1}). 

Weighted Leavitt path algebras are involutary graded rings with unit if $E^0$ is finite and local units otherwise. In fact,  the weighted Leavitt path algebra $L_K(E)$ is a $\mathbb Z^n$-graded ring, where $n=\max\{w(e) \mid e \in E\}$. 
The grading is defined as follows:  Consider the free ring $F_E$ generated by  $\{v \mid v \in E^0\}$, $\{ e_1,\dots e_{w(e)} \mid e \in E^1 \}$ and $\{ e_1^*,\dots e_{w(e)}^* \mid e \in E^1 \}$, with the coefficients in $K$, set for $v \in E^0$, $\deg(v)=0$, for $e \in E^1$, $1\leq i\leq w(e)$, $\deg(e_i)=(0,\dots,0,1,0,\dots)$ and $\deg(e_i^*)=(0,\dots,0,-1,0,\dots) \in  \mathbb Z^n$, where 
$1$ and $-1$ are in the $i$-th component, respectively. This defines a $\mathbb Z^n$-grading on the free ring $F_E$.  (Note that $n$ could be infinite).   Since all the relations in Definition~\ref{def3} involve  homogeneous elements, so the quotient of $F_E$  by the homogeneous ideal generated by these relations, i.e., $L_K(E)$ is also a $\mathbb Z^n$-graded ring.

\begin{example}\label{wlpapp}
The Leavitt path algebra of a weighted graph consisting of one vertex and $n+k$ loops of weight $n$ 
is isomorphic to the Leavitt algebra $L_K(n,n+k)$. To show this, let $E^1=\{y_1,\dots,y_{n+k}\}$ with $w(y_i)=n$, $1\leq i\leq n+k$. Denote the $n$ edges corresponding to the (structure) edge $y_i\in E^{1}$ by $\{y_{1i},\dots,y_{ni}\}$. We visualise this data as follows:
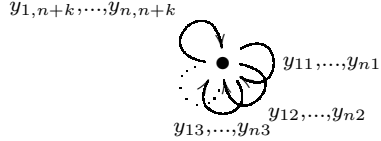
\begin{equation*}
\xymatrix{
& \bullet \ar@{.}@(l,d) \ar@(ur,dr)^{y_{11},\dots,y_{n1}} \ar@(r,d)^{y_{12},\dots,y_{n2}} \ar@(dr,dl)^{y_{13},\dots,y_{n3}} \ar@(l,u)^{y_{1,n+k},\dots,y_{n,n+k}}& 
}
\end{equation*}
Set $x_{sr}=y_{rs}^*$ for $1\leq r\leq n$ and $1\leq s\leq n+k$ and arrange the $y$'s and $x$'s in the matrices
\begin{equation*} 
Y=\left( 
\begin{matrix} 
y_{11} & y_{12} & \dots & y_{1,n+k}\\ 
y_{21} & y_{22} & \dots & y_{2,n+k}\\ 
\vdots & \vdots & \ddots & \vdots\\ 
y_{n1} & y_{n2} & \dots & y_{n,n+k} 
\end{matrix} 
\right), \qquad 
X=\left( 
\begin{matrix} 
x_{11\phantom{n+{},}} & x_{12\phantom{n+{},}} & \dots & x_{1n\phantom{n+{},}}\\ 
x_{21\phantom{n+{},}} & x_{22\phantom{n+{},}} & \dots & x_{2n\phantom{n+{},}}\\ 
\vdots & \vdots & \ddots & \vdots\\ 
x_{n+k,1} & x_{n+k,2} & \dots & x_{n+k,n} 
\end{matrix} 
\right) 
\end{equation*} 
Then Condition~(3) of Definition~\ref{def3} precisely says that $Y\cdot X=I_{n,n}$ and Condition~(4) is equivalent to 
$X\cdot Y=I_{n+k,n+k}$ which are the generators of $L_K(n,n+k)$. 
\end{example}

\begin{example}\label{exex1}
Let $(E,w)$ be a weighted graph where $w: E\rightarrow \N$ is the constant map $w(e)=1$, for all $e \in E$. Then $L_K(E,w)$ coincides with  the usual Leavitt path algebra $L_K(E)$. 
\end{example}

\begin{example}
Consider a weighted graph with one vertex and one edge $e$ of weight $n$   and an unweighted graph $F$ with one vertex and $n$ edges $\{e_1,\dots,e_n\}$. Then the map 
\begin{align*}
L_K(E,w) & \rightarrow L_K(F),\\
e_i &\mapsto e_i^*
\end{align*}
induces an isomorphism, i.e., 
\[\xymatrix{
L_K\big(\!\!\!\!\!\!\!\!\!\!\!\!\!& \bullet\ar@(ul,ur)^{\mathtt{e_{1},\dots,e_{n}}}
},w\big)\cong
\xymatrix{
L_K\big(\!\!\!&   \bullet \ar@{.}@(l,d) \ar@(ur,dr)^{e_{1}} \ar@(r,d)^{e_{2}} \ar@(dr,dl)^{e_{3}} 
\ar@(l,u)^{e_{n}}& 
}\big).\]
Note that this isomorphism is not graded as $L(E,w)$ is $\mathbb Z^n$-graded, whereas $L(F)$ is just $\mathbb Z$-graded. 
\end{example}

\subsection{Representations of weighted Leavitt path algebras}
Let $E$ be a weighted graph and $L_K(E)$ the weighted Leavitt path algebra associated to $E$. 
In this section, for a representation graph $(F,\phi)$ of $E$, we construct a $L_K(E)$-module $V_F$. This gives rise to a functor from the category of representations of the graph $E$ (see~\S\ref{catrepa}) to the category of (right) $L_K(E)$-modules:
\begin{align*}
V: \RG(E) &\longrightarrow \Modd L_K(E),\\
(F,\phi) & \longmapsto V_F.\notag
\end{align*}
 
We will then investigate the functor $V$.  In Theorem~\ref{thmirr}, we show that $V_F$ is a simple $L_K(E)$-module if and only if $F$ is an irreducible representation.  In Theorem~\ref{thmdecomp} we further show that a universal representation $T$ of $E$ gives an indecomposable $L_K(E)$-module $V_T$. 
 


For a representation graph $(F,\phi)$ of the weighted graph $E$,  let $V_F$ be the $K$-vector space with basis $F^0$.  We show that there is a natural action of the weighted Leavitt path algebra $L_K(E)$ on the $K$-vector space $V_F$, resulting $V_F$ to be a $L_K(E)$-module.  For any $u\in E^0$, $e\in E^1$ and $1\leq i\leq w(e)$,  define endomorphisms $\sigma_u,\sigma_{e_i},\sigma_{e_i^*}\in \End_K(V_F)$ by
\begin{align*}
\sigma_u(v)          &=\begin{cases}v,\quad\quad\quad             &\text{if }\phi^0(v)=u \\0,                                           & \text{else} \end{cases},\\
\sigma_{e_i}(v)     &=\begin{cases}r_F(f),\quad                       &\text{if }\exists f\in s_F^{-1}(v):~\phi^1(f)=e_i\\0,     & \text{else} \end{cases},\\
\sigma_{e_i^*}(v)  &=\begin{cases}s_F(f),\quad                        &\text{if }\exists f\in r_F^{-1}(v):~\phi^1(f)=e_i\\0,    & \text{else} \end{cases},
\end{align*}
where $v
\in F^0$. It follows from the universal property of $L_K(E)$ that there is an algebra homomorphism 
\begin{equation}\label{direst}
\pi:L_K(E)\longrightarrow \End_K(V_F)^{\op},
\end{equation}
 such that $\pi(u)=\sigma_u$, $\pi(e_i)=\sigma_{e_i}$ and $\pi(e_i^*)=\sigma_{e_i^*}$. We refer to this representation, the {\it representation of $L_K(E)$ defined by $(F,\phi)$}. Clearly $V_F$ becomes a right $L_K(E)$-module by defining $x \cdot a:= \pi(a)(x)$, for any $a\in L_K(E)$ and $x\in V_F$.

The following lemma describes the action of monomial elements of the weighted Leavitt path algebra $L_K(E)$ on the $K$-vector space $V_F$. Note that by Lemma \ref{lemwell}, for any $p\in \Path(\hat E_d)$ and $u\in F^0$ there is at most one $v\in F^0$ such that $p\in\phi({}_u\!\Path_v(F_d))$.
 
\begin{lemma}\label{lemaction}
Let $E$ be a weighted graph and $(F,\phi)$ a representation of $E$ with the induced map $\phi: \Path(F_d)\rightarrow \Path(\hat E_d)$. 
If $p\in \Path(\hat E_d)$ and $u\in F^0$, then
\[u\cdot p=\begin{cases}
v,\quad&\text{if }p\in \phi({}_u\!\Path_v(F_d)),\text{ for some }v\in F^0,\\
0,\quad&\text{otherwise.}
\end{cases}\]
\end{lemma}
\begin{proof}
Let $p\in \Path(\hat E_d)$ and $u\in F^0$. Suppose that $p=u'$ for some $u'\in E^0$. If $p\in \phi({}_u\!\Path_v(F_d))$, for some $v\in F^0$, then clearly $u=v$ and $u'=\phi^0(v)$. Hence $u\cdot p=u=v$ as desired. If there is no $v\in F^0$ such that $p\in \phi({}_u\!\Path_v(F_d))$, then clearly $u'\neq\phi^0(v)$. Hence $u\cdot p=0$ as desired.\\
Suppose now that $p$ is nontrivial, say $p=x_1\dots x_n$ where $n\geq 1$ and $x_1,\dots,x_n\in \hat E_d^1$. We proceed by induction on $n$.

Case $n=1$: Suppose $p=e_i$ for some $e\in E^1$ and $1\leq i\leq w(e)$. If $p\in \phi({}_u\!\Path_v(F_d))$ for some $v\in F^0$, then clearly $p=\phi^1(f)$ for some $f\in F^1$ such that  $s_F(f)=u$ and $r_F(f)=v$. Hence $u\cdot p=r_F(f)=v$ as desired. If there is no $v\in F^0$ such that $p\in \phi({}_u\!\Path_v(F_d))$, then clearly there is no $f\in s_F^{-1}(v)$ such that $\phi^1(f)=e_i$. Hence $u\cdot p=0$ as desired.

Suppose $p=e_i^*$ for some $e\in E^1$ and $1\leq i\leq w(e)$. If $p\in \phi({}_u\!\Path_v(F_d))$ for some $v\in F^0$, then clearly $e_i=\phi^1(f)$ for some $f\in F^1$ such that $s_F(f)=v$ and $r_F(f)=u$. Hence $u\cdot p=s_F(f)=v$ as desired. If there is no $v\in F^0$ such that $p\in \phi^{-1}({}_u\!\Path_v(F_d))$, then clearly there is no $f\in r_F^{-1}(u)$ such that $\phi^1(f)=e_i$. Hence $u\cdot p=0$ as desired.

Case $n\to n+1$:
Suppose $p=x_1\dots x_nx_{n+1}$. If $p\in \phi({}_u\!\Path_v(F_d))$ for some $v\in F^0$, then clearly $x_1\dots x_n\in \phi({}_{u}\!\Path_{u'}(F_d))$ for some $u'\in F^0$ and moreover $x_{n+1}\in \phi({}_{u'}\!\Path_{v}(F_d))$. By the induction assumption we have $v \cdot x_1\dots x_n=u'$. It follows from the case $n=1$ that $u\cdot p=u'\cdot x_{n+1}=v$. \\
Suppose now that there is no $v\in F^0$ such that $p\in \phi({}_u\!\Path_v(F_d))$. If $x_1\dots x_n\in \phi^{-1}({}_{u}\!\Path_{u'}(F_d))$ for some $u'\in F^0$, then there is no $v\in F^0$ such that $x_{n+1}\in \phi({}_{u'}\!\Path_{v}(F_d))$. Hence, by the induction assumption, $u\cdot p=u'\cdot  x_{n+1}=0$ as desired. If there is no $u'\in F^0$ such that $x_1\dots x_n\in \phi({}_{u}\!\Path_{u'}(F_d))$, then, by the induction assumption, $u\cdot p=(u\cdot x_1\dots x_n)\cdot x_{n+1}=0\cdot x_{n+1}=0$ as desired. 
\end{proof}

\begin{corollary}\label{coraction}
If $a=\sum_{p\in \Path(\hat E_d)}k_pp\in L_K(E)$ and $u\in F^0$, then
\[u\cdot a=\sum_{v\in F^0}\Big(\sum_{p\in\phi({}_u\!\Path_v(F_d))}k_p\Big)v.\]
\end{corollary}

Let $T\rightarrow E$ be a covering of the weighted graph $E$ and $F$ a representation for $T$. By Lemma~\ref{hfgfghhffeee}, the $L_K(T)$-module $V_F$ is also a $L_K(E)$-module. In fact, for a particular kind of coverings, namely universal coverings, we can say more.

\begin{corollary}
Let $E$ be a weighted graph,  $T$ the universal cover of  $E$ and $F \hookrightarrow \hat T$ a connected representation graph of $T$. Then any $L_K(T)$-submodule $A\subseteq V_F$ is also an $L_K(E)$-module. 
\end{corollary}
\begin{proof}
Since by Lemma~\ref{hfgfghhffeee}, $F$ is a representation for $E$, $V_F$ is a an $L_K(E)$-module as well as an $L_K(T)$-module. Let $A \subseteq V_F$ be a $L_K(T)$-submodule and 
$a=\sum_i k_i v_i\in A$. It is enough to show that for an edge $e\in L_K(E)$, we have $ae, ae^* \in A$. We have $ae=\sum_i k_i v_ie=\sum_i k_i r(f_i)$, for those $i$ such that $\sigma(f_i)=e$. Since $F \hookrightarrow T$, we then have $af_i=k_ir(f_i)\in A$ as $A$ is a $L_K(T)$-module. Thus $ae=\sum_i k_i r(f_i) \in A$. The argument for $ae^*$ is similar.  
\end{proof}

Recall that a weighted Leavitt path algebra is a $\mathbb Z^n$-graded ring, where $n$ is the maximum weight in the graph (see ~\S\ref{weightedrbgh}). In particular, Leavitt path algebras (associated to graphs of weight one) are $\mathbb Z$-graded. This grading has played a crucial role in determining the algebraic structure of these algebras. We next characterise the representation graphs whose associated $L_K(E)$-modules are graded.  To do this we will use the following statement whose proof is easy and we will leave it to the reader.

\begin{lemma}\label{windynhgf}
 Let $A$ be a $\Gamma$-graded ring and $M$ a right $A$-module, where $\Gamma$ is a totally ordered abelian group. If there is a homogeneous element $a\in A$ with $\deg(a)\not = 0$ and $0\not = m\in M$ such that $ma=m$, then $M$ cannot be $\Gamma$-graded.  
 \end{lemma}

Let $F$ be a connected representation graph for the weighted graph $E$. Since $F$ is an immersion into $\hat E$, we have a monomorphism $\pi(F)\rightarrow \pi(\hat E)$. Combining this with the length map~(\ref{hfbvhdfhddds}), we obtain a homomorphism $| \, | : \pi(F) \rightarrow \mathbb Z^n$. This allows us to give a criterion on when the $L_K(E)$-module $V_F$ is a graded module.

\begin{theorem}\label{gradedrepres}
Let $E$ be a weighted graph and $F$ a connected representation graph of $E$. Then the $L(E)$-module $V_F$ is graded if and only if for (reduced) paths $p,q$ of $F$ with $s(p)=s(q)$ and $r(p)=r(q)$, we have $|p|=|q|$.  
\end{theorem}
\begin{proof}
Suppose  for any  paths $p,q$ of $F$, with $s(p)=s(q)$ and $r(p)=r(q)$ we have $|p|=|q|$.  
Choose a base vertex $v$ in $F$ and assigne $\deg(v)=0$. Define a map $\deg:F^0\rightarrow \mathbb Z^n$, by $\deg(w)= |q|$, where $q$ is a (reduced) path with $s(q)=v$ and $r(q)=w$. Since $F$ is connected and the length of paths connecting $v$ to $w$ are the same,   the map is well-defined. For $\alpha \in \mathbb Z^n$, define 
\[(V_F)_\alpha:=\Big \{ \sum k_i v_i\,  | \deg(v_i)=\alpha \Big \}.\]
Clearly $V_F=\bigoplus_{\alpha \in \mathbb Z^n} (V_F)_\alpha$. Next we check that $(V_F)_\alpha L_K(E)_\beta \subseteq (V_F)_{\alpha+\beta}$, where $\alpha,\beta \in \mathbb Z^n$. It is enough to show that $w p \in (V_F)_{\alpha+\beta}$, where $w \in (V_F)_{\alpha} \cap F^0$ and $p$ is a monomial in $L_K(E)_\beta$. Either $wp=0$ or by the definition of representation, one can lift $p$ to a unique path $q$ in $F$ such that $|q|=|p|$ and by Lemma~\ref{lemaction}, $wq=z$, where $z=r(q)$. Since there is a path $t$ with $s(t)=v$, $r(t)=w$ and $|t|=\alpha$, we have $s(tq)=v$, $r(tq)=z$. Since $|tq|=|t|+|q|$ we have $z\in (V_F)_{\alpha+\beta}$.  

Conversely, suppose $V_F$ is a graded $L_K(E)$-module. If there are paths $p,q$ in $F$ with $s(p)=s(q)$ and $r(p)=r(q)$ but  $|p|\not =|q|$, then $s(p)\not = pq^* \in \pi_1(F, s(p))$. Let $t$ be the image of $pq^*$ in $\hat E$. Then by the definition of representation, $v t= v$. Since $|pq^*|\not =0$, it follows from Lemma~\ref{windynhgf} that $V_F$ is not graded, which is a contradiction.  
\end{proof}


\begin{example}
Consider the weighted graph $ \xymatrix{
 v \ar@(dl,ul)^{\mathtt{e_1,e_2}} \ar@(dr,ur)_{\mathtt{f_1,f_2}} &
}$ from Example~\ref{excatone}. Theorem~\ref{gradedrepres} will then give the $L_K(2,2)$-modules $V_{F_1}$ and $V_{F_2}$ are $\mathbb Z^2$-graded  whereas $V_{F_3},V_{F_4}, V_{F_5}, V_{F_6}$ and the simple module $V_{F_7}$ are not graded.   

Consider the graph $E=\xymatrix{\bullet \ar@(u,r)^{\mathtt{e}} \ar@(ur,rd)^{\mathtt{f}} \ar@(r,d)^{\mathtt{g}} &}$ from the Introduction and its representation graphs $F$ and $G$ (see~(\ref{wlpa4}) and (\ref{wlpa45})). Theorem~\ref{gradedrepres} will then give the $L_K(E)$-module $V_G$ is $\mathbb Z$-graded whereas $V_F$ is not a graded module.  

\end{example}

\begin{corollary}\label{acyclicrep}
If $T$ is a universal representation of $E$ then $V_T$ is graded. 
\end{corollary}
\begin{proof}
If $T$ is a universal cover, then $\pi(T)=1$. This implies the condition of Theorem~\ref{gradedrepres}. 
\end{proof}

\subsection{Irreducible representations of weighted Leavitt path algebras}
In this section we characterise those representation graphs that induce simple modules on the level of Leavitt path algebras. We continue to denote by $E$ a weighted graph, $(F,\phi)$ a representation graph for $E$ and $V_F$ the $L_K(E)$-module defined by $(F,\phi)$. In the following theorem we show that a representation graph $F$ is irreducible (see Definition~\ref{defrepirres}) if and only if $V_F$ is a simple $L_K(E)$-module.


\begin{theorem}\label{thmirr} Let $E$ be a weighted graph, $(F,\phi)$ a representation graph for $E$ and $V_F$ the $L_K(E)$-module defined by $F$.
Then the following are equivalent.
\begin{enumerate}[\upshape(i)]
\item The $L_K(E)$-module $V_F$ is simple.
\smallskip
\item The representation graph $F$ is connected and 
\[\text{for any }x\in V_F\setminus\{0\}, \text{ there is }a\in L_K(E), \text{ such that } x \cdot a \in F^0.\]

\item The representation graph  $F$ is connected and  
\[\text{for any }x\in V_F\setminus\{0\},\text{ there is a }k\in K \text{ and a } p\in \Path(\hat E_d), \text{ such that } x \cdot kp \in F^0. \]

\item The representation graph  $F$ is an irreducible representation.
\end{enumerate}
\end{theorem}
\begin{proof} 
(i) $\Longrightarrow$ (iv). If $C$ is a connected component of $F$, then the subspace of $V_F$ generated by $C^0$ is clearly invariant under the action of $L:=L_K(E)$. Hence $F$ must be connected. Now assume that there are $u\neq v\in F^0$ such that $\phi({}_u\!\Path(F_d))= \phi({}_v\!\Path(F_d))$. Consider the submodule $(u-v)L\subseteq V_F$. Since $V_F$ is simple by assumption, we have $(u-v)L=V_F$. Hence there is an $a\in L$ such that $(u-v)a=v$. Clearly there is an $n\geq 1$, $k_1,\dots,k_n\in K^\times$ and pairwise distinct $p_1,\dots,p_n\in \Path(\hat E_d)$ such that $a=\sum_{s=1}^nk_sp_s$. We may assume that $(u-v)p_s\neq 0$, for any $1\leq s\leq n$. It follows from Lemma \ref{lemaction} that $p_s\in \phi({}_u\!\Path(F_d))= \phi({}_v\!\Path(F_d))$, for any $s$ and moreover, that $(u-v)p_s=u_s-v_s$ for some distinct $u_s,v_s\in F^0$. Hence 
\[v=(u-v)a=(u-v)(\sum_{s=1}^nk_sp_s)=\sum_{s=1}^nk_s(u_s-v_s)\]
which contradicts Lemma \ref{lembasis}.

\medskip 
(iv) $\Longrightarrow$ (iii). Let $x\in V_F\setminus \{0\}$. Then there is an $n\geq 1$, pairwise disjoint $v_1,\dots, v_n\in F^0$ and $k_1,\dots,k_n\in K^\times$ such that $x=\sum_{s=1}^nk_sv_s$. If $n=1$, then $x\cdot k_1^{-1}\phi^{0}(v_1)= v_1$. Suppose now that $n>1$. By assumption, we can choose a $p_1\in \phi({}_{v_1}\!\Path(F_d))$ such that $p_1\not\in\phi({}_{v_2}\!\Path(F_d))$. Clearly $x\cdot p_1\neq 0$  is a linear combination of at most $n-1$ vertices from $F^0$. Proceeding this way, we obtain paths $p_1,\dots,p_m$ such that $x\cdot p_1\dots p_m=kv$ for some $k\in K^\times$ and $v\in F^0$. Hence $x\cdot k^{-1}p_1\dots p_m=v$.

\medskip 

(iii) $\Longrightarrow$ (ii). Trivial.

\medskip 

(ii) $\Longrightarrow$ (i). Let $U\subseteq V_F$ be a nonzero $L_K(E)$-submodule and $x\in U\setminus\{0\}$. By assumption, there is an $a\in L_K(E)$ and a $v\in F^0$ such that $v=x\cdot a\in U$. Let now $v'$ be an arbitrary vertex in $F^0$. Since by assumption $F$ is connected, there is a $p\in {}_{v}\!\Path_{v'}(F_d)$. It follows from Lemma \ref{lemaction} that $v'=v\cdot \phi(p)\in U$. Hence $U$ contains $F^0$ and thus $U=V_F$.
\end{proof}

\subsection{The fullness of the functor $V$}

Let $\alpha:(F,\phi)\rightarrow(G,\psi)$ be a morphism in $\RG(E)$. Then it induces a surjective $L_K(E)$-module homomorphism $V_\alpha:V_F\rightarrow V_G$ such that $V_\alpha(u)=\alpha^0(u)$, for any $u\in F^0$.

The example below shows that $V$ is not full, namely, there can be $L_K(E)$-module homomorphisms $V_F\rightarrow V_G$ that are not induced by a morphism $(F,\phi)\rightarrow(G,\psi)$. 

\begin{example}\label{excat11}
Suppose $E$ is the weighted graph 
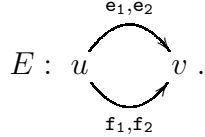
\[E:\xymatrix{
 u \ar@/^1.3pc/^{\mathtt{e_1,e_2}}[r]\ar@/_1.3pc/_{\mathtt{f_1,f_2}}[r] &v
}.\]
Consider the representation graphs $(F,\phi)$ and $(G,\psi)$ for $E$ given below.\\
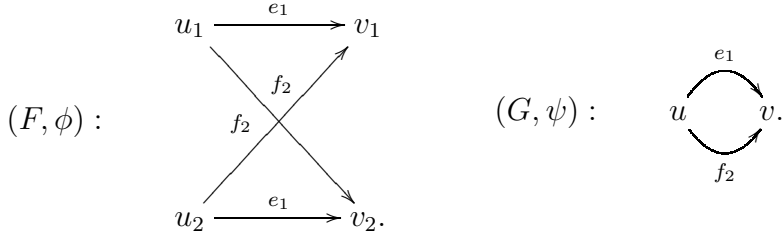
\begin{tabular}{ll}
$\xymatrix@C=20pt@R=20pt{
&u_1\ar^{e_1}[rr]\ar^(.4){f_2}[ddrr]&&v_1\\
(F,\phi):&&&\\
&u_2\ar^{e_1}[rr]\ar[uurr]^(.4){f_2}&&v_2.
}$
&\quad\quad 
$\xymatrix@C=20pt@R=20pt{&&\\
(G,\psi):&u \ar@/^1.3pc/^{e_1}[r]\ar@/_1.3pc/_{f_2}[r] &v.\\&&}$
\end{tabular}\\
\\

Since $(G,\psi)$ is a quotient of $(F,\phi)$, by Theorem \ref{thmadm}, there is a morphism $\alpha:(F,\phi)\rightarrow(G,\psi)$, which induces a homomorphism $V_\alpha:V_F\rightarrow V_G$. Although $(F,\phi)$ is not a quotient of $(G,\psi)$, there is an $L_K(E)$-module homomorphism $\sigma:V_G\rightarrow V_F$ in the opposite direction such that $\sigma(u)=u_1+u_2$ and $\sigma(v)=v_1+v_2$. Note that $V_\alpha$ and $\sigma$ are not inverse to each other.
\end{example}

\begin{question}\label{Q1}
Can it happen that $V_F\cong V_G$ as $L_K(E)$-modules although $(F,\phi)\not\cong(G,\psi)$ in $\RG(E)$?
\end{question}


The authors do not know the answer to Question \ref{Q1}. But we will show that if $V_F\cong V_G$, then $(F,\phi)$ and $(G,\psi)$ must be equivalent.

\begin{lemma}\label{lemmodule}
Let $(F,\phi)$ and $(G,\psi)$ be objects in $\RG(E)$ and let $\sigma:V_F\rightarrow V_G$ be a $L_K(E)$-module homomorphism. Let $u\in F^0$ and $\sigma(u)=\sum_{s=1}^nk_sv_s$, where $n\geq 1$, $k_1,\dots,k_n\in K^\times$ and $v_1,\dots,v_n$ are pairwise distinct vertices from $G^0$. Then $\phi({}_u\!\Path(F_d))=\psi({}_{v_s}\!\Path(G_d))$, for any $1\leq s\leq n$.
\end{lemma}
\begin{proof} 
Let $p\in \Path(\hat E_d)$ such that $p\not\in\phi({}_u\!\Path(F_d))$. Then
\[0=\sigma(0)=\sigma(u\cdot p)=\sigma(u)\cdot p=\sum_{s=1}^nk_sv_s\cdot p=\sum_{s=1}^nk_s(v_s\cdot p)\]
by Lemma \ref{lemaction}. One more application of Lemma \ref{lemaction} gives that $v_s\cdot p =0$, for any $1\leq s\leq n$, whence $p\not\in \psi({}_{v_s}\!\Path(G_d))$ for any $1\leq s\leq n$. Hence we have shown that $\phi({}_u\!\Path(F_d))\supseteq\psi({}_{v_s}\!\Path(G_d))$, for any $1\leq s\leq n$. It follows from Lemma \ref{lempath} that $\phi({}_u\!\Path(F_d))=\psi({}_{v_s}\!\Path(G_d))$, for any $1\leq s\leq n$.
\end{proof}

\begin{proposition}\label{prop42}
Let $(F,\phi)$ and $(G,\psi)$ be objects in $\RG(E)$. If there is a nonzero $L_K(E)$-module homomorphism $\sigma:V_F\rightarrow V_G$, then $(F,\phi)\leftrightharpoons (G,\psi)$.
\end{proposition}
\begin{proof}
The  proposition follows from Lemma \ref{lemmodule} and the definition of the equivalence $\leftrightharpoons$.
\end{proof}

\begin{proposition}\label{prop43}
Let $(F,\phi)$ and $(G,\psi)$ be irreducible representation graphs for the weighted graph $E$. Then $V_F\cong V_G$ as $L_K(E)$-modules if and only if  $(F,\phi)\cong(G,\psi)$.
\end{proposition}
\begin{proof}
Clearly isomorphic objects in $\RG(E)$ yield isomorphic $L_K(E)$-modules. Hence we only have to show that $V_F\cong V_G$ implies $(F,\phi)\cong(G,\psi)$. Suppose that $V_F\cong V_G$. Then $(F,\phi)\leftrightharpoons(G,\psi)$ by Proposition \ref{prop42}, i.e., $(F,\phi)$ and $(G,\psi)$ are in the same $\leftrightharpoons$-equivalence class $C$. Since they are irreducible, it follows from Corollary \ref{corcat1} that $(F,\phi)\cong (S_C,\zeta_C)\cong (G,\psi)$.
\end{proof}

\subsection{Indecomposability of $L_K(E)$-modules $V_F$}
Recall that for a ring $R$, an $R$-module is called \emph{indecomposable} if it is non-zero and cannot be written as a direct sum of two non-zero submodules. It is easy to see that an $R$-module $M$ is indecomposable if and only if the only idempotent elements of the endomorphism ring $\End_R(M)$ are $0$ and $1$. 

We will show in Theorem~\ref{thmdecomp} that for any universal representation $T$ of the graph $E$, the $L_K(E)$-module $V_T$ is indecomposable.  However, Example \ref{exdec} below shows that in general the indecomposibility of $L_K(E)$-module $V_F$, for a representation  $(F,\phi)$, depends on the ground field $K$.  \begin{example}\label{exdec}
Let $E$ be the graph 
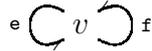
\[\xymatrix{
 v \ar@(ul,dl)_{\mathtt{e}} \ar@(dr,ur)_{\mathtt{f}} &
}\]
and the graph $F$ below the representation graph for $E$:
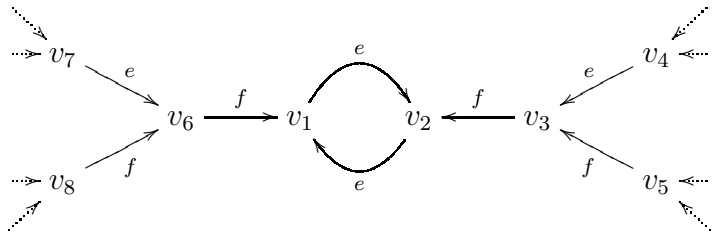
\[\xymatrix@=10pt{
& && && &  & & & & &&&&&&\\
& & & && &  & & & & &&&&&\\
& & & v_7 \ar@{<.}[l] \ar@{<.}[ul] \ar[drr]^{e} &&&&&  &&&&&v_4 \ar@{<.}[r] \ar@{<.}[ur]   \ar[lld]_{e}& & &  \\
& & & && v_6\ar[rr]^{f}&& v_1 \ar@/^1.7pc/[rr]^{e} &&v_2 \ar@/^1.7pc/[ll]^{e} & & v_3\ar[ll]_{f}& & & && & & &\\
& & & v_8  \ar@{<.}[l]  \ar@{<.}[dl] \ar[urr]_{f}&&&&&&&& &&v_5 \ar@{<.}[r] \ar@{<.}[dr]   \ar[llu]^{f}& &  \\
& & & &&& &  & & & & &&&&&\\
}\]
If $\epsilon\in \End_L(V_F)$, then by Lemma \ref{lemmodule} there are $k,l\in K$ such that $\epsilon(v_1)=kv_1+lv_2= v_1\cdot(kv+le)$. Conversely, if $k,l\in K$, then there is a uniquely determined endomorphism $\epsilon\in \End_L(V_F)$ such that $\epsilon(v_1)=kv_1+lv_2$. This yields a bijection between $\End_L(V_F)$ and $K\times K$.

Let now $\epsilon\in \End_L(V_F)$ be the endomorphism corresponding to a pair $(k,l)\in K\times K$. Since $v_1$ generates the $L_K(E)$-module $V_F$, the endomorphism $\epsilon$ is idempotent if and only if  $\epsilon(v_1)=\epsilon^2(v_1)$. Clearly
\[\epsilon^2(v_1)=v_1\cdot(kv+le)^2 =v_1\cdot(k^2v+2kle+l^2e^2) =(k^2+l^2)v_1+2klv_2.\]
Thus $\epsilon$ is idempotent if and only if 
\begin{equation}
    \label{eq:lk:1}
    k=k^2+l^2\text{ and }l=2kl.
\end{equation}
If $2=0$ in $K$, then the only solutions for \eqref{eq:lk:1} are $(k,l)=(0,0)$ and $(k,l)=(1,0)$. The corresponding endomorphisms are $\epsilon=0$ and $\epsilon=\id_{V_F}$. Thus $V_F$ is an indecomposable $L_K(E)$-module in this case. If $2\neq 0$, then there are two more solutions for \eqref{eq:lk:1}, namely $(k,l)=(1/2,1/2)$ and $(k,l)=(1/2,-1/2)$. Thus $V_F$ is not indecomposable in this case. 
\end{example}

Fix an $\leftrightharpoons$-equivalence class $C$ and define $(T,\xi)=(T_C,\xi_C)$ and $(S,\zeta)=(S_C,\zeta_C)$ as in Subsection~\ref{subsecnn}. Since the $L_K(E)$-module $V_S$ is simple, it is indecomposable. We will show that the $L_K(E)$-module $V_T$ is indecomposable too.

Let $p,p'\in \Path(\hat E_d)$ be non-backtracking paths such that $r(p)=s(p')=v$. We define a non-backtracking path $p\ast p'\in \Path(\hat E_d)$ as follows. If $p,p'\in \Path(\hat E_d)\setminus \{v\}$ let $p\ast p'$ be the element of $\Path(\hat E_d)$ one gets by removing all subwords of the form $e_ie_i^*$ and $e_i^*e_i$ from the 
juxtaposition $pp'$ (if $p'=p^*$, then $p\ast p':=v)$. Moreover define $v\ast p:=p$ and $p\ast v:=p$.

Define the group 
\begin{equation}\label{gfreegreop}
G:=\big \{p\in \phi({}_u\!\Path(F_d))_{\nb}\mid v_p\sim v_{\phi(u)}\big \}.
\end{equation} In Proposition~\ref{propgroup}, we will show that $(G,\ast)$ is a free group. We need several lemmas. 

\begin{lemma}\label{lemT0}
If $p\in \phi({}_u\!\Path(F_d))_{\nb}$, then $v_p=v_{\phi(u)}\cdot p$.
\end{lemma}
\begin{proof}
Clearly $v_p=v_{\phi(u)}\cdot p$ if $p=\phi(u)$. Suppose now that $p=x_1\dots x_n$ for some $x_1,\dots ,x_n\in \hat E^1\cup (\hat E^1)^*$. For $1\leq i\leq n$ define 
\[\hat e_{x_1\dots x_i}:=\begin{cases}e_{x_1\dots x_i},&\text{ if }x_i\in \hat E^1,\\
e^*_{x_1\dots x_i},&\text{ if }x_i\in (\hat E^1)^*.\end{cases}.\]
One checks easily that $q:=\hat e_{x_1}\hat e_{x_1x_2}\dots \hat e_{x_1\dots x_n}\in {}_{v_{\phi(u)}\!\Path_{v_p}}(T_d)$ and $\xi(q)=p$. It follows from Lemma \ref{lemaction} that $v_{\phi(u)}\cdot p=v_p$. 
\end{proof}

\begin{lemma}\label{lemT1}
Let $p,p'\in G$. Then $p\ast p'$ is defined and $p \ast p'\in \phi({}_u\!\Path(F_d))_{\nb}$.
\end{lemma}
\begin{proof}
Let $r,r'\in {}_u\!\Path(F_d)$ such that $\phi(r)=p$ and $\phi(r')=p'$. By the proof of Lemma \ref{lemT0} there are (uniquely determined) paths $q\in {}_{v_{\phi(u)}\!\Path_{v_{p}}}(T_d)$ and $q'\in {}_{v_{\phi(u)}\!\Path_{v_{p'}}}(T_d)$ such that $\xi(q)=p$ and $\xi(q')=p'$. 
Let $\alpha:(T,C)\to (F,\phi)$ be the morphism defined in the proof of Proposition \ref{propuniv}. One checks easily that $\alpha(q)=r$ and $\alpha(q')=r'$. Since $v_{p}\sim v_{\phi(u)}$, there is a $q''\in {}_{v_{p}}\!\Path(T_d)$ such that $\xi(q'')=p'=\xi(q')$. Set $r'':=\alpha(q'')$. Clearly $s(r'')=s(\alpha(q''))=\alpha(s(q''))=\alpha(v_{p'})=\alpha(r(q))=r(\alpha(q))=r(r)$. Hence $rr''\in {}_u\!\Path(F_d)$. It follows that $pp'=\phi(rr'')\in \phi({}_u\!\Path(F_d))$. In particular $r(p)=s(p')=\phi(u)$ and hence $p \ast p'$ is defined.\\
Write $pp'=x_1\dots x_n$ and $rr''=y_1\dots y_n$. Suppose that $x_jx_{j+1}=e_ie_i^*$ for some $1\leq j\leq n-1$, $e\in E^1$ and $1\leq i\leq w(e)$. Since $\phi(y_j)=x_j$ and $\phi(y_{j+1})=x_{j+1}$, there are $f,g\in F^1$ such that $y_j=f$ and $y_j=g^*$. Moreover $r(f)=r(g)$ and $\phi(f)=e_i=\phi(g)$. It follows from Definition \ref{defwp}(2) that $f=g$. Hence $y_1\dots y_{j-1}y_{j+1}\dots y_n\in {}_u\!\Path(F_d)$ and $\phi(y_1\dots y_{j-1}y_{j+1}\dots y_n)=x_1\dots x_{j-1}x_{j+1}\dots x_n$. Similarly, if $x_jx_{j+1}=e_i^*e_i$ for some $1\leq j\leq n-1$, $e\in E^1$ and $1\leq i\leq w(e)$, then $y_1\dots y_{j-1}y_{j+1}\dots y_n\in {}_u\!\Path(F_d)$ and $\phi(y_1\dots y_{j-1}y_{j+1}\dots y_n)=x_1\dots x_{j-1}x_{j+1}\dots x_n$ (follows from Definition \ref{defwp}(1)). Thus $p \ast p'\in \phi({}_u\!\Path(F_d))_{\nb}$.
\end{proof}

\begin{lemma}\label{lemT2}
Let $p,p'\in G$. Then $p\ast p'\in G$.
\end{lemma}
\begin{proof}
By the previous lemma $p\ast p'$ is defined and $p \ast p'\in \phi({}_u\!\Path(F_d))_{\nb}$. It remains to show that $v_{p\ast p'}\sim v_{\phi(u)}$. By the proof of Lemma \ref{lemT0} there are paths $q\in {}_{v_{\phi(u)}}\!\Path_{v_{p}}(T_d)$ and $q'\in {}_{v_{\phi(u)}}\!\Path_{v_{p'}}(T_d)$ such that $\xi(q)=p$ and $\xi(q')=p'$. Since $v_{p}\sim v_{\phi(u)}$, there is a path $q''\in {}_{v_{p}}\!\Path(T_d)$ such that $\xi(q'')=p'=\xi(q')$. By Lemmas \ref{lemT0} and \ref{lemaction} we have 
\[v_{p\ast p'}=v_{\phi(u)}\cdot (p\ast p')=v_{\phi(u)}\cdot (pp')=r(qq'')=r(q'')
\]
and hence $q''\in {}_{v_{p}}\!\Path_{v_{p\ast p'}}(T_d)$.

We will show that $v_{p\ast p'}\sim v_{p'}$ which implies $v_{p\ast p'}\sim v_{\phi(u)}$. Let $p''\in \xi({}_{v_{p'}}\!\Path(T_d))$. Then there is an $r\in {}_{v_{p'}}\!\Path(T_d)$ such that $\xi(r)=p''$. Since $q'r\in {}_{v_{\phi(u)}}\!\Path(T_d)$, we have $p'p''=\xi(q'r)\in\xi({}_{v_{\phi(u)}}\!\Path(T_d))$. Since $v_{\phi(u)}\sim v_{p}$, there is an $s\in {}_{v_{p}}\!\Path(T_d)$ such that $\xi(s)=p'p''$. Write $s=s_1s_2$ such that $\xi(s_1)=p'$ and $\xi(s_2)=p''$. Clearly $s(s_1)=v_{p}$. It follows from Lemma \ref{lemwell} that $s_1=q''$. Since $q''\in {}_{v_{p}}\!\Path_{v_{p\ast p'}}(T_d)$, we obtain $s_2\in {}_{v_{p\ast p'}}\!\Path(T_d)$. Hence $p''=\xi(s_2)\in \xi({}_{v_{p\ast p'}}\!\Path(T_d))$. We have shown that $\xi({}_{v_{p'}}\!\Path(T_d))\subseteq  \xi({}_{v_{p\ast p'}}\!\Path(T_d))$. It follows from Lemma \ref{lempath} that $\xi({}_{v_{p'}}\!\Path(T_d))= \xi({}_{v_{p\ast p'}}\!\Path(T_d))$ and thus $v_{p\ast p'}\sim v_{p'}$.
\end{proof}

\begin{lemma}\label{lemT3}
Let $p\in G$. Then $p^*\in G$.
\end{lemma}
\begin{proof}
By the proof of Lemma \ref{lemT0} there is a path $q\in {}_{v_{\phi(u)}}\!\Path_{v_{p}}(T_d)$ such that $\xi(q)=p$. Clearly $q^*\in {}_{v_{p}}\!\Path_{v_{\phi(u)}}(T_d)$ and $\xi(q^*)=p^*$. Since $v_{p}\sim v_{\phi(u)}$, there is a path $q'\in {}_{v_{\phi(u)}}\!\Path(T_d)$ such that $\xi(q')=p^*$. Let $\alpha:(T,\xi)\to (F,\phi)$ be the morphism defined in the proof of Proposition \ref{propuniv}. Then $\alpha(q')\in {}_{u}\!\Path(F_d)$ and $\phi(\alpha(q'))=\xi(q')=p^*$. Hence $p^*\in \phi({}_{u}\!\Path(F_d))$. Moreover, $p^*$ is not backtracking since $p$ is not backtracking.

It remains to show that $v_{p^*}\sim v_{\phi(u)}$. Let $p'\in \xi({}_{v_{p^*}}\!\Path(T_d))$. Then there is an $r\in{}_{v_{p^*}}\! \Path(T_d)$ such that $\xi(r)=p'$. By the proof of Lemma \ref{lemT0} there is a path $q''\in {}_{v_{\phi(u)}}\!\Path_{v_{p^*}}(T_d)$ such that $\xi(q'')=p^*$. Clearly $q''r\in {}_{v_{\phi(u)}}\!\Path(T_d)$ and hence $p^*p'=\xi(q''r)\in\xi({}_{v_{\phi(u)}}\!\Path(T_d))$. Since $v_{\phi(u)}\sim v_{p}$, there is an $s\in {}_{v_{p}}\!\Path(T_d)$ such that $\xi(s)=p^*p'$. Write $s=s_1s_2$ such that $\xi(s_1)=p^*$ and $\xi(s_2)=p'$. Clearly $s_1\in {}_{v_{p}}\!\Path(T_d)$. It follows from Lemma \ref{lemwell} that $s_1=q^*$. Since $q^*\in {}_{v_{p}}\!\Path_{v_{\phi(u)}}(T_d)$, we obtain $s_2\in {}_{v_{\phi(u)}}\!\Path(T_d)$. Hence $p'=\xi(s_2)\in \xi({}_{v_{\phi(u)}}\!\Path(T_d))$. We have shown that $\xi({}_{v_{p^*}}\!\Path(T_d))\subseteq  \xi({}_{v_{\phi(u)}}\!\Path(T_d))$. It follows from Lemma \ref{lempath} that $\xi({}_{v_{p^*}}\!\Path(T_d))= \xi({}_{v_{\phi(u)}}\!\Path(T_d))$ and thus $v_{p^*}\sim v_{\phi(u)}$.
\end{proof}

\begin{proposition}\label{propgroup}
The group $(G,\ast)$ defined in (\ref{gfreegreop}) is a free group.
\end{proposition}
\begin{proof}
By Lemma \ref{lemT2}, $\ast$ defines a binary operation on $G$. Clearly this operation is associative. Moreover, $\phi(u)\ast p=p\ast \phi(u)$ for any $p\in G$. If $p\in G$, then $p^*\in G$ by Lemma \ref{lemT3}. Clearly $pp^*=p^*p=\phi(u)$. Hence $(G,\ast)$ is a group.

Let $F$ be the free group on $X:=\{e_i \mid e\in E^1,1\leq i\leq w(e)\}$. We identify $F$ with the set of all reduced words over the alphabet $X\cup X^{-1}$ (a word over $X\cup X^{-1}$ is reduced if it does contain any subwords of the form $xx^{-1}$ or $x^{-1}x$ where $x\in X$). The product of two reduced words $w$ and $w'$ is the reduced word one gets by removing all subwords of the form $xx^{-1}$ or $x^{-1}x$ from the juxtaposition $ww'$. 

Define a map $\theta:G\to F$ as follows. If $p=x_1\dots x_n\in G\setminus \{\phi(u)\}$, let $\theta(p)$ be the word one gets by replacing any letter $e_i^*$ by $e_i^{-1}$. Clearly $\theta(p)$ is a reduced word since $p$ is not backtracking. Moreover, we define $\theta(\phi(u))$ as the identity element of $F$, namely the empty word. Clearly $\theta$ is a group homomorphism. Moreover, $\theta$ is injective and hence $G$ is isomorphic to a subgroup of a free group. It follows from the Nielsen-Schreier theorem that $G$ is a free group.
\end{proof}

Let $W\subseteq V_T$ be the linear span of the $\sim$-equivalence class of $v_{\phi(u)}$, i.e. $W=\langle v_p\mid p\in G\rangle$. Let $A$ be the subalgebra of $L_K(E)$ generated by the image of the group $G$ in $L_K(E)$.

\begin{lemma}\label{lemAW}
The $K$-vector space $W$ is a right $A$-module where the action of $A$ on $W$ is induced by the action of $L_K(E)$ on $W$.
\end{lemma}
\begin{proof}
It suffices to show that if $p\in G$ and $v_{p'}\sim v_{\phi(u)}$, then $v_{p'}\cdot p =v_{p''}$ for some $v_{p''}\sim v_{\phi(u)}$. Since $v_{p'}\sim v_{\phi(u)}$, we have $p'\in G$. It follows from Lemmas \ref{lemT0} and \ref{lemT1} that
\[v_{p'}\cdot p =( v_{\phi(u)}\cdot p') \cdot p=v_{\phi(u)} \cdot p'p=v_{\phi(u)} \cdot p'\ast p=v_{p\ast p'}.\]
By Lemma \ref{lemT2} we have $p\ast p'\in G$. Thus $v_{p\ast p'}\sim v_{\phi(u)}$.
\end{proof}

Since $W$ is a right $A$-module, there is a $K$-algebra homomorphism $\delta:A\to \End_K(W)^{\op}$ defined by $\delta(a)(w)=w \cdot a$. Define $\bar A:=A/\ker(\delta)$. Let $~\bar{}:A\to \bar A$ be the canonical $K$-algebra homomorphism and $\bar \delta:\bar A\to \End_K(W)^{\op}$ the $K$-algebra homomorphism induced by $\delta$. Then the following diagram commutes.
\[\xymatrix{ A\ar[r]^(.35){\delta}\ar[d]_{\bar{}~}&\End_K(W)^{\op}\\ \bar A\ar[ur]_{\bar\delta}&}.\] 
$W$ becomes a right $\bar A$-module by defining $w\cdot \bar a=\bar\delta(\bar a)(w)=\delta(a)(w)=w\cdot a$.

\begin{proposition}\label{propKG} The algebra  $\bar A$ is isomorphic to the group algebra $K[G]$.
\end{proposition}
\begin{proof}
Let $a\in A$. Then $a=\sum_{i=1}^n k_ip_{i,1}\dots p_{i,m_i}$ for some $n, m_1,\dots, m_n\in \N$, $k_1,\dots, k_n\in K$ and $p_{i,j}\in G~(1\leq i\leq n, 1\leq j \leq m_i)$. Clearly $p_{i,1}\dots p_{i,m_i}-p_{i,1}\ast \dots \ast p_{i,m_i}\in \ker(\delta)$ for any $1\leq i\leq n$. Hence $a\equiv \sum_{i=1}^n k_ip_{i,1}\ast \dots \ast p_{i,m_i}\bmod \ker(\delta)$. We have shown that any element of $\bar A$ has a representative of the form $\sum_{p\in G}k_pp$.

Suppose that $\sum_{p\in G}k_pp\equiv \sum_{p\in G}l_pp \bmod \ker(\delta)$. Then $\delta(\sum_{p\in G}(k_p-l_p)p)=0$ and hence 
\[0=\delta(\sum_{p\in G}(k_p-l_p)p)(v_{\phi(u)})=v_{\phi(u)}\cdot (\sum_{p\in G}(k_p-l_p)p)=\sum_{p\in G}(k_p-l_p)v_{p}\]
by Lemma \ref{lemT0}. Since the $v_p$ are linearly independent, we obtain $k_p=l_p$ for any $p\in G$. Hence any element of $\bar A$ has precisely one representative of the form $\sum_{p\in G}k_pp$.

Define a map $\eta:\bar A\to K[G]$ by $\eta(\overline{\sum_{p\in G}k_pp})=\sum_{p\in G}k_pp$. By the previous two paragraphs $\eta$ is well-defined. Moreover, $\eta$ is clearly bijective. We leave it to the reader to check that $\eta$ is a $K$-algebra homomorphism.
\end{proof}

\begin{proposition}\label{propfree}
The $A$-module $W$ is free of rank $1$ as a $\bar A$-module.
\end{proposition}
\begin{proof}
Let $\bar a\in \bar A$. Then $\bar a=\overline{\sum_{p\in G}k_pp}$ for some $k_p\in K$ (that was shown in the proof of Proposition \ref{propKG}). Hence 
\[v_{\phi(u)}\cdot \bar a= v_{\phi(u)}\cdot\overline{\sum_{p\in G}k_pp} =v_{\phi(u)}\cdot\sum_{p\in G}k_pp =\sum_{p\in G}k_pv_p\]
by Lemma \ref{lemT0}. It follows that $\{v_{\phi(u)}\}$ is a basis for the $\bar A$-module $W$.
\end{proof}

We are in position to prove the main result of this section. 

\begin{theorem}\label{thmdecomp}
The $L_K(E)$-module $V_{T}$ defined by $(T,\xi)=(T_C,\xi_C)$ is indecomposable.
\end{theorem}
\begin{proof}
Let $\epsilon\in \End_L(V_T)$ be an idempotent endomorphism. It follows from Lemma \ref{lemmodule} that $\epsilon(W)=W$. Hence $\epsilon|_W\in \End_K(W)$. Clearly $\epsilon(\bar w\cdot a)=\epsilon(w\cdot a)=\epsilon(w)\cdot a=\epsilon(w)\cdot \bar a$ for any $a\in A$ and $w\in W$. Hence $\epsilon|_W\in \End_{\bar A}(W)$. By Proposition \ref{propKG} and \ref{propfree} we have $\End_{\bar A}(W)\cong \bar A\cong K[G]$. Since $G$ is free by Proposition \ref{propgroup}, the group ring $K[G]$ has no zero divisors by \cite[Theorem 12]{higman} (note that as explained in the paragraph before \cite[Theorem 12]{higman}, any free group is indicable throughout). It follows that $0$ and $1$ are the only idempotents of $K[G]$. Hence $\epsilon|_W=0$ or $\epsilon|_W=\id_W$.

Let $v_p\in T^0$. Then $\epsilon(v_p)=\epsilon( v_{\phi(u)}\cdot p)=\epsilon(v_{\phi(u)})\cdot p$ by Lemma \ref{lemT0}. Hence $\epsilon=0$ if $\epsilon|_W=0$ and $\epsilon=\id$ if $\epsilon|_W=\id_W$. Thus we have shown that $V_T$ is indecomposable. 
\end{proof}

\section{Irreducible representations of Leavitt path algebras}

Let $K$ be a field, $E$ a directed graph and $L_K(E)$ its associated Leavitt path algebra. For an infinite path $p$ in $E$, Xiao-Wu Chen constructed the left $L_K(E)$-module
$V_{[p]}$ of the space of infinite paths tail-equivalent to $p$ and proved that it is an irreducible representation
of the Leavitt path algebra (see~\cite[Theorem 3.3]{C}). A similar construction was also given for paths ending on a sink vertex.  In this section we recover these irreducible modules via the representation graphs of $E$. This gives another way to express these modules. Besides being more visual, this approach allows for carrying calculus on these modules with ease and produce indecomposable $L_K(E)$-modules via universal representations.

\subsection{Chen simple modules via representation graphs}

We briefly recall the construction of simple modules via infinite paths following the paper of Chen~\cite{C}. 

The set of all right-infinite paths in $E$ is denoted by $E^\infty$. If $p=p_1p_2\dots\in E^\infty$ and $n\geq 0$, then we set $\tau_{\leq n}(p):=p_{1}\dots p_{n}\in E^n$ and $\tau_{>n}(p):=p_{n+1}p_{n+2}\dots\in E^\infty$ (if $n=0$, then $\tau_{\leq n}(p)=s(p)$ and $\tau_{>n}(p)=p$). Two right-infinite paths $p,q\in E^\infty$ are called {\it tail-equivalent}, denoted by $p\sim q$, if there are $m,n\geq 0$ such that $\tau_{>m}(p)=\tau_{>n}(q)$. This defines an equivalence relation on $E^\infty$. We denote by $\widetilde {E^\infty}$ the set of
tail-equivalence classes, and for a path $p\in E^\infty$ denote the corresponding class by $[p]$.

A right-infinite path $p\in E^\infty$ is called {\it cyclic} if $p=ccc\dots$ where $c$ is a (finite) closed path. A $p\in E^\infty$ is called {\it rational} if $p$ is tail-equivalent to a cyclic path. Otherwise $p$ is called {\it irrational}. The class $[p]$ is called {\it rational} if $p$ is rational and {\it irrational} if $p$ is irrational.

If $[p]\in \widetilde {E^\infty}$, then the corresponding Chen module is the $K$-vector space $V_{[p]}$ with basis $[p]$. One can make the vector space $V_{[p]}$ a  right $L_K(E)$-module as follows. For any $v\in E^0$, $e\in E^1$ and $q\in [p]$ define
\begin{align*}
q\cdot v&=\begin{cases}q,\quad \quad \, \,   &\text{if }v=s(q) \\0,\quad&\text{else}
\end{cases}, \\
q\cdot e&=\begin{cases}\tau_{>1}(q),  \, \, &\text{if }e=\tau_{\leq 1}(q)\\0,\quad&\text{else}
\end{cases},\\
q\cdot e^*&=\begin{cases}eq,\quad \, \, \, \, \, &\text{if }r(e)=s(q)\\0,\quad&\text{else}
\end{cases}.
\end{align*}
The $K$-linear extension of this action to all of $V_{[p]}$ endows $V_{[p]}$ with the structure of a right $L_K(E)$-module. Chen has proven that the module 
 $V_{[p]}$  is simple and that $V_{[p]}\simeq V_{[q]}$ as right $L_K(E)$-modules if and only if $[p]=[q]$ \cite{C}.

Let $u\in E^0$ be a {\it sink}, i.e. a vertex that emits no edges. The corresponding Chen module is the $K$-vector space $\mathcal{N}_{u}$ with basis the set of finite paths ending in $u$. The $K$-vector space $\mathcal{N}_{u}$ becomes a right $L_K(E)$-module as follows. For a finite path $q=q_1\dots q_l\in E^l$ and $0\leq n \leq l$ we define $\tau_{\leq n}(q):=q_{1}\dots q_n\in E^n$ and $\tau_{>n}(q):=q_{n+1}\dots q_l\in E^{l-n}$ (if $n=0$, then $\tau_{\leq n}(q)=s(q)$ and if $n=l$, then $\tau_{>n}(q)=r(q)$). For a $v\in E^0$ set $\tau_{\leq 1}(v):=v$ and $\tau_{> 1}(v):=v$. For any $v\in E^0$, $e\in E^1$ and finite path $q$ ending in $u$ define $q\cdot v$, $q\cdot e$ and $q\cdot e^*$ as in the previous paragraph. The $K$-linear extension of this action to all of $\mathcal{N}_{u}$ endows $\mathcal{N}_{u}$ with the structure of a right $L_K(E)$-module. Chen has proven that the module $\mathcal{N}_{u}$ is simple and that $\mathcal{N}_{u}\simeq \mathcal{N}_{v}$ as right $L_K(E)$-modules if and only if $u=v$. Moreover, if $p\in E^\infty$ and $u$ is a sink, then $V_{[p]}\not\simeq \mathcal{N}_{u}$ as right $L_K(E)$-modules \cite{C}.

Translated to the unweighted setting the definition of a representation graph (Definition~\ref{defwp}) reduces to the following. 


\begin{definition}\label{defwp2}
Let $E$ be a graph. A {\it representation graph} for $E$ is a pair $(F,\phi)$, where $F=(F^0,F^1,s_F,r_F)$ is a directed graph and $\phi=(\phi^0,\phi^1):F\rightarrow E$ is a homomorphism of directed graphs such that the following hold:
\begin{enumerate}
\item For any $v\in F^0$ such that $\phi^0(v)$ is not a sink, we have $|s_F^{-1}(v)|=1$.
\smallskip
\item For any $v\in F^0$, the map $\phi^1: r^{-1}(v)\rightarrow r^{-1}(\phi^0(v))$ is bijective. 
\end{enumerate}
\end{definition}

\begin{example}\label{hfynvhfhraaaaa}
Let $E$ be a graph with one vertex and two loops: 
\[\xymatrix{ v\ar@(dr,ur)_{\mathtt{e}}\ar@(dl,ul)^{\mathtt{f}}}.\] 
Then the universal covering graph of $E$ is the Cayley graph of the free group with two generators $\mathbb F_2$. 


\end{example}

Let $(F,\phi)$ be a representation graph for $E$ and $V_F$ the $K$-vector space with basis $F^0$. For any $v\in E^0$, $e\in E^1$ and $u\in F^0$ define 
\begin{align*}
u\cdot v&=\begin{cases}u,\quad\quad \quad&\text{if }\phi^0(u)=v\\0,& \text{else} \end{cases},\\
u\cdot e&=\begin{cases}r_F(f),\quad&\text{if }\exists f\in s_F^{-1}(u):~\phi^1(f)=e\\0,& \text{else} \end{cases},\\
u\cdot e^*&=\begin{cases}s_F(f),\quad&\text{if }\exists f\in r_F^{-1}(u):~\phi^1(f)=e\\0,& \text{else}
 \end{cases}.
\end{align*}
The $K$-linear extension of this action to all of $V_F$ endows $V_F$ with the structure of a right $L_K(E)$-module. We call $V_F$ the {\it $L_K(E)$-module defined by $(F,\phi)$}.

We will construct for any Chen module $V_{[p]}$ (resp. $\mathcal{N}_{u}$) a representation graph $(F,\phi)$ such that the $L_K(E)$-module $V_F$ is isomorphic to $V_{[p]}$ (resp. $\mathcal{N}_{u}$). We divide this into three cases: 

\subsubsection{The case of $\mathcal{N}_{u}$, where $u$ is a sink}
Let $u\in E^0$ be a sink. We denote by $P$ the set of all nontrivial finite paths ending in $u$. Define a directed graph $F$ by 
\begin{align*}
F^0&=\{v\}\sqcup \{v_{p}\mid p\in P \},\\
F^1&=\{f_{p}\mid p\in P\},\\
s_F(f_{p})&=v_{p},\quad r_F(f_{p})=\begin{cases}v_{\tau_{>1}(p)},\quad&\text{if }|p|\geq 2,\\v,\quad&\text{if }|p|=1.\end{cases}
\end{align*}
Define $\phi^0(v)=u$, $\phi^0(v_{p})=s(p)$ and $\phi^1(f_{p})=\tau_{\leq 1}(p)$. One checks easily that $\phi=(\phi^0,\phi^1):F\to E$ is a homomorphism of directed graphs and that Conditions (1) and (2) in Definition \ref{defwp2} are satisfied. Hence $(F,\phi)$ is a representation graph for $E$.

Define a map $\alpha:F^0\to P\cup\{u\}$ by $\alpha(v)=u~\text{  and }~\alpha(v_{p})=p$.
Obviously $\alpha$ is a bijection and hence it induces an isomorphism $\hat\alpha:V_F\rightarrow \mathcal{N}_{u}$ of $K$-vector spaces. One checks easily that $\hat\alpha$ is an isomorphism of $L_K(E)$-modules.

\subsubsection{The case of $V_{[p]}$, where  $[p]$ is irrational} \label{irrationalcase}
Let $[p]\in \widetilde {E^\infty}$ be an irrational class. Then 
\begin{equation}
\label{eq:taumn:1}
\tau_{>m}(p)\neq\tau_{>n}(p)\text{ for any distinct }m,n\geq 0.
\end{equation}

Write $p=p_1p_2p_3\dots$. For any $i\in\N$ we denote by $P_i$ the set of all nontrivial finite paths $q$ such that $r(q)=s(p_i)$ and the last letter of $q$ is not equal to $p_{i-1}$ if $i\geq 2$. Define $F$ by 
\begin{align*}
F^0&=\{v_i\mid i\in \N\}\sqcup \{v_{i,q}\mid i\in \N, q\in P_i\},\\
F^1&=\{f_i\mid i\in \N\}\sqcup \{f_{i,q}\mid i\in \N, q\in P_i\},\\
s_F(f_i)&=v_i,\quad r_F(f_i)=v_{i+1},\\
s_F(f_{i,q})&=v_{i,q},\quad r_F(f_{i,q})=\begin{cases}v_{i,\tau_{>1}(q)},\quad&\text{if }|q|\geq 2,\\v_{i},\quad&\text{if }|q|=1.\end{cases}
\end{align*}

Define $\phi^0(v_i)=s(p_i)$, $\phi^0(v_{i,q})=s(q)$, $\phi^1(f_i)=p_i$ and $\phi^1(f_{i,q})=\tau_{\leq 1}(q)$. One checks easily that $\phi=(\phi^0,\phi^1):F\to E$ is a homomorphism of directed graphs and that Conditions (1) and (2) in Definition \ref{defwp2} are satisfied. Hence $(F,\phi)$ is a representation graph for $E$.

\begin{example}
For the graph $E$ of (\ref{lustigdog}) and the infinite irrational path $p=efef^2ef^3\cdots$, the above construction gives the representation graph (\ref{wlpa45}). 
\end{example}

Define a map $\beta:F^0\to [p]$ by 
\[\beta(v_i)=p_ip_{i+1}\dots~\text{  and }~\beta(v_{i,q})=qp_ip_{i+1}\dots.\]

\begin{lemma}\label{lembet}
The map $\beta$ defined above is bijective.
\end{lemma} 
\begin{proof}
(Surjectivity) Let $p'\in [p]$. Let $k\geq 0$ be minimal such that $\tau_{>k}(p')=\tau_{>l}(p)$ for some $l\geq 0$. If $k=0$, then $p'=p_ip_{i+1}\dots$ for some $i$. Hence $p'=\beta(v_i)$. If $k>0$, then $p'=p'_1\dots p'_kp_ip_{i+1}\dots$ for some $i$. Clearly $p'_k\neq p_{i-1}$ because of the minimality of $k$. Hence $q:=p'_1\dots p'_k\in P_i$ and $p'=\beta(v_{i,q})$. 

(Injectivity) Because of \eqref{eq:taumn:1}, it cannot happen that $\beta(v_i)=\beta(v_j)$ for some $i\neq j\in\N$. Assume now that $\beta(v_{i,q})=\beta(v_j)$ where $1\leq i,j\leq n$ and $q\in P_i$. Then clearly 
\[p_ip_{i+1}\dots=\tau_{>|q|}(\beta(v_{i,q}))=\tau_{>|q|}(\beta(v_j))=p_{r}p_{r+1}\dots\]
for some $r\in\N$. It follows from \eqref{eq:taumn:1} that $r=i$. Hence
\[q_{|q|}p_ip_{i+1}\dots=\tau_{>|q|-1}(\beta(v_{i,q}))=\tau_{>|q|-1}(\beta(v_j))=p_{i-1}p_ip_{i+1}\dots\]
and hence $q_{|q|}=p_{i-1}$, which contradicts the assumption that $q\in P_i$.

Assume now that $\beta(v_{i,q})=\beta(v_{j,q'})$ where $1\leq i,j\leq n$, $q\in P_i$ and $q'\in P_j$. If $|q|\neq |q'|$, then we obtain a contradiction as in the previous paragraph (say $|q|<|q'|$; then consider $\tau_{>|q|}(\beta(v_{i,q}))=\tau_{>|q|}(\beta(v_{j,q'})$). Suppose now that $|q|=|q'|$. It follows that $q=q'$ since $\beta(v_{i,q})=\beta(v_{j,q'})$. Moreover, \eqref{eq:taumn:1} implies that $i=j$ as desired.
\end{proof}

Since $\beta$ is a bijection, it induces an isomorphism $\hat\beta:V_F\rightarrow V_{[p]}$ of $K$-vector spaces. One checks easily that $\hat\beta$ is an isomorphism of $L_K(E)$-modules.

\subsubsection{The case of $V_{[p]}$, where  $[p]$ is rational}\label{rationalcase}
Let $[p]\in \widetilde {E^\infty}$ be a rational class. Then we may assume that $p$ is cyclic, i.e. $p=ccc\dots$, where $c$ is a closed path. We also may assume that $c$ is {\it simple}, i.e. $c$ is not a power of a shorter closed path. 

Suppose that $c=c_1\dots c_n$. For any $1\leq i\leq n$ we denote by $P_i$ the set of all nontrivial finite paths $q$ such that $r(q)=s(c_i)$ and the last letter of $q$ is not equal to $c_{i-1}$, respectively $c_n$ if $i=1$. Define $F$ by 
\begin{align*}
F^0&=\{v_i\mid 1\leq i\leq n\}\sqcup \{v_{i,q}\mid 1\leq i\leq n, q\in P_i\},\\
F^1&=\{f_i\mid 1\leq i\leq n\}\sqcup\{f_{i,q}\mid 1\leq i\leq n, q\in P_i\},\\
s_F(f_i)&=v_i,\quad r_F(f_i)=v_{i+1},\\
s_F(f_{i,q})&=v_{i,q},\quad r_F(f_{i,q})=\begin{cases}v_{i,\tau_{>1}(q)},\quad&\text{if }|q|\geq 2,\\v_{i},\quad&\text{if }|q|=1.\end{cases}
\end{align*}
Here we use the convention $n+1=1$. Define $\phi^0(v_i)=s(c_i)$, $\phi^0(v_{i,q})=s(q)$, $\phi^1(f_i)=c_i$ and $\phi^1(f_{i,q})=\tau_{\leq 1}(q)$. One checks easily that $\phi=(\phi^0,\phi^1):F\to E$ is a homomorphism of directed graphs and that Conditions (1) and (2) in Definition \ref{defwp2} are satisfied. Hence $(F,\phi)$ is a representation graph for $E$.

\begin{example}
For the graph $E$ of (\ref{lustigdog}) and the infinite rational path $p=efgefg\cdots$, the above construction gives the representation graph (\ref{wlpa4}). 
\end{example}

Define a map $\gamma:F^0\to [p]$ by 
\[\gamma(v_i)=c_i\dots c_nccc\dots~\text{  and }~\gamma(v_{i,q})=qc_i\dots c_nccc\dots.\]

\begin{lemma}\label{lemword}
Let $x_1x_2x_3\dots$ be a right-infinite word over some alphabet. Let $n\in \N$ such that $x_{r}=x_{r+n}$ for any $r\in \N$. If there is an $n'\in \N$ such that $n'<n$ and $x_{r}=x_{r+n'}$ for any $r\in \N$, then there is an $m\in \N$ such that $m<n$, $m|n$ and $x_{r}=x_{r+m}$ for any $r\in \N$.
\end{lemma}
\begin{proof}
Set $A:=\{n'\in \{1,\dots,n-1\}\mid x_{r}=x_{r+n'}\text{ for any }r\in \N\}$ and $m:=\min(A)$. Assume that $m$ does not divide $n$. Then there are $s, t\in \N$ such that $t<m$ and $n=sm+t$. Clearly $x_r=x_{r+n}=x_{r+sm+t}=x_{r+t}$ for any $r\in \N$. Hence $t\in A$, which contradicts the minimality of $m$. Thus $m|n$.
\end{proof}

\begin{lemma}\label{lemgam}
The map $\gamma$ defined above is bijective.
\end{lemma} 
\begin{proof}
(Surjectivity) Let $p'\in [p]$. Let $k\geq 0$ be minimal such that $\tau_{>k}(p')=\tau_{>l}(p)$ for some $l\geq 0$. If $k=0$, then $p'=c_i\dots c_nccc\dots$ for some $i$. Hence $p'=\gamma(v_i)$. If $k>0$, then $p'=p'_1\dots p'_kc_i\dots c_nccc\dots$ for some $i$. Clearly $p'_k\neq c_{i-1}$ because of the minimality of $k$. Hence $q:=p'_1\dots p'_k\in P_i$ and $p'=\gamma(v_{i,q})$. 

(Injectivity) Assume $\gamma(v_i)=\gamma(v_j)$ for some $1\leq i<j\leq n$. Write $\gamma(v_i)=x_1x_2x_3\dots$ and $\gamma(v_j)=y_1y_2y_3\dots$. Clearly 
\begin{equation*}
x_{r}=x_{r+n}\text{ for any }r\in \N.
\end{equation*}
Moreover, $y_r=x_{r+j-i}$ and hence 
\begin{equation*}
x_r=x_{r+j-i}\text{ for any }r\in \N
\end{equation*}
since by assumption $x_r=y_r$. Obviously $1\leq j-i<n$. It follows from Lemma \ref{lemword} that there is an $m\in \N$ such that $m<n$, $m|n$ and $x_{r}=x_{r+m}$ for any $r\in \N$. Since $\gamma(v_i)=c_i\dots c_nc_1\dots c_nc_1\dots c_n\dots$ we obtain $c=dd\dots d$ where $d=c_1\dots c_m$. But this contradicts the simplicity of $c$.

Assume now that $\gamma(v_{i,q})=\gamma(v_j)$ where $1\leq i,j\leq n$ and $q\in P_i$. Then clearly 
\[c_i\dots c_nccc\dots=\tau_{>|q|}(\gamma(v_{i,q}))=\tau_{>|q|}(\gamma(v_j))=c_r\dots c_nccc\dots\]
for some $1\leq r\leq n$. If $r\neq i$, then we obtain a contradiction as in the previous paragraph. Suppose now that $r=i$. Then
\[q_{|q|}c_i\dots c_nccc\dots=\tau_{>|q|-1}(\gamma(v_{i,q}))=\tau_{>|q|-1}(\gamma(v_j))=c_{i-1}c_i\dots c_nccc\dots\]
and hence $q_{|q|}=c_{i-1}$, which contradicts the assumption that $q\in P_i$.

Assume now that $\beta(v_{i,q})=\beta(v_{j,q'})$ where $1\leq i,j\leq n$, $q\in P_i$ and $q'\in P_j$. If $|q|\neq |q'|$, then we obtain a contradiction as in the previous paragraph (say $|q|<|q'|$; then consider $\tau_{>|q|}(\gamma(v_{i,q}))=\tau_{>|q|}(\gamma(v_{j,q'})$). Suppose now that $|q|=|q'|$. It follows that $q=q'$ since $\gamma(v_{i,q})=\gamma(v_{j,q'})$. If $i\neq j $, then we obtain a contradiction as in the paragraph before the previous one. Thus we also get $i=j$ as desired.
\end{proof}

Since $\gamma$ is a bijection, it induces an isomorphism $\hat\gamma:V_F\rightarrow V_{[p]}$ of $K$-vector spaces. One checks easily that $\hat\gamma$ is an isomorphism of $L_K(E)$-modules.

We can recover the properties of Chen simple modules via our machinery of  representation graphs. 

\begin{theorem}
Let $E$ be a graph. Let $p$ and $q$ be infinite paths. Then

\begin{enumerate}[\upshape(i)]

\item The Chen module $V_{[p]}$ is a simple $L_K(E)$-module. 

\smallskip 

\item The $L_K(E)$-modules $V_{[p]}$ and $V_{[q]}$ are isomorphic if and only if $[p]=[q]$. 

\smallskip 

\item The $L_K(E)$-module $V_{[p]}$ is graded if and only if $p$ is an irrational infinite path. 
\end{enumerate}
\end{theorem}
\begin{proof}
(i) Suppose first that $p$ is irrational. Let $F$ be the representation graph constructed in (\ref{irrationalcase}). Then $V_F\cong V_{[p]}$. Clearly $F$ is connected. Moreover, $\phi({}_{v}\!\Path(F))=\{\tau_{\leq n}(\beta(v))\mid n\geq 0\}$ for any $v\in F^0$. It follows that $\phi({}_{u}\!\Path(F))\neq \phi({}_{v}\!\Path(F))$ for any $u\neq v\in F^0$ since the map $\beta:F^0\to [p]$ is injective by Lemma \ref{lembet}. Clearly this implies that $\phi({}_{u}\!\Path(F_d))\neq \phi({}_{v}\!\Path(F_d))$ for any $u\neq v\in F^0$ and hence $F$ satisfies Condition (2) in Definition \ref{defrepirres}. Thus $V_F\cong V_{[p]}$ is simple by Theorem \ref{thmirr}. The case that $p$ is rational is similar (just replace $\beta$ by $\gamma$).

(ii) Let $(F,\phi)$ and $(G,\psi)$ be the representation graphs constructed in (\ref{irrationalcase}) respectively (\ref{rationalcase}) such that $V_F\cong V_{[p]}$ and $V_G\cong V_{[q]}$. It suffices to show that $V_F\cong V_G$ implies $[p]=[q]$. So suppose that $V_F\cong V_G$. Then by Proposition \ref{prop42} we have $(F,\phi)\leftrightharpoons (G,\psi)$. Hence there are vertices $u\in F^0$ and $v\in G^0$ such that $\phi({}_{u}\!\Path(F_d))= \psi({}_{v}\!\Path(G_d))$. Clearly this implies $\phi({}_{u}\!\Path(F))= \psi({}_{v}\!\Path(G))$. But $\phi({}_{u}\!\Path(F))=\{\tau_{\leq n}(p')\mid n\geq 0\}$ for some $p'\in [p]$ and $\phi({}_{v}\!\Path(G))=\{\tau_{\leq n}(q')\mid n\geq 0\}$ for some $q'\in [q]$ (compare the previous paragraph). It follows that $p'=q'$ and thus $[p]=[p']=[q']=[q]$.

(iii) Note that by the constructions in (\ref{irrationalcase}) and (\ref{rationalcase}),  if the infinite path $p$ is irrational, then the representation graph $F$ is a tree, whereas if $p$ is rational then the representation graph $F$ has a cycle.   Now the statement follows immediately from Theorem~\ref{gradedrepres} and the fact that $V_F \cong V_{[p]}$ as $L_K(E)$-modules. 
\end{proof}

\section{Branching systems of weighted graphs and representations} \label{branchhonda}

Branching systems for Leavitt path algebras were systematically studied by Gon\c{c}alves and Royer (\cite{GR,GR-0,GR-1}). A branching system of a graph gives rise to a representation for its associated Leavitt path algebra. This notion was also generalised to other type of graphs, such as separated graphs and ultragraphs (\cite{GR-4}). 
Here we adopt this notion for the weighted graphs and show that a representation graph for a weighted graph gives a branching system for this graph. We will show that under certain assumptions on the field
each representation defined by a branching system contains a subrepresentation isomorphic to one
given by a representation graph. Conversely, every representation graph defines a branching system
that defines the same (isomorphic) representation as the graph itself.

\begin{definition}[Branching system of a weighted graph]
    Let $(E,w)$ be a weighted graph. Let $X$ be a set, $\{R_{e_i} \mid e_i \in \hat{E}^1 \}$ and $\{ D_{v} \mid v \in E^0 \}$ families of subsets of $X$ and $\{ g_{e_i} \colon R_{e_i} \hookrightarrow D_{r(e)}
    \mid e_i \in \hat{E}^1 \}$ a family of injections such that:
    \begin{enumerate}
    	\item $\{ D_{v} \mid v \in E^0 \}$ is a partition of $X$ (i.e. $D_v \cap D_u = \emptyset$ whenever
    		$v \neq u$ and $\bigcup_{v \in E^0}\limits D_v = X$);
    	\item for each $v \in E^0$ and $1 \leq i \leq w(v)$ the family of sets 
    		$\{ R_{e_i} \mid e \in s^{-1}(v), w(e) \geq i \}$ forms a partition of $D_v$;
    	\item Set $D_{e_i} = g_{e_i}(R_{e_i})$ for each $e_i \in \hat{E}^1$. 
    		Then for each $e \in E^1$ the family $\{ D_{e_i} \mid 1 \leq i \leq w(e) \}$ forms a 
    		partition of $D_{r(e)}$.
    \end{enumerate}
    We call the quadruple $X=(X, \{R_{e_i} \}, \{ D_{v} \}, \{g_{e_i} \})$ an {\it $(E,w)$-branching system} 
    or simply an {\it $E$-branching system}.
\end{definition}

When $w(E)=1$ the definition above reduces to the definition of an $E$-algebraic branching system of \cite{GR} with the only twist that the bijections $f_{e}$ thereof are called $g^{-1}_{e_1}$ here.
Note that we do not assume by default that the sets comprising a branching system are nonempty.

Let $M$ be the $K$-module of all functions $X \rightarrow K$ with respect to point-wise operations. 
Let $M_0$ denote the $K$-submodule of $M$ that consists of all functions with finite support.
We are going to define a structure of a right $L_K(E)$-module on $M$. 
In order to simplify notations, we will abuse the notation as follows. Let $Z \subseteq Y \subseteq X$ 
and $\psi : Y \rightarrow K$. By $\chi_Z \cdot \psi$ we denote the function $X \rightarrow K$
\[
	x \mapsto \begin{cases}
		\psi(x) & \text{ if } x \in Z \\
		0 & \text{ otherwise }.
	\end{cases}
\]
Using this convention, set for any $\phi \in M$, any $e_i \in \hat{E}^1$ and any $v \in E^0$
\begin{equation}
\label{eq:brs:action}
\begin{aligned}
    \phi.e_i &= \chi_{D_{e_i}} \cdot (\phi \circ g^{-1}_{e_i}) \\
    \phi.e_i^* &= \chi_{R_{e_i}} \cdot (\phi \circ g_{e_i}) \\
    \phi.v &= \chi_{D_v} \cdot \phi.
\end{aligned}
\end{equation}

We will show that the multiplication defined above may be extended by linearity to the
structure of a right $L_K(E)$-module on $M$. Prior to that we would like to stress that
\begin{itemize}
	\item by switching the action of $e_i$ and $e_i^*$ on $M$ we get a left $L_K(E)$-module,
		in case of an unweighted		
		graph $E$ this module coincides with the one defined in \cite{GR};
	\item the same action of $L_K(E)$ may be defined on the $K$-module $M_0$ of all
		functions in $M$ with \textit{finite support}. The proof of the next theorem
		remains valid if $M$ is replaced by $M_0$.
	\item the notion of a branching system generalises easily to the case of a not necessarily
		row-finite graph $(E,w)$. In this case one should simply drop the assumption
		that the family of sets $\{ R_{e_i} \mid e \in s^{-1}(v), w(e) \geq i \}$ covers $D_v$ 
		whenever this family is infinite. However one must still assume that this family
		is disjoint.		
\end{itemize}
 
\begin{theorem}
    \label{t.br1}
    Let $(E,w)$ be a weighted graph and $X$ an $(E,w)$-branching system. Then 
    the equalities \eqref{eq:brs:action} define a structure of a 
    right $L_K(E)$-module on $M$.
    \begin{proof}
        Due to the universal nature of $L_K(E)$ we only need to check that the
        standard generators $e_i, e_i^*$ and $v$ of $L_K(E)$ treated as endomorphisms
        of $M$ satisfy the relations (1)--(4) in Definition \ref{def3}.
        
        (i) Clearly $v$ is an idempotent for each $v \in E^0$; and for each $v \neq u \in E^0, \phi \in M$
        one has $\phi.v.u = 0$ as $D_v$ and $D_u$ are disjoint.
        
        (ii) Note that 
        \[
	        \phi.s(e).e_i
	        	= \chi_{D_{e_i}} \cdot ((\chi_{D_{s(e)}} \cdot \phi) \circ g^{-1}_{e_i})
	        	= \chi_{D_{e_i}} \cdot (\phi \circ g^{-1}_{e_i}) = \phi.e_i,
	    \]
	    where the second equality follows from the fact that $g^{-1}_{e_i} (D_{e_i}) =
	    R_{e_i} \subseteq D_{s(e)}$.
		Similarly,
		\begin{align*}
			\phi.e_i.r(e) &= \chi_{D_{r(e)}} \cdot (\chi_{D_{e_i}} \cdot (\phi \circ g^{-1}_{e_i}))
			= \chi_{D_{e_i}} \cdot (\phi \circ g^{-1}_{e_i}) = \phi.e_i,
		\end{align*}			    
		as $D_{e_i} \subseteq D_{r(e)}$. In the same way,
		\begin{align*}
			\phi.e_i^*.s(e) &= \chi_{D_{s(e)}} \cdot (\chi_{R_{e_i}} \cdot (\phi \circ g_{e_i}))
			= \chi_{R_{e_i}} \cdot (\phi \circ g_{e_i}) = \phi.e_i^*,
		\end{align*}			    
		for $R_{e_i} \subseteq D_{s(e)}$ and
		\begin{align*}
			\phi.r(e). e_i^* &= 
			\chi_{R_{e_i}} \cdot ((\chi_{D_{r(e)}} \cdot \phi) \circ g_{e_i})
			= \chi_{R_{e_i}} \cdot (\phi \circ g_{e_i}) = \phi.e_i^*,
		\end{align*}			    
		as $g_{e_i}(R_{e_i}) = D_{e_i} \subseteq D_{r(e)}$.
		
        (iii) Now let $e,f \in E^1$ such that $s(e) = s(f) = v \in E^0$. Then
        \begin{equation}
        	\label{eq:BR:t.br1:1}
        	\sum_{1 \le i \le w(v)} \phi.e_i^*.f_i
        	= \sum_{1 \le i \le \min(w(e),w(f))} \chi_{D_{f_i}} \cdot ((\chi_{R_{e_i}} \cdot \phi \circ g_{e_i})
        	\circ g^{-1}_{f_i}).
        \end{equation}
        Recall that $e_i^* = e_i = 0$ whenever $i > w(e)$.
        Note that if $f \neq e$ the image $g^{-1}_{f_i}(D_{f_i}) = R_{f_i}$ is disjoint with 
        $R_{e_i}$ and thus each summand in \eqref{eq:BR:t.br1:1} equals zero. If
        $f=e$ then we get 
        \[
        	(\chi_{R_{e_i}} \cdot \phi \circ g_{e_i}) \circ g^{-1}_{f_i} 
        = \phi 
        \] when restricted to $D_{e_i}$. This allows to rewrite the 
       	right-hand side of \eqref{eq:BR:t.br1:1} as
       	\[
       		\sum_{1 \le i \le w(e)} \chi_{D_{e_i}} \cdot \phi = \chi_{D_{r(e)}} \cdot \phi = \phi.r(e),
       	\]
       	where the first equality follows from the fact $\{ D_{e_i} \mid 1 \le i \le w(e)\}$ 
       	is a partition of $D_{r(e)}$.
       	
       	(iv) Finally, let $v \in E^0$ and $1 \leq i,j \le w(v)$. 
       	Then
       	\begin{equation}
	       	\label{eq:BR:t.br1:2}
       		\sum_{e \in s^{-1}(v)} \phi.e_i. e_j^* = 
       		\sum_{e \in s^{-1}(v)} 
       		\chi_{R_{e_j}} \cdot (\chi_{D_{e_i}} \cdot \phi \circ g^{-1}_{e_i}) \circ g_{e_j}.
       	\end{equation}
       	If $i \neq j$ then $g_{e_j}(R_{e_j}) = D_{e_j}$ which is disjoint with $D_{e_i}$ and each
       	summand above equals zero. Assume $j=i$. As $g_{e_i}(R_{e_i}) = D_{e_i}$, we get
       	\[
       		(\chi_{D_{e_j}} \cdot \phi \circ g^{-1}_{e_i}) \circ g_{e_j} = 
       		 \phi
       	\]
       	when restricted to $R_{e_i}$. Then the right-hand side of \eqref{eq:BR:t.br1:2} rewrites as 
       	\[
       		\sum_{e \in s^{-1}(v)} 
       		\chi_{R_{e_i}} \cdot \phi = \chi_{D_{v}} \cdot \phi = \phi.v,
       	\]
       	as $\{ R_{e_i} \mid e \in s^{-1}(v)\}$ partition $D_{s(e)}$.
       	Note that this computation makes sense even if the number of summands in
       	\eqref{eq:BR:t.br1:2} is infinite, as the supports of the summands are disjoint. However,
       	this is unnecessary as the corresponding relation in a weighted Leavitt algebra does not 
       	have to hold.
    \end{proof}
\end{theorem}

\medskip

\subsection{Branching systems on an interval}

The first series of examples of branching system is a generalization of the one given in 
\cite[Theorem 3.1]{GR}. However, in the 
weighted case it becomes clear that the distinction between sinks and not sinks made there
is redundant, as not only \textit{source sets}, but also \textit{range sets} must be partitioned.
Thus, even in the case of weight 1, we don't get \textit{precisely} the same branching 
system as in \cite{GR}, but \textit{morally} the same (in particular, the resulting classes
of representations coincide).

Let $(E,w)$ be a \textit{at most countable} weighted graph, i.e. $E^0$ and $E^1$ are both 
finite or countable. By fixing some linear order on $E^0$ we may write
$E^0 = \{ v^1, v^2, \dots \}$. For each
$i$ set $D_{v^i} = [i-1,i)$. Clearly, such sets are disjoint. Put $X = \bigcup_i D_{v^i}$.

Now fix a vertex $v^i \in E^0$ and $1 \le j \le w(v^i)$.
The set $X^i_j = \{ e \in s^{-1}(v^i) \mid w(e) \geq j \}$ is finite (as $E$ is row-finite). 
By ordering this set we can rewrite it as
$X^i_j = \{e^{i,1}, e^{i,2}, \dots \}$. For each $1\leq k \leq |X^i_j|$ set
\[
	R_{e^{i,k}_j} = [i-1+\frac{k-1}{|X^i_j|},i-1+\frac{k}{|X^i_j|}).
\]
It is clear that the set $\{ R_{e_j} \mid e \in X^i_j \}$ forms a partition of $D_{v^i}$.

In a similar fashion fix some $e \in E^1$ and let $v^i = r(e)$. For each $1 \leq j \leq w(e)$
set
\[
	D_{e_j} = [i-1+\frac{j-1}{w(e)}, i-1+\frac{j}{w(e)}).
\]
Clearly, the family of sets $\{ D_{e_j} \mid 1 \le j \le w(e) \}$ forms a partition of $D_{v^i}$.

Finally, the bijections $g_{e^i_j} \colon R_{e^i_j} \rightarrow D_{e^i_j}$ may be
chosen arbitrary, for example, a composition of a translation, scaling and another translation.
The following theorem is now obvious.

\begin{theorem}
	The sets defined above form an $(E,w)$-branching system.
\end{theorem}

\subsection{Branching systems and representation graphs}

Now we will show that every representation graph defines a branching system is a natural
way and representations induced by these structures are isomorphic in a reasonable sense.

Let $(E,w)$ be a weighted graph and $(F, \phi)$ a representation graph for $(E,w)$. 
Put $X = F^0$ and $D_v = (\phi^0)^{-1}(v)$ for each $v \in E^0$. Clearly,
$\{ D_v \mid v \in E^0\}$ is a partition of $X$. 

Fix $v \in E^0$ and a tag $1\leq i\leq w(v)$. For each $e \in s^{-1}(v)$ such that $i \le w(e)$ set
\[R_{e_i} = \{ u \in D_v \mid \text{ there exists } f \in s_F^{-1}(u) : \phi^1(f)=e_i \}.\]
The Condition (1) of Definition \ref{defwp} of a representation 
graph translates in this setting as follows:
each vertex $u$ in $D_v$ is contained in one and only one $R_{e_i}$, where
$e$ ranges over all edges of weight at least $i$ emitted by $v$. In other words,
the family $\{ R_{e_i} \mid e \in s^{-1}(v), w(e) \geq i \}$ forms a
partition for $D_v$. 

Now fix a vertex $v \in F^0$ and an edge $e \in r^{-1}(v)$. For each $1 \leq i \leq w(v)$
set 
\[D_{e_i} = \{ u \in D_v \mid \text{ there exists } f \in r_F^{-1}(u) \colon
\phi^1(f)=e_i\}.\]
The Condition (2) of Definition \ref{defwp} of a representation 
graph guarantees that the family of sets 
$\{ D_{e_i} \mid 1 \leq i \leq w(E) \}$ forms a partition for $D_v$.

Finally, fix some $e_i \in \hat{E}^1$. We need to define a bijection
$g_{e_i} \colon R_{e_i} \rightarrow D_{e_i}$. For each $u \in R_{e_i}$
there exists precisely one edge $f \in s^{-1}(u) \cap (\phi^1)^{-1}(e_i)$.
Set $g_{e_i}(u) = r(f)$. 

\begin{theorem}
	The quadruple $X = (X, \{ R_{e_i} \}, \{ D_{e_i} \}, \{ g_{e_i} \})$ defined above
	is a $(E,w)$-branching system. The right $L_K(E,w)$-module induced by $X$
	on $M_0$ is isomorphic to the module defined by $(F,\phi)$.
	\begin{proof}
		It is clear that $X$ is a branching system. Let $\pi$ denote the 
		representation on $K^{F^0}$ induced by the representation graph.
		The module $M_0$ has a $K$-basis of delta-functions:
		\[
			\delta_x \colon y \mapsto \begin{cases}
				1 & \text{ if } y=x \\
				0 & \text{ otherwise},
			\end{cases}
		\]
		which correspond bijectively to elements of $X$. For the sake of simplicity,
		we will identify x with $\delta_x$. In order to show that
		the module structure induced on $M_0$ by $X$ is isomorphic
		to the module structure defined by $\pi$ it is enough to show that
		\begin{align*}
			\pi(v)(x) &= x.v \\
			\pi(e_i)(x) &= x.e_i \text{ and }\\
			\pi(e^*_i)(x) &= x.e_i^*,
		\end{align*}
		for each $x \in X = F^0$, each $v \in E^0$ and each $e_i \in \hat{E}^1$.
		
		By definition of $\pi$, $\pi(v)(x) = x$ iff $x \in D_v$ and $0$ otherwise.
		On the other hand, $x.v = \chi_{D_v} \cdot \delta_x = \delta_x = x$ iff $x \in D_v$
		and 0 otherwise.
		
		Further, let $e_i \in \hat{E}^1$. Then $\pi(e_i)(x) = r(f)$, if there exists a
		(unique!) edge $f$ in $s^{-1}(x)$ such that $\phi^1(f) = e_i$, and 0 otherwise.
		Similarly,
		$x.e_i = \chi_{D_{e_i}} \cdot (\delta_x \circ g^{-1}_{e_i})$.
		Note that the restriction of $\delta_x \circ g^{-1}_{e_i}$ to $D_{e_i}$ coincides
		with that of $\delta_{g_{e_i}(x)}$. If $x \in R_{e_i}$ then
		there exits a (again, unique) edge $f$ mapped to $e_i$ under $\phi^1$
		and $g_{e_i}$ takes $x$ to $r(f) \in D_{e_i}$ Therefore $x.e_i = r(f)$.
		If an edge $f$ as above does not exist, then $x$ emits a \textit{different}
		edge with tag $i$ and ($x \notin R_{e_i}$). Therefore the supports of
		$\chi_{D_{e_i}}$ and $\delta_x$ are disjoint and $x.e_i = 0$.
		
		In the same manner one checks that $\pi(e_i^*)(x)=x.e_i^*$ for each $x \in X$.
	\end{proof}
\end{theorem}

We will show next that given a $(E,w)$-branching system $X$ it is possible do describe a 
representation graph that defines the same representation as $X$ induces on $M_0$.

\begin{theorem}
	\label{t.easyisrg}
	Let $(E,w)$ be a weighted graph and $V$ be a vector space over $K$ 
	equipped with the structure of a right $L_K(E)$-module.
	Suppose there exists a $K$-basis $B$ of $V$ such that 
	\begin{enumerate}[\upshape(i)]
		\item for each $a \in B$ and each $e \in \hat{E}_d^1$ holds $a.e \in B \cup \{ 0\}$;
		\smallskip
		\item for each $a \in B$ and each $v \in E^0$ holds $a.v \in \{a, 0\}$;
		\smallskip
		\item for each $a \in B$ $a. L_K(E) \neq \{ 0\}$;
		\smallskip
		\item either $\operatorname{char}(K) = 0$ or for each $a \in B$ holds:
			$a.e_i^*f_i = 0$ for each $e \neq f$ and $a.e_ie_j^* = 0$ for each $i \neq j$.
	\end{enumerate}
	Then there exists a representation graph that defines an $L_K(E)$-module isomorphic to $V$.
	\begin{proof}
		First, we will show that each $a \in B$ is fixed by precisely 
		one $v \in E^0$ and annihilated by the rest. We will refer to this
		as \textit{v-property}. By the assumption (iii),
		there exists an element $x \in L_K(E)$ such that 
		$a.x \neq 0$. Without loss of generality we assume that
		$x$ is a standard generator of $L_K(E)$, 
		i.e $x \in E^0 \cup \hat{E}_d^1$. Suppose $x$ is an edge.
		By the property (2) in Definition \ref{def3}
		one has $0 \neq a.x = a.s(x)x$. Therefore $a.s(x) \neq 0$. Similarly, if
		$x$ is a ghost edge then $0 \neq a.x = a.r(x)x$. Therefore $a.r(x) \neq 0$.
		Summing up, $a.v \neq 0$ for some $v \in E^0$. By assumption (i), $a.v=a$.
		Suppose $u \neq v \in E^0$. By the property (1) in Definition \ref{def3},
		$a.u=a.vu=a.0=0$. Therefore each vertex except for $v$ annihilates $a$.
		With this in mind set $F^0 = B$ and for each $a \in B$ set $\phi^0(a) = v$,
		where $v$ is the only vertex in $E^0$ such that $a.v\neq 0$.
		
		Note that we have proven a useful fact on the way: 
		if $a.e = b$ for some edge $e$ and $a,b \in B$ then $a.s(e) =a$ and $b.r(e) = a.er(e) = a.e=b$. 
		Similarly, if $c.e^*=d$ for some ghost edge $e^*$ and $c,d \in B$ then
		$c.r(e) = c$ and $d.s(e)=d$. We shall refer to this as the \textit{sr-property}.

		Next, we construct $F^1$ and $\phi^1$. Start with $F^1=\emptyset$. For each
		$a \in B$ and $e_i \in \hat{E}^1$ such that $a.e_i \neq 0$ add an edge
		$g_{a,e_i}$ from $a$ to $a.e_i$ (cf. assumption (i)) to $F^1$ and set
		$\phi^1(g_{a,e_i}) = e_i$. For convenience, also set $\tg(g_{a,e_i})=i$,
		$\mot(g_{a,e_i}) = e$ and $w(a) = w(\phi^0(a))$. 
		Therefore, $\tg(g) = \tg(\phi^1(g))$ and $\mot(g)=\mot(\phi^1(g))$ for each $g \in F^1$. 
		In order to show that $(F,\phi)$ is a representation graph for $(E,w)$
		it is enough to show that each $a \in F^0$ emits precisely one edge
		with tag $i$ for each $1 \leq i \le w(a)$ and receives precisely one edge
		with the structured edge $e$ for each $e \in r_E^{-1}(\phi^0(a))$.
		
		We will show existence of such edges first. Suppose there is a vertex $a$
		with $\phi^0(a)=v$ and an index $1 \leq i \leq w(v)$ such that
		for each $e \in s^{-1}_E(v)$ we have $a.e_i = 0$. But this contradicts the property
		(4) in Definition \ref{def3}. Indeed, in this case
		\begin{align*}
			a=a.v = a.\left( \sum_{e \in s^{-1}_E(v)} e_i e_i^*\right)
			= \sum_{e \in s^{-1}_E(v)} (a.e_i).e_i^* = 0,
		\end{align*}
		where the first equality is by the v-property. Therefore $a$ emits an edge
		with tag $i$ for each $1 \le i \le w(v)$ as required. The symmetric property,
		namely that $a$ receives an edge $f$ such that $\mot(f)=e$
		for each $e \in r_E^{-1}(a)$, is shown exactly in the same way using 
		the property (3) in Definition \ref{def3}.
		
		Next we will show that the first alternative in assumption (iv) yields the
		second one, which in fact tell us that for each relation of form
		(3) or (4) in Definition \ref{def3} with zero right-hand side already each summand
		acts like zero on each $a \in B$. Indeed, pick any such relation, for example
		$\sum_{e \in s_E^{-1} (v)}\limits e_i e_j^* = 0$, where $i \neq j$. Choose any $a \in B$.
		Then $0 = a.0 = \sum_{e \in s_E^{-1}( v)} a.e_i e_j^*$. Note that \textit{each}
		summand on the right-hand side of the last expression is either zero or a basis
		element. When $\operatorname{char}(K)=0$ a sum of basis elements is never zero,
		thus all summands must be zero. The relation (3) in Definition \ref{def3}
		is treated similarly.
		
		Finally, we will show that each $a \in F^0$ cannot emit two edges with the same
		tag or receive two edges with the same structured edge. Fix some $a \in B$
		and let $v = \phi^0(a)$. By the property (4) in Definition \ref{def3},
		\begin{equation}
			\label{eq:t.easyisrg:1}
			a = a.v = a.\left( \sum_{g \in s_E^{-1}(v)} g_i g_i^* \right)
				= \sum_{g \in s_E^{-1}(v)} a.g_i g_i^*
		\end{equation}
		The right-hand side of \eqref{eq:t.easyisrg:1} is a sum of basis elements and zeros
		and thus must contain the basis element $a$. Thus, $a = a.g_i g_i^*$ for some $g$.
		By the assumption (iv) we get 
		\[
			a.f_i = a.g_i g_i^* f_i = (a.g_i).(g_i^* f_i) = 0
		\]
		for each $f \neq g$. Therefore each vertex in $F^0$ emits precisely one edge
		with each tag.
		
		Similarly, for a fixed $e \in r_E^{-1}(v)$ by the property (iii) in Definition \ref{def3}
		we have
		\begin{equation}
			\label{eq:t.easyisrg:2}
			a = a.v = a.\left(\sum_{1 \leq i \leq w(e)} e_i^* e_i \right) = 
			\sum_{1 \leq i \leq w(e)} a.e_i^* e_i.
		\end{equation}
		As in the previous case, the right-hand side of \eqref{eq:t.easyisrg:2}
		is a sum of basis elements and zeros and thus contains the term $a$. That is
		$a = a.e_i^* e_i$ for some $1 \leq i \leq w(e)$. Then by the assumption (iv),
		\[
			a.e_j^* = a.e_i^*e_ie_j^* = (a.e_i^*).(e_ie_j^*)=0.
		\] Summing up, $a.e_i^* \neq 0$ for some fixed $i$ and $a.e_j^* = 0$ for any $j \neq i$.
		Now suppose $b.e_j = a$ for some $b \in B$. Then $(b.e_j).e_i^* = a.e_i^* \neq 0$.
		On the other hand if $j \neq i$ then by assumption (iv) one has
		$b.(e_j e_i^*) = 0$. Therefore $j$ must be equal to $i$. Thus, each vertex in $F^0$
		cannot receive two edges with the same structured edge.
		
		Summing up $(F,\phi)$ is a representation graph for $(E,w)$. It is clear from the construction
		of $(F,\phi)$ that the $L_K(E)$-module defined by $(F,\phi)$ is isomorphic to the $L_K(E)$-module $V$.
	\end{proof}
\end{theorem}

Note that the assumption on $\operatorname{char}(K)$ in Theorem \ref{t.easyisrg} may be weakened.
It is clear from the proof of the theorem that it suffices that $\operatorname{char}(K)$ exceeds the number of terms of each relation of form (3) or (4) in Definition \ref{def3}. Moreover, a slight
modification of the proof works when $\operatorname{char}(K)$
exceeds each relation of \textit{one} of the forms (iii) or (iv). Thus, assumption (iv)
of Theorem \ref{t.easyisrg} is fulfilled for any unweighted graph $E$ or, for example,
any weighted chain. 

It is clear that any representation of $L_K(E)$ on $M_0$ induced by a branching system
satisfies the assumptions (i)---(iii) of Theorem \ref{t.easyisrg} with respect to the basis of
delta-functions. Further, the proof of Theorem \ref{t.br1} shows in particular that
property (iv) is also satisfied for such representations. Further, the representation 
induced by a branching system on $M_0$ is a subrepresentation of the one induced on $M$.
Summing up, we get the following corollary.

\begin{corollary}
	Any representation of a weighted Leavitt path algebra induced by a branching system
	$X$ on the module $M$ of functions $X \rightarrow K$ contains a subrepresentation
	isomorphic to a representation given by a representation graph. This representation is 
	precisely the one induced by the same branching system on the subspace $M_0$ of functions 
	with finite support.
\end{corollary}

It is unclear to the authors, if a basis satisfying the assumptions of Theorem \ref{t.easyisrg}
can be chosen for the \textit{whole} $L_K(E)$-module $M$ defined by a branching system.

\subsection{Exceptional characteristic}

Here we provide an example that shows that the assumption (iv) in Theorem \ref{t.easyisrg} is
nonredundant. We will construct a representation that satisfies the assumptions (i)---(iii)
thereof, but not (iv) and thus is not defined by any representation graph or a branching system. 
As mentioned before, such an example does not exist for a graph of weight 1.

Let $K = \mathbb{F}_2$. Consider the weighted graph 
\[
	E = \xymatrix{ v\ar@(dr,ur)_{\mathtt{f_1,f_2}}\ar@(dl,ul)^{\mathtt{e_1, e_2}}}.
\]
Introduce the structure of a right $L_K(E)$-module on $K^1$ by specifying the
action of standard generators:
\begin{align*}
	1.e_2 = 1.f_1^* = 0, \quad
	1.v = 1.e_1 = 1.e_1^* = 1.e_2^* = 1.f_1 = 1.f_2 = 1.f_2^* = 1.
\end{align*}
The easiest way to check that this is a well defined 
representation is to check all the relations in Definition \ref{def3}. The relations 
(i) and (ii) are trivially satisfied and the relations (iii) and (iv) can be checked with
a straightforward computation:
\begin{align*}
	1.(e_1^* e_1 + e_2^* e_2) &= 1+0 = 1, &
	1.(e_1 e_1^* + f_1 f_1^*) &= 1+0 = 1, \\
	1.(f_1^* f_1 + f_2^* f_2) &= 0+1 = 1, &
	1.(e_2 e_2^* + f_2 f_2^*) &= 0+1 = 1, \\
	1.(e_1^* f_1 + e_2^* f_2) &= 1+1 = 0, &
	1.(e_1 e_2^* + f_1 f_2^*) &= 1+1 = 0, \\
	1.(f_1^* e_1 + f_2^* e_2) &= 0+0 = 0, &
	1.(e_2 e_1^* + f_2 f_1^*) &= 0+0 = 0.
\end{align*}

\appendix   

\section{Representation graphs for quivers with relations} \label{appendaa}

In this Appendix we give a general construction of representation graphs for a directed graph with relations (Definition~\ref{def0.1}) which in turn gives rise to (simple) modules for path algebras with relations. The case of weighted graphs and thus weighted Leavitt path algebras falls into this general construction. However, the trade off is that this general construction does not give the elegant and easy to use representation graphs that we obtained in the case of (weighted) graphs.

Let $E$ be a directed graph.  Consider $E$ as a category whose objects are the vertices and for two vertices $v,w$, the morphism set $\Hom_E(v,w)$ are the paths from $v$ to $w$. A covariant functor from $E$ to $\vc K$, the category of vector spaces over a field $K$, is called a \emph{representation} of $E$. By $\Rep(E)$ we denote the functor category whose objects are the representations of $E$ and morphisms are the natural transformations. A \emph{relation}  in $E$ is a formal sum, $\sum_i k_i p_i$, where $k_i\in K\backslash \{0\}$ and $p_i \in \Hom_E(v,w)$, for fixed vertices $v,w$. Let $r$ be a set of relations in $E$. Then $\Rep(E,r)$ is a full subcategory of $\Rep(E)$ which satisfies the relations in $r$, i.e., representations $\rho: E \rightarrow \vc K$ such that 
$\sum_i k_i \rho(p_i)=0$, where $\sum_i k_i p_i$ is a relation in $r$. Let $\langle r \rangle$ be the two sided ideal of the path algebra $KE$ generated by the set of relations $r$. Set  \begin{equation}\label{genehfghfyf}
A_K(E,r):= KE/\langle r \rangle.  
\end{equation} 
Then, by a result of Green~\cite[Theorem~1.1]{green}, for a finite graph $E$, there exists an exact equivalent functor 
\begin{equation*}\label{hng32}
 \Modd \, A_K(E,r) \longrightarrow \Rep(E,r),
\end{equation*}
where $\Modd \, A_K(E,r)$ is the category of right $A_K(E,r)$-modules. 

We are in a position to define the representation graphs for a given graph in this setting.

\begin{definition} Let $E$ be a directed graph. 
A {\it representation graph} for $E$ is a pair $(F,\phi)$, where $F$ is a directed graph and $\phi:F\rightarrow E$ is a homomorphism such that for any $e\in E^1$ and $u\in F^0$ there is at most one $f\in F^1$ such that $\phi(f)=e$ and $s(f)=u$.
\end{definition}

Hence a representation graph for $E$ is an immersion of the graph $E$. The following lemma is easy to check.
\begin{lemma}\label{lem0.1} Let $E$ be a graph and $\phi:F\rightarrow E$ a representation for $E$. 
If $q\neq q'\in \Path(F)$ and $\phi(q)=\phi(q')$, then $s(q)\neq s(q')$.
\end{lemma}

Let $(F,\phi)$ be a representation graph for $E$. Define a representation $\rho=\rho_{(F,\phi)}$ of $E$ by 
\begin{equation}\label{ghghgdfd234}
\rho(v):=\sum_{u\in\phi^{-1}(v)}Ku
\end{equation}
for any $v\in \Ob(E)=E^0$, and 
\[\rho(p)(u):=\begin{cases}r(q),&\text{ if }\exists q\in \Path(F): \phi(q)=p\land s(q)=u,\\0,&\text{ otherwise},\end{cases}\]
for any $p\in \Hom_E(v,w)$ and $u\in\phi^{-1}(v)$. Note that $\rho(p)$ is well-defined by Lemma \ref{lem0.1}.

Clearly if $\rho$ is a representation of $E$ defined by a representation graph $(F,\phi)$, then there are linear bases $B_v\subseteq \rho(v)~(v\in \Ob(E))$ such that \eqref{stumih} and \eqref{stumih2} below are satisfied.

\begin{align}
&B_v\cap B_w=\emptyset\text{ for any } v\neq w\in \Ob(E).\label{stumih}
\\
& \text{For any } p\in \Hom_E(v,w) \text{ and } x\in B_{v} \text{ either }\rho(p)(x)=0 \text{ or } \rho(p)(x)\in \bigcup\limits_{u\in \Ob(E)}B_{u}.
\label{stumih2}
\end{align}
It follows from Theorem \ref{thm0.1} below that the representations $\rho_{(F,\phi)}$ cover precisely the representations $\rho$ of $E$ for which there are linear bases $B_v\subseteq \rho(v)~(v\in \Ob(E))$ such that (\ref{stumih}) and (\ref{stumih2}) are satisfied.

\begin{theorem}\label{thm0.1}
Let $\rho$ be a representation of $E$ such that there are linear bases $B_v\subseteq \rho(v)~(v\in \Ob(E))$ for which (\ref{stumih}) and (\ref{stumih2}) are satisfied. Then there is a representation graph $(F,\phi)$ for $E$ such that $\rho$ is isomorphic to $\rho_{(F,\phi)}$.
\end{theorem}
\begin{proof}
Define a representation graph $(F,\phi)$ by 
\begin{align*}
F^0&=\bigcup\limits_{u\in \Ob(E)}B_{u},\\
F^1&=\{f_{e,x}\mid e\in E^1,x\in B_{s(e)},\rho(e)(x)\neq 0\},\\
s(f_{e,x})&=x,\\
r(f_{e,x})&=\rho(e)(x),\\
\phi^0(x)&=v \text{ if }x\in B_v,\\
\phi^1(f_{e,x})&=e.
\end{align*}
Set $\rho':=\rho_{(F,\phi)}$. Recall that for any $v\in \Ob(E)$ we have $\rho'(v)=\sum_{x\in\phi^{-1}(v)}Kx=\sum_{x\in B_v}Kx$. Since $B_v$ is a basis for $\rho(v)$, there is an isomorphism $\eta_v:\rho'(v)\to \rho(v)$ such that $\eta_v(x)=x$ for any $x\in B_v$. We leave it to the reader to check that for any $p\in \Hom_E(v,w)$ the diagram
\[\xymatrix{\rho'(v)\ar[d]_{\rho'(e)}\ar[r]^{\eta_v}&\rho(v)\ar[d]^{\rho(e)}\\\rho'(w)\ar[r]^{\eta_w}&\rho(w)}\]
commutes, i.e. that $\eta:\rho'\to \rho$ is an isomorphism of functors.
\end{proof}

Next we consider graphs with relations. Let $E$ be a directed graph
and $r$ a set of relations in $E$. Call a path in $E$ {\it trivial} if it has positive length (i.e. it is not a vertex) and {\it nontrivial} otherwise. We assume that a coefficient $k_i\in K$ in a relation $\sum_ik_ip_i\in r$ equals $1$ if $p_i$ is nontrivial and $-1$ if $p_i$ is trivial (note that the defining relations of weighted and unweighted Leavitt path algebras satisfy this condition). Hence a relation $r$ is either of type
\[\sum\limits_{i=1}^np_i\tag{A}\]
where $n\in \N$, $k_1,\dots,k_n\in K^{\times}$ and $p_1,\dots,p_n\in\Hom_E(v,w)$ are pairwise distinct nontrivial paths, or of type 
\[\sum\limits_{i=1}^np_i-v\tag{B}\]
where $n\in \N$, $k_1,\dots,k_n\in K^{\times}$ and $p_1,\dots,p_n\in\Hom_E(v,v)$ are pairwise distinct nontrivial paths. 

Recall that a representation of $(E,r)$ is a functor $\rho: E \rightarrow \vc K$ such that 
$\sum_i \rho(p_i)=0$ for any relation $\sum_i p_i$ of type (A) and $\sum_i \rho(p_i)=\rho(v)$ for any relation $\sum\limits_{i=1}^np_i-v$ of type (B).

\begin{definition}\label{def0.1}
A representation graph $(F,\phi)$ for $E$ is called a {\it representation graph for $(E,r)$} if the following hold. 
\begin{enumerate}
\item If a path $p$ appears in a relation of type $(A)$, then there is no $q\in \Path(F)$ such that $\phi(q)=p$.
\smallskip

\item If $\sum\limits_{i=1}^np_i-v$ is a relation of type (B), then for any $u\in\phi^{-1}(v)$ there is precisely one $1\leq j\leq n$ such that there exists a $q\in \Path(F)$ for which $s(q)=u$ and $\phi(q)=p_j$. Moreover, $r(q)=u$.
\end{enumerate}
\end{definition}

Let $(F,\phi)$ be a representation graph for $(E,r)$. Then clearly the representation $\rho=\rho_{(F,\phi)}$ of $E$ defined in (\ref{ghghgdfd234}) is a representation for $(E,r)$. Moreover, there are linear bases $B_v\subseteq \rho(v)~(v\in \Ob(E))$ such that (\ref{stumih}),(\ref{stumih2}) and (\ref{stumih3}),(\ref{stumih4}) below are satisfied.

\begin{align}
&\text{If a path } p \text{ appears in a relation of type (A), then }\rho(p)=0. \label{stumih3}\\
&\text{If } \sum\limits_{i=1}^np_i-v  \text{ is a relation of type (B), then for any } x\in B_{v} \text{ there is a }1\leq j\leq n \text{ such that } \notag\\
&\rho(p_j)(x)=x \text{ and } \rho(p_i)(x)=0\text{ for any }i\neq j. \label{stumih4}
\end{align}
It follows from Theorem \ref{thm0.2} below that the representations $\rho_{(F,\phi)}$ where $(F,\phi)$ is a representation graph for $(E,r)$ cover precisely the representations $\rho$ of $(E,r)$ for which there are linear bases $B_v\subseteq \rho(v)~(v\in \Ob(E))$ such that (\ref{stumih})-(\ref{stumih4}) are satisfied.

\begin{theorem}\label{thm0.2}
Let $\rho$ be a representation of $(E,r)$ and $B_v\subseteq \rho(v)~(v\in \Ob(E))$ linear bases such that (\ref{stumih}),(\ref{stumih2}),(\ref{stumih3}) and (\ref{stumih4}) are satisfied. Then there is a representation graph $(F,\phi)$ for $(E,r)$ such that $\rho$ is isomorphic to $\rho_{(F,\phi)}$.
\end{theorem}
\begin{proof}
Let $(F,\phi)$ be the representation graph for $E$ defined in the proof of Theorem \ref{thm0.1}. One checks easily that $(F,\phi)$ is a representation graph for $(E,r)$.
In the proof of Theorem \ref{thm0.1} it is shown that $\rho$ is isomorphic to $\rho_{(F,\phi)}$.
\end{proof}

\subsection{Irreducible representation graphs of a graph with relations}
Let $E$ be a graph with a set of relations $r$. Furthermore, let $(F,\phi)$ be a representation graph for $(E,r)$ and $\rho$ be the representation of $(E,r)$ defined by $(F,\phi)$. Let $V_F$ be the $K$-vector space with basis $F^0$. Define the algebra $A:=KE/\langle r \rangle$. Then $V_F$ becomes a right $A$ -module by defining
\[u.p=\begin{cases}r(q),&\text{ if }\exists q\in \Path(F) \text{ such that } \phi(q)=p \text{ and } s(q)=u,\\0,&\text{ otherwise},\end{cases}\] 
for any $p\in \Hom_E(v,w)$ and $u\in F^0$. 

In order to prove the main theorem of this subsection (Theorem~\ref{thm0.3}), we need the following lemma which is also used in the main text in the proof of Theorem~\ref{thmirr}. 

\begin{lemma}\label{lembasis}
Let $W$ be a $K$-vector space and $B$ a linearly independent subset of  $W$. Let $k_i\in K$ and $u_i,v_i\in B$, where $1\leq i \leq n$. Then $\sum_{s=1}^nk_s(u_s-v_s)\not\in B$.
\end{lemma}
\begin{proof}
Clearly we may assume that $u_s\neq v_s$ for any $1\leq s\leq n$. Moreover, we may assume that $n\geq 2$. Let $w_1,\dots,w_m$ be the distinct elements of the set $\{u_s,v_s\mid 1\leq s\leq n\}$. Clearly there are $l_{ij}\in K~(1\leq i<j\leq m)$ such that 
\begin{equation}\label{gdhhh1}
\sum\limits_{s=1}^nk_s(u_s-v_s)=\sum\limits_{1\leq i<j\leq m}l_{ij}(w_i-w_j).
\end{equation}
One checks easily that 
\begin{equation}\label{gdhhh2}
\sum\limits_{1\leq i<j\leq m}l_{ij}(w_i-w_j)=\sum\limits_{1\leq i\leq m}(\sum\limits_{i<j\leq m}l_{ij}-\sum\limits_{1\leq j<i}l_{ji})w_i.
\end{equation}
We prove by induction on $m$ that 
\begin{equation}\label{gdhgdthe}
\sum\limits_{1\leq i\leq m}(\sum\limits_{i<j\leq m}l_{ij}-\sum\limits_{1\leq j<i}l_{ji})=0.
\end{equation}

Case $m=2$: Clearly $\sum\limits_{1\leq i\leq 2}(\sum\limits_{i<j\leq 2}l_{ij}-\sum\limits_{1\leq j<i}l_{ji})=l_{12}-l_{12}=0$ as desired.\\

\medskip

Case $m\rightarrow m+1$: Clearly
\begin{align*}
&\sum\limits_{1\leq i\leq m+1}(\sum\limits_{i<j\leq m+1}l_{ij}-\sum\limits_{1\leq j<i}l_{ji})\\
=&\sum\limits_{1\leq i\leq m}(\sum\limits_{i<j\leq m+1}l_{ij}-\sum\limits_{1\leq j<i}l_{ji})- \sum\limits_{1\leq j<m+1}l_{j,m+1}\\
=&\sum\limits_{1\leq i\leq m}(\sum\limits_{i<j\leq m}l_{ij}-\sum\limits_{1\leq j<i}l_{ji})+\sum\limits_{1\leq i\leq m}l_{i,m+1}- \sum\limits_{1\leq j<m+1}l_{j,m+1}\\
=&\sum\limits_{1\leq i\leq m}(\sum\limits_{i<j\leq m}l_{ij}-\sum\limits_{1\leq j<i}l_{ji})=0
\end{align*}
by the induction assumption. Hence (\ref{gdhgdthe}) holds true. Now suppose that $\sum\limits_{s=1}^nk_s(u_s-v_s)\in B$. Then, in view of (\ref{gdhhh1}) and (\ref{gdhhh2}), precisely one of the coefficients \[\sum\limits_{i<j\leq m}l_{ij}-\sum\limits_{1\leq j<i}l_{ji}~(1\leq i\leq m)\] equals $1$ and the remaining coefficients are $0$. Hence the sum of the coefficients equals $1$ which contradicts (\ref{gdhgdthe}). This completes the proof. 
\end{proof}

Recall that $q\neq q'\in \Path(F)$ and $\phi(q)=\phi(q')$ implies $s(q)\neq s(q')$ by Lemma \ref{lem0.1}. We need a definition. 
\begin{definition}
Let $E$ be a graph with the set of relations $r$ and $(F,\phi)$ a representation for $(E,r)$. We say $(F,\phi)$ is \emph{ well-behaved} if $q\neq q'\in \Path(F)$ and $\phi(q)=\phi(q')$ implies $r(q)\neq r(q')$. We call $F$ {\it strongly connected} if none of the sets ${}_{v}\!\Path_w(F)~(v,w\in F^0)$ is empty.
\end{definition}

\begin{theorem}\label{thm0.3} Let $E$ be a graph with the set of relations $r$ and let $A=KE/\langle r\rangle$ be the $K$-algebra associated to  $(E,r)$. Further suppose $(F,\phi)$ is a well-behaved representation for $(E,r)$. Then the following are equivalent.
\begin{enumerate}[\upshape(i)]
\item The $A$-module $V_F$ is simple.
\smallskip
\item $F$ is strongly connected and for any $x\in V_F\setminus\{0\}$ there is an $a\in A$ and a $v\in F^0$ such that $x.a=v$.
\smallskip

\item $F$ is strongly connected and for any $x\in V_F\setminus\{0\}$ there is a $k\in K$, a $p\in \Path(E)$ and a $v\in F^0$ such that $x.kp=v$.
\smallskip

\item $F$ is strongly connected and $\phi({}_u\!\Path(F))\neq \phi({}_v\!\Path(F))$, for any $u\neq v\in F^0$.
\end{enumerate}
\end{theorem}
\begin{proof} 
(i)$\Longrightarrow$ (iv) Suppose there are $v,w\in F^0$ such that ${}_{v}\!\Path_w(F)=\emptyset$. Then clearly $w\not\in v.A$. Hence $v.A$ is a proper submodule of $V_F$ and therefore $V_F$ is not simple. Thus $F$ must be strongly connected. Now assume there are $u\neq v\in F^0$ such that $\phi({}_u\!\Path(F))= \phi({}_v\!\Path(F))$. Consider the submodule $(u-v)A\subseteq V_F$. Since $V_F$ is simple by assumption, we have $(u-v)A=V_F$. Hence there is an $a\in A$ such that $(u-v).a=v$. Clearly there is an $n\geq 1$, $k_1,\dots,k_n\in K^\times$ and pairwise distinct $p_1,\dots,p_n\in \Path(E)$ such that $a=\sum_{s=1}^nk_sp_s$. We may assume that $(u-v).p_s\neq 0$ for any $1\leq s\leq n$. Hence $p_s\in \phi({}_u\!\Path(F))= \phi({}_v\!\Path(F))$ for any $s$. Since $(F,\phi)$ is well-behaved, we have $(u-v).p_s=u_s-v_s$ for some distinct $u_s,v_s\in F^0$. Hence 
\[v=(u-v).a=(u-v).(\sum_{s=1}^nk_sp_s)=\sum_{s=1}^nk_s(u_s-v_s)\]
which contradicts Lemma~\ref{lembasis}.

\medskip 
(iv)$\Longrightarrow$ (iii). Let $x=V_F\setminus \{0\}$. Then there is an $n\geq 1$, pairwise disjoint $v_1,\dots, v_n\in F^0$ and $k_1,\dots,k_n\in K^\times$ such that $x=\sum_{s=1}^nk_sv_s$. If $n=1$, then $x.k_1^{-1}\phi(v_1)= v_1$. Suppose now that $n>1$. Since by assumption (iv) holds, we may assume that there is a $p_1\in \phi({}_{v_1}\!\Path(F))$ such that $p_1\not\in\phi({}_{v_2}\!\Path(F))$. Since $(F,\phi)$ is well-behaved, $x.p_1\neq 0$ is a linear combination of at most $n-1$ vertices from $F^0$. Proceeding like that we obtain paths $p_1,\dots,p_m$ such that $x.p_1\dots p_m=kv$ for some $k\in K^\times$ and $v\in F^0$. Hence $x.p_1\dots p_mk^{-1}=v$.

\medskip 

(iii)$\Longrightarrow$ (ii). Trivial.

\medskip 

(ii)$\Longrightarrow$ (i). Let $U\subseteq V_F$ be a nonzero $A$-submodule and $x\in U\setminus\{0\}$. Since by assumption (ii) is satisfied, there is an $a\in A$ and a $v\in F^0$ such that $v=x.a\in U$. Let now $v'$ be an arbitrary vertex in $F^0$. Since by assumption $F$ is strongly connected, there is a $p\in {}_{v}\!\Path_{v'}(F)$. Hence $v'=v.\phi(p)\in U$. Hence $U$ contains $F^0$ and thus $U=V_F$.
\end{proof}

\section*{Acknowledgments}
We thank John Meakin who brought Stallings' paper~\cite{stallings} to our attention. The initial discussion on this work started in 2015 in the University of Bielefeld, where Hazrat was a Humboldt Fellow and Shchegolev was a PhD student supported by the Deutscher Akademischer Austauschdienst (DAAD). The subsequent work was done while Preusser was a PostDoc fellow at the St. Petersburg State University supported by the Russian Science Foundation grant 19-71-30002.

\end{document}